\documentclass{article}

\usepackage[T1]{fontenc} 
\usepackage{mathrsfs}
\usepackage{enumerate}
\usepackage{indentfirst}
\usepackage{ulem}
\usepackage[dvips]{graphicx}
\usepackage{titlesec,float}
\usepackage{amsmath}
\usepackage{xcolor}
\usepackage{listings}
\usepackage{appendix}
\usepackage{titlesec,titletoc}
\usepackage{times}
\usepackage{mathtools}
\usepackage[top=30mm,bottom=30mm,left=30mm,right=30mm]{geometry}
\usepackage{amsfonts}
\usepackage{amssymb}
\usepackage{setspace}
\usepackage{accents}
\usepackage{statex}
\usepackage{amssymb}
\usepackage{tikz}

\numberwithin{equation}{section}
\makeatletter
\def\wideubar{\underaccent{{\cc@style\underline{\mskip10mu}}}}
\makeatother

\makeatletter
\def\widebar{\accentset{{\cc@style\underline{\mskip10mu}}}}
\makeatother

\begin{document}

\newtheorem{theorem}{Theorem}[section]
\newtheorem{assume}[theorem]{Assumption}
\newtheorem{prop}[theorem]{Proposition}
\newtheorem{corollary}[theorem]{Corollary}
\newtheorem{lemma}[theorem]{Lemma}
\newtheorem{definition}[theorem]{Definition}
\newtheorem{remark}[theorem]{Remark}
\newtheorem{claim}[theorem]{Claim}

\newenvironment{proof}{{\noindent \it{Proof}}\ }{\hfill $\square$\par}
\newenvironment{pf}{{\noindent\it Proof of the claim.}\ }{\hfill $\square$\par}

\title{Spreading properties of a three-component reaction-diffusion model for the population of farmers and hunter-gatherers} 

\author{Dongyuan Xiao, Ryunosuke Mori}
\date{February 26 2019}
\maketitle

\begin{abstract}
In this paper, we investigate the spreading properties of solutions of farmer and hunter-gatherer model which is a three-component reaction-diffusion system. Ecologically, the model describes the geographical spreading of an initially localized population of farmers into a region occupied by hunter-gatherers. This model was proposed by Aoki, Shida and Shigesada in 1996. By numerical simulations and some formal linearization arguments, they concluded that there are four different types of spreading behaviors depending on the parameter values. Despite such intriguing observations, no mathematically rigorous studies have been made to justify their claims. The main difficulty comes from the fact that the comparison principle does not hold for the entire system. In this paper, we give theoretical justification to all of the four types of spreading behaviors observed by Aoki {\it et al.}. Furthermore, we show that a logarithmic phase drift of the front position occurs as in the scalar KPP equation. We also investigate the case where the motility of the hunter-gatherers is larger than that of the farmers, which is not discussed in the paper of Aoki {\it et al.}.

\vspace{10pt}
\hspace{-14pt}\textbf{Key words:\:Farmer and hunter-gatherer model;\:long time behavior;\:spreading speed;\: logarithmic correction}
\end{abstract}

\section{Introduction} 

Early in the stone age, our ancestors lived as hunter-gatherers, which means, instead of growing their food, they lived on hunting, fishing and gathering berries and eggs of birds in the forest.
As humans evolved, agriculture gradually appeared independently in different parts of the globe.
At least 11 separate regions of the Old and New World were identified as independent centers of the origin of agriculture. Tracing the origins of early farming and its spreading has been the major subject of interest for a long time. For instance, the study of the origins of farming in the Near East and its dispersal to Europe has been done by many archaeologists, anthropologists, linguists, and geneticists. 
Many archaeological evidences certified that agriculture emerged about 11,000 years ago in the Near East before reaching North Europe about 5,000 years later. Furthermore, genetic studies tended to support that farmers in the North Europe have noticeable genetic affinity to Near East populations. It suggests that agriculture did not only spread solely across Europe as a cultural process, but also in concert with a migration of people. This fact motivated ecologists to model this kind of geographical spreading of an initially localized farmers into a region occupied by hunters and gatherers as a reaction-diffusion process in an infinite habitat.

In 1996, Aoki, Shida and Shigesada \cite{shigesada model} proposed the following three-component reaction-diffusion system to study the process of Neolithic transition from hunter-gatherers to farmers (actually, in \cite{shigesada model}, they only considered the case $N=1$, but in the present paper we formulate the problem in a general space dimension $N\ge 1$): 
\begin{equation*}
\left\{
\begin{aligned}
&\partial_tF=D\Delta F+r_fF(1-(F+C)/K),\\
&\partial_tC=D\Delta C+r_cC(1-(F+C)/K)+e(F+C)H,\ \ \ \ \ \ \mbox{in}\ \mathbb{R}^N,\\
&\partial_tH=D_h\Delta H+r_hH(1-H/L)-e(F+C)H.
\end{aligned}
\right.
\end{equation*}
The population densities of initial farmers, converted farmers and hunter-gatherers are represented by $F$, $C$ and $H$, respectively.
This model contains seven positive parameters: $D_h$ is the diffusion coefficient of hunter-gatherers, which is assumed to be greater than or equal to the diffusion coefficient $D>0$ of farmers; $r_f$, $r_c$ and $r_h$ are the intrinsic growth rates; $K$ and $L$ are the carrying capacities of farmers and hunter-gatherers; $e$ is the conversion rate of hunter-gatherers to farmers. Note that, in \cite{shigesada model}, $D_h$ is always assumed to be equal to $D$. However, we do not see any reason to assume that the hunter-gatherers diffuse at the same speed as the farmers. It may be more natural to imagine that the hunter-gatherers would diffuse faster than the hunters. For this reason, in the present paper, we assume $D_h\ge D$ rather than $D_h=D$. As we will see later, the case $D_h>D$ is much harder to analyze than the case $D_h=D$.

As shown in \cite{shigesada model}, by a suitable change of variables, the above system is converted to:
\begin{equation}\label{fch-equation}
\left\{
\begin{aligned}
&\partial_tF=\Delta F+aF(1-F-C),\\
&\partial_tC=\Delta C+C(1-F-C)+sH(F+C),\ \ \ \ \ \ \mbox{in}\ \mathbb{R}^N,\\
&\partial_tH=d\Delta H+bH(1-H-g(F+C)),
\end{aligned}
\right.
\end{equation}
where $d=D_h/D\ge1$, $a=r_f/r_c$, $b=r_h/r_c$, $s=eL/r_c$ and $g=eK/r_h$.
We consider the following initial condition
\begin{equation}\label{intial data}
H(0,x)\equiv 1,\ \ C(0,x)\equiv 0,\ \ F(0,x)=F_0(x)\ge0\ \ (F_0\not\equiv0),
\end{equation}
where $F_0(x)$ is a compactly supported continuous function. The reason why we consider such initial data is because our goal is to understand how agriculture spread over a region that was originally occupied by hunter-gatherers. Thus our problem is to analyze the "spreading fronts" of the farmer populations that start from localized initial data. 

In the special case $a=1$, by setting $G=F+C$, the total population density of all farmers, the system (\ref{fch-equation}) reduces to the two-component system of predator-prey type:
\begin{equation}\label{pp system}
\left\{
\begin{aligned}
&\partial_tG=\Delta G+G(1+sH-G),\\
&\partial_tH=d\Delta H+bH(1-H-gG).
\end{aligned}
\right.
\end{equation}
The long time behavior of the system \eqref{pp system} (with $a=1$) has been studied by Hilhost, Mimura and Weidenfel in \cite{wp}. Among other things they proved that the solution converges to a spatially constant steady state as $t\rightarrow\infty$. They also considered the case where the coefficients $s$, $g$ are replaced by $s/\varepsilon$, $g/\varepsilon$ and studied the limit of the solution as $\varepsilon\rightarrow0$. However, how fast the front propagates to infinity was not discussed in \cite{wp}. The first result on the spreading speeds for predator-prey systems was obtained by Ducrot, Giletti and Matano in \cite{pp}, where they treated two-species predator-prey system that are different from \eqref{pp system}.

In the general case $a\neq 1$, the ODE system corresponding to (\ref{fch-equation}) possesses the following four different types of steady states: 
\[
(F,C,H)=
\left\{
\begin{aligned}
&(0,0,0),\ (0,0,1),\\
&(\widehat{F},\widehat{C},0),\ \ \mbox{where}\ \ \widehat{F}+\widehat{C}=1,\ \widehat{F},\widehat{C}\ge0,\\
&(0, C^*, H^*),\ \ \mbox{where}\ \  C^*=(1+s)/(1+sg),\ H^*=(1-g)/(1+sg)).
\end{aligned}
\right.
\] 
The first two steady states always exist and unstable. The third one is a line of neutral equilibria, which always exists and stable if $g\ge 1$. The fourth one  exists and is stable if and only if $g<1$. It implies that there exist different expanding patterns which is determined by the value of $g$. Let us remark that $g=eK/r_h$, which means the value of $g$ highly depends on the value of the conversion rate $e$. Note that, throughout this paper, we call the case $g\ge 1$ by high conversion rate case and the case $g<1$ by low conversion rate case.  Ecologically speaking, if the conversion rate is sufficiently high, hunter-gatherers will completely convert into farmers,  whereas hunter-gatherers and farmers can coexist if the conversion rate is low enough. 

\subsection{Observations by Aoki {\it et al.}}

In the above mentioned paper \cite{shigesada model}, the authors observed four different types of spreading behaviors depending on the parameter values, namely those of $a$, $s$ and $g$ in the system (\ref{fch-equation}). Their observation of the spreading fronts was done by numerical simulations. Strictly speaking, they did not consider truly localized initial data for the farmers, but they set $C_0\equiv 0$ and chose $F_0$ to be a Heaviside function: $F_0=1\ (x\le 0)$, $F_0(x)=0\ (x>0)$. In order to estimate the speed of the spreading fronts, they again relied mainly on numerical simulations, but also calculated by formal analysis of the minimal speed of the traveling wave (comprising of the advancing front of the farmers and the retreating front of the hunter-gatherers) and confirmed that the numerically observed spreading speed well agreed with the formally calculated minimal traveling wave speed. The following figures illustrate the shape of what they call the transient waveforms.
\vspace{10pt}
\begin{center}
\begin{tikzpicture}[scale = 1.6]
\draw[<->](0,1.65)-- (0,0) -- (4.7,0) node[right] {$x$};
\draw[dotted, thick] (0,1) to [out=0,in=180] (1,0.8)--(3,0.8) to [out=0, in=180] (3.5,0)--(4.5,0.025);
\draw[very thick] (0,0.02) to [out=0,in=175] (1,0.3)--(3.1,0.3) to [out=5,in=175] (3.4,0.025)--(4.5,0.025);
\draw[thin] (0,0.02)--(3,0.02)to[out=0,in=180](3.5,1)--(4.5,1);
\draw[dashed] (3.25,1.5)--(3.25,-0.2);
\node[left] at (0,1) {\footnotesize{1}};
\node[left] at (0,0) {\footnotesize{0}};
\node[above] at (0.25,1) {\small{$F$}};
\node[above] at (2,0.4) {\small{$C$}};
\node[above] at (4,1) {\small{$H$}};
\node[right] at (3.25,0.5) {\textbf{$\rightarrow$}\ $c^*$};
\node[below] at (1.75,0) {\footnotesize{Final zone (for $F+C$)}};
\node[below] at (4.25,0) {\footnotesize{Leading edge}};
\node[below] at (2.5,-0.4) {\textbf{Figure 1:} High conversion rate case, $a> 1+s$.};
\end{tikzpicture}
\vspace{10pt}

\begin{tikzpicture}[scale = 1.6]
\draw[<->](0,1.65)-- (0,0) -- (4.7,0) node[right] {$x$};
\draw[dotted, thick] (0,1) to [out=0,in=180] (1,0.03)--(4.5,0.03);
\draw[very thick] (0,0.02) to [out=0,in=180] (1,1)--(3,1) to [out=0,in=180] (3.5,0.025)--(4.5,0.025);
\draw[thin] (0,0.02)--(3,0.02)to[out=0,in=180](3.5,1)--(4.5,1);
\draw[dashed] (3.25,1.5)--(3.25,-0.2);
\node[left] at (0,1) {\footnotesize{1}};
\node[left] at (0,0) {\footnotesize{0}};
\node[above] at (0.25,1) {\small{$F$}};
\node[above] at (2,1) {\small{$C$}};
\node[above] at (4,1) {\small{$H$}};
\node[right] at (3.25,0.5) {\textbf{$\rightarrow$}\ $c^*$};
\node[below] at (1.75,0) {\footnotesize{Final zone (for $F+C$)}};
\node[below] at (4.25,0) {\footnotesize{Leading edge}};
\node[below] at (2.5,-0.4) {\textbf{Figure 2:} High conversion rate case, $a<1+s$.};
\end{tikzpicture}
\vspace{10pt}

\begin{tikzpicture}[scale = 1.6]
\draw[<->](0,1.65)-- (0,0) -- (4.7,0) node[right] {$x$};
\draw[dotted, thick] (0,0.03) -- (2.8,0.03) to[out=0,in=210](3.25,0.5) to[out=350,in=170](3.6,0.03)--(4.5,0.03);
\draw[very thick] (0,1.3) --(2.7,1.3)to [out=0,in=180] (3.5,0.025)--(4.5,0.025);
\draw[thin] (0,0.8)--(3,0.8)to[out=0,in=180](3.5,1)--(4.5,1);
\draw[dashed] (3.25,1.5)--(3.25,-0.2);
\node[left] at (0,1) {\footnotesize{1}};
\node[left] at (0,0) {\footnotesize{0}};
\node[left] at (3.1,0.5) {\small{$F$}};
\node[above] at (1.5,1.3) {\small{$C$}};
\node[above] at (4,1) {\small{$H$}};
\node[right] at (3.25,0.6) {\textbf{$\rightarrow$}\ $c^*$};
\node[below] at (1.75,0) {\footnotesize{Final zone}};
\node[below] at (4.25,0) {\footnotesize{Leading edge}};
\node[below] at (2.5,-0.4) {\textbf{Figure 3:} Low conversion rate case, $a> 1+s$.};
\end{tikzpicture}
\vspace{10pt}

\begin{tikzpicture}[scale = 1.6]
\draw[<->](0,1.65)-- (0,0) -- (4.7,0) node[right] {$x$};
\draw[dotted, thick] (0,0.02)--(4.5,0.02);
\draw[very thick] (0,1.3) --(2.7,1.3)to [out=0,in=180] (3.5,0.025)--(4.5,0.025);
\draw[thin] (0,0.8)--(3,0.8)to[out=0,in=180](3.5,1)--(4.5,1);
\draw[dashed] (3.25,1.5)--(3.25,-0.2);
\node[left] at (0,1) {\footnotesize{1}};
\node[left] at (0,0) {\footnotesize{0}};
\node[left] at (1.6,0.3) {\small{$F$}};
\node[above] at (1.5,1.3) {\small{$C$}};
\node[above] at (4,1) {\small{$H$}};
\node[right] at (3.25,0.6) {\textbf{$\rightarrow$}\ $c^*$};
\node[below] at (1.75,0) {\footnotesize{Final zone}};
\node[below] at (4.25,0) {\footnotesize{Leading edge}};
\node[below] at (2.5,-0.4) {\textbf{Figure 4:} Low conversion rate case, $a<1+s$.};
\end{tikzpicture}
\end{center}

In order to make our motivation clearer, we give a brief explanation of what was observed by Aoki {\it et al.} \cite{shigesada model} before stating our main results. As mentioned above, their observation was done by numerical simulation combined with a formal analysis of the minimal speed of traveling wave.
\begin{itemize}
\item The spreading speed of the solution of the system \eqref{fch-equation} is always determined by $\max\{2\sqrt{a},2\sqrt{1+s}\}$.
\item The behaviors of solutions on the leading edge where the hunter-gatherers have little contact with the farmers are almost the same between the high conversion rate case and the low conversion rate case (see Figure 1-4, Leading edge).
\item In the case where $a<1+s$, a wave of advance of initial farmers $F$ is not generated (see Figure 2 and Figure 4). If, in addition that $g<1$, the initial farmers $F$ disappear entirely (see Figure 4). Whereas, if $g\ge 1$, behind the wavefront, farmers have almost reached carrying capacity and hunter-gatherers have just about disappeared (see Figure 2, Final zone).
\item In the case where $a>1+s$, an advancing wavefront of initial farmers $F$ is also generated (see Figure 1 and Figure 3). If, in addition that $g<1$, the waveform is a small peak with leading edge and trailing edge that converge to $0$ (see Figure 3).
\end{itemize}

The goal of the present paper is to give rigorous justification to all of the above observations, and also to discuss the case where $d>1$ that has not been treated in \cite{shigesada model}. The case $d>1$, turns out to be much harder to analyze. Further, we also prove logarithmic drift of the fronts for some cases.

\subsection{Outline of the paper}

We will present our main results in Section 2. Some results are stated for the special case $d=1$ (Theorems \ref{thm final zone 1+s>a} and \ref{thm profile g>1 1+s>a} and part of Theorems \ref{thm profile g<1} and \ref{thm of log delay}), which is the case that was treated in \cite{shigesada model}. Other results are stated for $d\ge 1$.The proof of the  main results will be carried out in the subsequent sections.

 In Section 3, we prove Theorem \ref{thm leading edge} which is concerned with the behaviors at the leading edge that appears in all of Figure 1-4. Since there is little interaction between the farmers and the hunter-gatherers on the leading edge, the analysis of this zone is rather straightforward.

In Section 4, we prove Theorem \ref{thm of log delay} on the logarithmic phase drift of the front. The logarithmic phase drift for the scalar KPP equation was studied in detail by Bramson \cite{log delay 1, log delay 2} by using a probabilistic method. More recently, a much simpler PDE proof has been proposed by Hamel, Nolen, Requejoffre and Ryzhik \cite{log delay 3}. Their method is based on a super and sub-solution argument. Our proof of Theorem \ref{thm of log delay} is also based on a super and sub-solution argument, but the sub-solution is quite different from that in \cite{log delay 3} since their sub-solution does not work for systems of equations. 

In Section 5, we prove Theorems \ref{thm final zone 1+s>a} and \ref{thm final zone 1+s<a} that are concerned with the uniform positivity of solution in the final zone. Theorem \ref{thm final zone 1+s>a} is concerned with the case $1+s\ge a$ (which corresponds to Figure 2 and 4 above), while Theorem \ref{thm final zone 1+s<a} is concerned with the case $1+s<a$ (which corresponds to Figure 1 and 3 above). Intriguingly, the proof of the two theorem are quite different: the proof of Theorem \ref{thm final zone 1+s>a} uses a limit argument that was employed in \cite{pp} to show "pointwise spreading". The proof of Theorem \ref{thm final zone 1+s<a} uses the result on the logarithmic drift stated in Theorem \ref{thm of log delay}.

In Section 6, we will study final asymptotic profiles of solutions in the final zone and complete the proof of Theorems \ref{thm profile g>1 1+s>a},  \ref{thm profile g>1 1+s<a} and \ref{thm profile g<1}, thus comfirming the profiles shown in Figure 1-4. The proof is based on the conclusion of the results in Theorems \ref{thm final zone 1+s>a} and \ref{thm final zone 1+s<a}, and a certain limit argument.

Finally in the appendix, we give the proof of Proposition \ref{Dirichlet eq} which plays an important role in the proof of Theorem \ref{thm of log delay} in Section 4

\section{Main results}
Front propagation for scalar reaction-diffusion equations has been studied extensively and there is vast literature on this theme. Early in 1937, Fisher \cite{F} and Kolmogorov, Petrovsky and Piskunov \cite{KPP} introduced a scalar reaction-diffusion equation with monostable nonlinearity as a model equation in population genetics, which studies the propagation of dominant gene in homogeneous environment. In particular, \cite{KPP} made an important early analysis of the structure of the set of traveling waves for a special class of monostable reaction-diffusion equation, the so-called KPP equation. Application of reaction-diffusion equations to ecology was pioneered by Skellam \cite{S} in 1951. As regards the propagation of fronts for solutions with compactly supported initial data (that is, the so-called ``spreading front'') is concerned, the first rigorous mathematical results in multi-dimensional homogeneous environment were provided by Aronson and Weinberger \cite{AW}, for the case of monostable nonlinearity and also for bistable nonlinearity. The paper \cite{AW} introduced the notion of what we now call ``spreading speed'', the meaning of which is specified Remark \ref{remark spreading speed} below. 
However, apart from some recent works such as \cite{pp}, much less is known about the spreading properties for systems of equations for which the comparison principle does not hold.

Note that the nonlinear terms of the first two equations in the system (\ref{fch-equation}) are similar to the well-known KPP nonlinearity if we neglect the term $H$. Hence,
we first recall the classical spreading
result for the monostable scalar equation from Aronson and Weinberger \cite{AW}:
\begin{prop}[\textbf{Spreading for the scalar KPP equation} (\cite{AW})]\label{prop of kpp equation}
Consider the so-called KPP equation

\begin{equation}\label{kpp equation}
\partial_tu(t,x)=d\Delta u(t,x)+uf(u(t,x)),\ t>0,\ x\in\mathbb{R}^N,
\end{equation}
wherein $f\in C^1(\mathbb{R})$ satisfies
$$f(1)=0,\ f(u)>0\ \ \mbox{and}\ \ f(u)\le f(0)\ \ \mbox{for all}\ \ u\in [0,1).$$
Define $c^*(d,f)=2\sqrt{df(0)}$. Then, for any nontrivial compactly supported and continuous function $u_0(x)$, the solution $u\equiv u(t,x;u_0(x))$ of (\ref{kpp equation}) with the initial data $u_0(x)$ satisfies:
$$\underset{t\to+\infty}{\lim}\underset{\lVert x\rVert \le ct}{\sup}|1-u(t,x)|=0,\ 0<c<c^*(d,f),$$
and
$$\underset{t\to+\infty}{\lim}\underset{\lVert x\rVert\ge ct}{\sup}u(t,x)=0,\ c>c^*(d,f).$$
Moreover, the quantity $c^*(d,f)$ coincides with the minimal speed of traveling wave solutions of the equation (\ref{kpp equation}) connecting $0$ to $1$.
\end{prop}

\begin{remark}\label{remark spreading speed}
Usually the term ``spreading front'' regards to a front of a solution that propagates to infinity from localized initial data, typically those that are compactly supported. The term ``spreading speed'' refers to a quantity $c^*$ for which the last two estimates of Proposition \ref{prop of kpp equation} hold. Thus the distance between the sphere of radius $c^*t$ and the actual position of the front is of order $o(t)$. 
\end{remark}

\subsection{Uniform spreading properties}
In this subsection, we present our main results on the spreading properties of solutions of the system (\ref{fch-equation}). 
Throughout this paper, we define the two quantities $c^*$ and $c^{**}$ as $c^*:=\max\{2\sqrt{a},2\sqrt{1+s}\}$ and $c^{**}:=\min\{2\sqrt{a},2\sqrt{1+s}\}$.
Our first result is about the analysis of the leading edge, which provides an upper estimate on the spreading speed. As we observed from Figure $1$-$4$, for all four cases, the behaviors of solutions are almost same on the leading edge. 
The first three theorems hold regardless of the size of the conversion rate. 

\begin{theorem}\label{thm leading edge}
For any given $c>c^*$, the solution $(F,C,H)$ of the system \textnormal{(\ref{fch-equation})} with the initial data \textnormal{(\ref{intial data})} satisfies:
 \begin{equation}\label{upper estimate  c>c*}
\lim_{ t \to \infty}\sup_{ \lVert x\rVert\ge ct}\Big(\|1-H(t,x)\|+F(t,x)+C(t,x)\Big)=0.
\end{equation}
\end{theorem}

The most difficult part of the analysis is the behaviors of the solutions in the final zone, where original hunter-gatherers, initial farmers and converted farmers heavily interact with each other. A large part of the present thesis is devoted to the analysis in this final zone.
Our third result deals with the propagation of farmers in the final zone, which provides an lower estimate on the spreading speed. As we observed from Figure $1$-$4$, for all four cases, the spreading speed is determined by the larger value of $2\sqrt{1+s}$ and $2\sqrt{a}$. 

\begin{theorem}\label{thm final zone 1+s>a}
If $d=1$ and $1+s\ge a$, then for any given $0\le c<c^*$, there exists $\varepsilon>0$ such that the solution $(F,C,H)$ of the system \textnormal{(\ref{fch-equation})} with the initial data \textnormal{(\ref{intial data})} satisfies:
\begin{equation}\label{esimate of F+C 1+s>a c<c*}
\underset{t\to\infty}{\liminf}\underset{\lVert x\rVert\le ct}{\inf}(F+C)(t,x)\ge\varepsilon,
\end{equation}
\begin{equation}\label{estimate of H 1+s>a c<c*}
\underset{t\to\infty}{\limsup}\underset{\lVert x\rVert\le ct}{\sup}H(t,x)\le 1-\varepsilon.
\end{equation}
\end{theorem}

\begin{theorem}\label{thm final zone 1+s<a}
If $1+s<a$, then for any given $0\le c<c^*$, there exists $\varepsilon>0$ such that  the solution $(F,C,H)$ of the system \textnormal{(\ref{fch-equation})} with the initial data \textnormal{(\ref{intial data})} satisfies:
\begin{equation}\label{esimate of F+C 1+s<a c<c*}
\underset{t\to\infty}{\liminf}\underset{\lVert x\rVert\le ct}{\inf}(F+C)(t,x)\ge \varepsilon,
\end{equation}
\begin{equation}\label{estimate of H 1+s<a c<c*}
\underset{t\to\infty}{\limsup}\underset{\lVert x\rVert\le ct}{\sup}H(t,x)\le 1-\varepsilon.
\end{equation}
\end{theorem}

\begin{remark}
In present paper, we use two different approaches to prove the results of Theorem \ref{thm final zone 1+s>a} and Theorem \ref{thm final zone 1+s<a}, respectively. We specially note that the result of Theorem \ref{thm final zone 1+s>a} could be extended to the general case $d\ge 1$ by applying a similar argument to that in section 5, and the details will be provided in our future work.
\end{remark}

\begin{remark}\label{remark of hair trigger}
The results of Theorem \ref{thm leading edge},  Theorem \ref{thm final zone 1+s>a} and Theorem \ref{thm final zone 1+s<a} can be regarded as an analogue of the well-known "hair-trigger effect" for scalar monostable equation \cite{AW}. Moreover, one may find that the spreading speed of the system is always equal to $c^*$, which is only determined by the parameters in the $F$-equation and $C$-equation. This is due to the initial data $H_0(x)\equiv 1$, which means hunter-gatherers have already spread to the whole region.  
\end{remark}

At last, we show our main results about the asymptotic profiles of solutions in the final zone. The precise results will be presented in two cases, the high conversion rate case ($g\ge 1$) and low conversion rate case ($g<1$), respectively. For the high conversion rate case, our results read as:
\begin{theorem}[High conversion rate case, $1+s\geq a$]\label{thm profile g>1 1+s>a}
If $d=1$, $1+s\ge a$ and $g\ge 1$, then for any given $0\le c< c^*$, the solution $(F,C,H)$ of the system \textnormal{(\ref{fch-equation})} with the initial data \textnormal{(\ref{intial data})} satisfies:
\begin{equation}\label{profile of H g>1 1+s>a}
\limsup_{ t \to \infty}\sup_{ \lVert x\rVert\le ct}H(t,x)=0,
\end{equation}
\begin{equation}\label{profile of F+C g>1 1+s>a}
\limsup_{ t \to \infty}\sup_{ \lVert x\rVert\le ct}|1-(F+C)(t,x)|=0.
\end{equation}
Moreover, for any given $c^{**}< c_1<c_2< c^*$, it holds:
\begin{equation}\label{profile of F and C g>1 1+s>a}
\limsup_{ t \to \infty}\sup_{ c_1t\le\lVert x\rVert\le c_2t}\Big(F(t,x)+|1-C(t,x)|\Big)=0.
\end{equation}
\end{theorem}
\begin{theorem}[High conversion rate case, $1+s<a$]\label{thm profile g>1 1+s<a}
If  $1+s< a$ and $g\ge 1$, then for any given $0\le c< c^*$, the solution $(F,C,H)$ of the system \textnormal{(\ref{fch-equation})} with the initial data \textnormal{(\ref{intial data})} satisfies:
\begin{equation}\label{profile of H g>1 1+s<a}
\limsup_{ t \to \infty}\sup_{ \lVert x\rVert\le ct}H(t,x)=0,
\end{equation}
\begin{equation}\label{profile of F+C g>1 1+s<a}
\limsup_{ t \to \infty}\sup_{ \lVert x\rVert\le ct}|1-(F+C)(t,x)|=0.
\end{equation}
\end{theorem}

For the high conversion rate case, Theorem \ref{thm profile g>1 1+s>a} and Theorem \ref{thm profile g>1 1+s<a} imply that all of the original hunter-gatherers convert to farmers at last. However, the explicit profiles of the $F$-component and $C$-component in the final zone are yet to be investigated. Nevertheless, for the low conversion rate case, we will show that the $F$-component always converges to $0$ in the final zone. In order to do this, one need first prove that the population density of hunter-gatherers would stay positive uniformly. From some ecological observations, this may happen if the conversion rate is small enough or the product of the intrinsic growth rate and the diffusion speed of hunter-gatherers is large enough. Then,
we can investigate how the $C$-component and $H$-component behave in the finial zone by considering the dynamics of the underlying ODE system:
\begin{equation}\label{ode system}
\left\{
\begin{aligned}
&C_t=C(1-C)+sCH,\\
&H_t=bH(1-H-gC).
\end{aligned}
\right.
\end{equation}
We expect the solution of the PDE system (\ref{fch-equation}) to converge uniformly to the equilibrium $(C^*,H^*)$ in the final zone as $t\to+\infty$. We prove the conjecture by the established fact that a strict Lyapunov function exists. Note that, the ODE system (\ref{ode system}) is a well-known Lotka-Volterra system with logistic growth rate. In population dynamics, especially prey-predator systems, ODE models always admit a strict Lyapunov function. The concrete example will be given in the later section. 
\begin{theorem}[Low conversion rate case]\label{thm profile g<1}
If $g<1$, then for any given $0\le c<c^*$, for the solution $(F,C,H)$ of the system \textnormal{(\ref{fch-equation})} with the initial data \textnormal{(\ref{intial data})}, it holds:
\begin{itemize}
\item[(1)] there exists $\varepsilon>0$, such that: 
\begin{equation}\label{estimate of the profile H g<1}
\liminf_{ t \to \infty}\inf_{\lVert x\rVert\le ct}H(t,x)\ge \varepsilon,
\end{equation}
\begin{equation}\label{estimate of the profile C g<1}
\liminf_{ t \to \infty}\inf_{\lVert x\rVert\le ct}C(t,x)\ge 1+\varepsilon,
\end{equation}
and
\begin{equation}\label{estimate of the profile F g<1}
\lim_{ t \to \infty}\sup_{\lVert x\rVert\le ct}F(t,x)=0,
\end{equation}
provided that
$$g<\frac{\min\{1,a\}}{\min\{1,a\}+s}\ \ \mbox{or}\ \ bd\ge \frac{c^*}{1-g}.$$
\item[(2)] if $d=1$, one has:
\begin{equation}\label{estimate of the profile}
\lim_{ t \to \infty}\sup_{\lVert x\rVert\le ct}\Big(F(t,x)+|C^*-C(t,x)|+|H^*-H(t,x)|\Big)=0,
\end{equation}
provided that
$$g<\frac{\min\{1,a\}}{\min\{1,a\}+s}\ \ \mbox{or}\ \ b\ge \frac{c^*}{1-g}.$$
\end{itemize}
\end{theorem}
\begin{remark}
Note that, if we assume $b\le 1$ as in \cite{shigesada model} for the ecological motivation that the intrinsic growth rate of hunter-gatherers is supposed to be smaller than or equal to that of converted farmers, then for the first statement in Theorem \ref{thm profile g<1}, the condition $bd\ge c^*/(1-g)$ could hold for large enough $d$. 
\end{remark}

\subsection{Logarithmic correction}

Recall that, in Figure 3 and Figure 4, the behaviors of the $F$-component are sightly different on the wavefront. In Figure 4, it is described by Corollary \ref{col} that the distribution of initial farmers converges to $0$ everywhere. However, in Figure 3, the distribution of initial farmers is peaked with leading and trailing edges converging to $0$. The results stated in Theorem \ref{thm final zone 1+s>a} and Theorem \ref{thm final zone 1+s<a} can not explain why a small peak may occur on the wavefront. More precisely, the spreading speed being equal to $c^*$ does not mean that the front propagates parallel to the traveling wave of speed $c^*$. Even for the scalar KPP equation, this does not hold true. Therefore, we are motivated to consider the behaviors of solutions in the area enough close to $c^*t$.

A famous result of Bramson in \cite{log delay 1,log delay 2}
showed that for the scalar KPP equation, there is a backward phase drift of order $O(\log t)$ from the position $c^*(f)t$. More precisely, Bramson gave a sharp asymptotics of the location of the level sets of the solution $u(t,x)$ of the scalar KPP equation by using some probabilistic arguments. Let $E_m(t)$ be the set of points in $(0,+\infty)$ where $u(t,\cdot)=m$ and $m\in(0,1)$. Then, there exists a constant $B$ and a shift $x_m$ depending on $m$ and the initial data $u_0$ such that
$$E_m(t)\subset[c^*(f)t-\frac{3}{2\lambda^*}\ln t-x_m-\frac{B}{t}, c^*(f)t-\frac{3}{2\lambda^*}\ln t-x_m+\frac{B}{t}]\ \ \mbox{for}\ \ t\ \ \mbox{large enough},$$
with $\lambda^*=c^*(f)/2$.
Recently, this result  has been explained in simple PDE terms by Hamel et al. \cite{log delay 3}. They showed that, for every $m\in(0,1)$ there exists $B>0$ such that
$$E_m(t)\subset[c^*(f)t-\frac{3}{2\lambda^*}\ln t-B, c^*(f)t-\frac{3}{2\lambda^*}\ln t+B]\ \ \mbox{for}\ \ t\ \ \mbox{large enough}.$$
Moreover, Ducrot \cite{AD} extended this proposition to the multi-dimensional case by showing that there is a logarithmic backward phase drift. 

For the system (\ref{fch-equation}), the estimate of the upper bound for the level set is rather straightforward. However, in the final zone, since three components heavily interact with each other, it is hard to estimate the level set of the $F$-component or $C$-component separately. Hence, we just provide a slightly weak result in the present paper. The estimate of the low bound for the level set is yet to investigated. Our precise result reads as:
\begin{theorem}\label{thm of log delay}
For any $R>0$, for the solution $(F,C,H)$ of the system \textnormal{(\ref{fch-equation})} with the initial data \textnormal{(\ref{intial data})}, it holds:

\begin{equation}\label{log delay for F}
\underset{t\to+\infty}{\liminf}\inf_{x\in B_R, e\in S^{N-1}}\ (F+C)\Big(t,x+\Big(c^*t-\frac{(N+2)c^*}{\min\{1,a\}}\ln t\Big)e\Big)>0\ \ \mbox{if}\ \ a> 1+s,
\end{equation}
\begin{equation}\label{log delay for H 3}
\underset{t\to+\infty}{\limsup}\sup_{x\in B_R, e\in S^{N-1}}\ H\Big(t,x+\Big(c^*t-\frac{(N+2)c^*}{\min\{1,a\}}\ln t\Big)e\Big)<1\ \ \mbox{if}\ \ a> 1+s. 
\end{equation}
Moreover, if $d=1$, it holds:
\begin{equation}\label{log delay for C}
\underset{t\to+\infty}{\liminf}\inf_{x\in B_R, e\in S^{N-1}}\ C\Big(t,x+\Big(c^*t-\frac{N+2}{c^*}\ln t\Big)e\Big)>0\ \ \mbox{if}\ \ a< 1+s,
\end{equation}
\begin{equation}\label{log delay for H 1}
\underset{t\to+\infty}{\limsup}\sup_{x\in B_R, e\in S^{N-1}}\ H\Big(t,x+\Big(c^*t-\frac{N+2}{c^*}\ln t\Big)e\Big)<1\ \ \mbox{if}\ \ a< 1+s, 
\end{equation}
\begin{equation}\label{log delay for H 2}
\underset{t\to+\infty}{\liminf}\inf_{x\in B_R, e\in S^{N-1}}\ H\Big(t,x+\Big(c^*t-\frac{N+2}{c^*}\ln t\Big)e\Big)>0\ \ \mbox{if}\ \ a\neq 1+s. 
\end{equation}
Furthermore, for the special case $a=1+s$ and $d=1$, it holds:
\begin{equation}\label{log delay for F a=1+s}
\underset{t\to+\infty}{\liminf}\inf_{x\in B_R, e\in S^{N-1}}\ (F+C)\Big(t,x+\Big(c^*t-\frac{N+2}{c^*}\ln t\Big)e\Big)>0,
\end{equation}
\begin{equation}\label{log delay for H 1 a=1+s}
\underset{t\to+\infty}{\limsup}\sup_{x\in B_R, e\in S^{N-1}}\ H\Big(t,x+\Big(c^*t-\frac{N+2}{c^*}\ln t\Big)e\Big)<1, 
\end{equation}
\begin{equation}\label{log delay for H 2 a=1+s}
\underset{t\to+\infty}{\liminf}\inf_{x\in B_R, e\in S^{N-1}}\ H\Big(t,x+\Big(c^*t-\frac{N}{c^*}\ln t\Big)e\Big)>0. 
\end{equation}

\end{theorem}

\begin{prop}\label{prop of small peak}
If $a>1+s$, then for the solution $(F,C,H)$ of the system \eqref{fch-equation} with initial data \eqref{intial data}, for any $c<c^*$, it holds:
\begin{equation}\label{small peak}
\limsup_{t\rightarrow\infty}\sup_{\lVert x\rVert\ge ct}F(t,x)>0.
\end{equation}
\end{prop}

\begin{remark}
By combining with Theorem \ref{thm leading edge} and Theorem \ref{thm profile g<1}, Proposition \ref{prop of small peak} explains the reason why a small peak of initial farmers can be observed on the wavefront just in the case $a> 1+s$ and $g<1$.
\end{remark}

\section{Upper estimates on the spreading speeds}
In this section, we deal with the spreading properties of solutions of the system (\ref{fch-equation}) on the leading edge and complete the proof of Theorem \ref{thm leading edge}. Since there is no interaction between farmers and hunter-gatherers on the leading edge, the analysis of this zone is rather straightforward. 
We first prove that the initial farmers $F$ and converted farmers $C$ cannot propagate faster than the speed $c^*$. Hence, on the leading edge, the contact between farmers and hunter-gatherers will never happen. Therefore, 
the population density of hunter-gatherers remain unchanged in this zone. This follows from a simple comparison argument.

We first concern with the analysis of the $F$-component and $C$-component on the leading edge and begin the proof with some simple upper estimates on the spreading speed. Note that, the nonlinear term $aF(1-C-F)$ in the $F$-equation is nonincreasing on the value of $C$, and the $C$-component is always nonnegative.    
Since $c^*\ge 2\sqrt{a}$, for all $e\in S^{N-1}$, one can construct a well-known super-solution $\widebar{F}_e:=A_1e^{-c^*(x\cdot e-c^*t)/2}$ of the equation
$$\partial_tF=\Delta F+aF(1-F).$$
Applying the comparison principle, one has 
\begin{equation}\label{eq of upper estimate of F}
\mbox{for all}\ \ c>c^*,\ \ \underset{t\to+\infty}{\lim}\underset{\lVert x\rVert\ge ct}{\sup}F(t,x)\le \underset{t\to+\infty}{\lim}\underset{\lVert x\rVert\ge ct}{\sup}\ \ \underset{e\in S^{N-1}}{\inf}\widebar{F}_e(t,x)=0,
\end{equation}
provided that $A_1$ is large enough such that $\widebar{F}_e(0,x)\ge F(0,x)$.

\begin{remark}\label{remark of F leading edge}
Even if $1+s>a$,  one may find that the function $\widebar{F}^*_e=A^*e^{-\sqrt{a}(x\cdot e-2\sqrt{a}t)}$ is always a super-solution of the $F$-component, provided that $A^*$ is large enough. It implies that , if $1+s>a$, then for any $c> 2\sqrt{a}$, it holds: 
$$\limsup_{t\to+\infty}\sup_{\lVert x\lVert\ge ct}\ F(t,x)=0.$$
\end{remark}

Moreover, concluding from Theorem \ref{thm profile g<1} and Remark \ref{remark of F leading edge}, the following corollary is an immediate result.
\begin{corollary}\label{col}
If $g<1$ and $a<1+s$, then the solution of the system \textnormal{(\ref{fch-equation})} with the initial data \textnormal{(\ref{intial data})} satisfies:
\begin{equation}\label{F converge to 0 every where}
\underset{t\to\infty}{\limsup}\underset{ x\in\mathbb{R}^N}{\sup}F(t,x)=0,
\end{equation}
provided that
$$g<\frac{\min\{1,a\}}{\min\{1,a\}+s}\ \ \mbox{or}\ \ db\ge \frac{c^*}{1-g}.$$
\end{corollary}

To get the estimate of the $C$-component, we need to construct another suitable super-solution. Note that the nonlinear term $C(1-C-F)+s(C+F)H$ in the $C$-equation is neither monotone increasing nor decreasing with respect to the value of $F$. Hence, to construct a super-solution for the $C$-component, it is necessary to control the value of $F$ suitably. However, thanks to the fact that the $F$-component is sufficient small on the leading edge, one just need to deal with the following equation
\begin{equation}\label{super sol of C}
\partial_tC=\Delta C+C(1-C)+s(C+\widebar{F}_e).
\end{equation}
Then, we introduce the new function
$$\widebar{C}_e:=A_2e^{-\lambda(x\cdot e-ct)}\ \ \mbox{for all}\ \ c>c^*,\ e\in S^{N-1}.$$ 
One may find that, for large enough $A_2>0$, there exists $\lambda=c^*/2$ such that $\widebar{C}_e$ satisfies
\begin{equation*}
\begin{aligned}
&\partial_t\widebar{C}_e-\Delta\widebar{C}_e-\widebar{C}_e(1+s-\widebar{C}_e)-s\widebar{F}_e\\
\ge& \Big((c\lambda-\lambda^2-1-s)A_2-sA_1\Big)e^{-c^*/2(x\cdot e-ct)}\\
\ge&\ 0.
\end{aligned}
\end{equation*}
Thus $\widebar{C}_e$ is a super-solution of the equation (\ref{super sol of C}) for any $t\ge 0$ and $x\in\mathbb{R}^N$.

Moreover, $C(0,x)=0$ for any $x\in\mathbb{R}^N$. We can now apply the comparison principle again to conclude that
\begin{equation}\label{eq of upper estimate of C}
\mbox{for all}\ \ c'>c,\ \ \underset{t\to+\infty}{\lim}\underset{\lVert x\rVert\ge c't}{\sup} C(t,x)\le \underset{t\to+\infty}{\lim}\underset{\lVert x\rVert\ge c't}{\sup}\ \underset{e\in S^{N-1}}{\inf}\widebar{C}_e(t,x)=0.
\end{equation}
For the reason that $c$ can be chosen arbitrarily close to $c^*$, one can conclude that the $C$-component does not spread faster that the spreading speed $c^*$.

To complete this section , we deal with the upper estimate of the $H$-component on leading edge. One may find that the $H$-component is always able to stay positive outside of the farmer's range. 
The main idea is to construct a suitable sub-solution for the $H$-component as follows:
$$\wideubar{H}_e:=1-g(A_1+A_2)e^{-c^*(x\cdot e-ct)/(2d)}\ \ \mbox{for all}\ \  c>c^*,\ e\in S^{N-1}. $$

In previous part of this subsection, we have shown that the $F$-component and $C$-component can be controlled from above for all $c>c^*$ and $e\in S^{N-1}$ by 
$$A_1e^{-c^*(x\cdot e-ct)/2}\ \ \mbox{and}\ \ A_2e^{-c^*(x\cdot e-ct)/2}.$$
Since $bH(1-H-g(F+C))$ is nonincreasing with respect to the value of $F+C$, we just need to construct sub-solution for $\wideubar{H}$ that satisfies
\begin{equation}\label{eq of sub-solution of H}
\left\{
\begin{aligned}
&\partial_t\wideubar{H}=d\Delta\wideubar{H}+b\wideubar{H}(1-g(A_1+A_2)e^{-c^*(x\cdot e-ct)/2}-\wideubar{H}),\\
&\wideubar{H}(0,x)=1.
\end{aligned}
\right.
\end{equation}
Furthermore, since $d\ge 1$, one has 
$$(A_1+A_2)e^{-c^*(x\cdot e-ct)/(2d)}\ge (A_1+A_2)e^{-c^*(x\cdot e-ct)/2}\ \  \mbox{for all}\ \ x\cdot e\ge ct.$$
Hence, for each $e\in S^{N-1}$, $\wideubar{H}_e$ is a sub-solution of the equation (\ref{eq of sub-solution of H}). More precisely, one has
\begin{equation*}
\begin{aligned}
&\partial_t\wideubar{H}_e-d\Delta\wideubar{H}_e-b\wideubar{H}_e\Big(1-g(A_1+A_2)e^{-c^*(x\cdot e-ct)/(2d)}-\wideubar{H}_e\Big)\\
\le& -\Big(cc^*/2d-(c^*)^2/4d\Big)\Big(g(A_1+A_2)e^{-c^*(x\cdot e-ct)/(2d)}\Big)\\
\le&\ 0.
\end{aligned}
\end{equation*}

Moreover, on the one hand, since $\wideubar{H}(0,x)=1$, one gets that 
$$\wideubar{H}(0,x)\ge\wideubar{H}_e(0,x).$$
On the other hand, one can choose $A_1$ and $A_2$ large enough such that, for any $e\in S^{N-1}$,
$$\wideubar{H}_e(t,x)<0\le\wideubar{H}(t,x),\ \ t\ge 0,\ \ x\cdot e=ct.$$
Now,  one can now infer from the comparison principle that 
$$\underset{t\to\infty}{\lim}\underset{\lVert x\lVert\ge c't}{\inf}H(t,x)\ge\underset{t\to\infty}{\lim}\underset{\lVert x\lVert\ge c't}{\inf}\ \underset{e\in S^{N-1}}{\sup}\wideubar{H}_e(t,x)=1\ \ \mbox{for all}\ \ c'>c.$$
Since we can choose $c$ arbitrarily close to $c^*$, the proof of Theorem \ref{thm leading edge} is complete. 
\section{Logarithmic Bramson correction}
At first, we would like to introduce an important result provided in \cite{AD,log delay 3}. Consider the linear equation with drift term as follows:
\begin{equation}\label{symmetric linear eq}
\partial_t z=\partial_{\xi}^2 z + \Big(c^*-\frac{\delta}{t+t_0}+\frac{N-1}{\xi+\xi^{\delta}_{t_0}(t)}\Big)\partial_{\xi} z + \lambda^{*2}z,\ \ t>0,\ \xi>0,
\end{equation}
where 
$$\lambda^*=c^*/2\ \ \mbox{and}\ \ \xi^{\delta}_{t_0}(t):=c^*(t+t_0)-\delta\ln\frac{t+t_0}{t_0}.$$
Then the following proposition holds:
\begin{prop}\label{Dirichlet eq}
Let $z^{\delta}_{t_0}(\xi,t)$ be the solution of the equation \eqref{symmetric linear eq} with boundary condition 
$$z^{\delta}_{t_0}(t,0)=0\ \ \mbox{for all}\ \ t>0,$$ 
and the initial data
$$z^{\delta}_{t_0}(0,\xi)=e^{-\lambda^*\xi}\zeta_0(t_0^{-1/2}\xi)\geq0\ \ \mbox{for all}\ \ \xi\geq0,$$
where $\zeta_0(\cdot)$ is an nontrivial compactly supported smooth function.
Then it holds: 
\begin{equation}\label{estimate Dirichlet eq}
z^{\delta}_{t_0}(t,\xi)=\frac{(t+t_0)^{\gamma-\frac{1}{2}}}{t_0^{\gamma}}\xi e^{-\lambda^*\xi}\Big\{\frac{\int_0^{\infty}\zeta_0(\rho)\rho d\rho+h_1(t,t_0)}{\sqrt{\pi}}e^{-\frac{\xi^2}{4(t+t_0)}}+h_2(t,\xi,t_0)\Big\},\ \ \xi\geq0,\ t\geq0,
\end{equation}
where $\gamma:=\delta\lambda^*-\frac{N+1}{2}$, $h_1$ and $h_2$ are smooth functions satisfying
\[
\left\{
\begin{aligned}
&|h_1(t,t_0)|\leq B_1t_0^{-1/2}\lVert\zeta_0\rVert_m,\\
&|h_2(t,\xi,t_0)|\leq B_2\Big\{\frac{t_0^{1/4}\lVert\zeta_0\rVert_m}{(t+t_0)^{1/2}}+\Big(\frac{t_0}{t+t_0}\Big)^{1-\frac{B_2}{\sqrt{t_0}}}\lVert\partial_{\rho}^2\zeta_0\rVert_m\Big\}e^{-\frac{\xi^2}{8(t+t_0)}},
\end{aligned}
\right.\ \ \ \ \xi\ge0,\ t\ge 0,
\]
for some positive constants $B_1$ and $B_2$. Here, the norm $\lVert\cdot\rVert_m$ is defined as
\[
\lVert\zeta_0\rVert_m^2:=\int_0^{\infty}\zeta_0(\rho)^2e^{\frac{\rho^2}{4}}d\rho.
\]  
\end{prop}
Note that, this proposition is slightly different from those of \cite{AD,log delay 3}. Hence, we will give the proof in the appendix for the sake of completeness.

Now, let us consider the solution $(F,C,H)$ of the system \eqref{fch-equation} with spherically symmetric initial data $(F_0,C_0,H_0)$. By spatially homogeneity of the system and the uniqueness of the solution, the solution $(F,C,H)$ is also spherically symmetric. By changing the variables, the solution $(F(t,r),C(t,r),H(t,r))$ where $r=\lVert x\lVert$ satisfies the following one-dimensional system
\begin{equation}\label{fch-symmetric}
\left\{
\begin{aligned}
&\partial_tF=\partial_r^2 F+\frac{N-1}{r}\partial_r F+aF(1-F-C),\\
&\partial_tC=\partial_r^2 C+\frac{N-1}{r}\partial_r C+C(1-F-C)+sH(F+C),\\
&\partial_tH=d\partial_r^2 H+d\frac{N-1}{r}\partial_r H+bH(1-H-g(F+C)),
\end{aligned}
\right.\quad\quad\xi>0,\ t>0.
\end{equation}
We consider the moving frame as $\xi=r-\xi_{t_0}(t)$. Then, the functions 
\[
(\mathcal{F}(t,\xi),\mathcal{C}(t,\xi),\mathcal{H}(t,\xi)):=(F(t,\xi+\xi_{t_0}(t)),C(t,\xi+\xi_{t_0}(t)),H(t,\xi+\xi_{t_0}(t))
\] 
satisfy the system as follows:
\begin{equation*}
\left\{
\begin{aligned}
&\partial_t \mathcal{F}=\partial_{\xi}^2 \mathcal{F}+\Big(c^*-\frac{N+2}{2\lambda^*(t+t_0)}+\frac{N-1}{\xi+\xi_{t_0}(t)}\Big)\partial_{\xi} \mathcal{F}+a\mathcal{F} (1-\mathcal{F}-\mathcal{C}),\\
&\partial_t \mathcal{C}=\partial_{\xi}^2 \mathcal{C}+\Big(c^*-\frac{N+2}{2\lambda^*(t+t_0)}+\frac{N-1}{\xi+\xi_{t_0}(t)}\Big)\partial_{\xi} \mathcal{C}+\mathcal{C}(1-\mathcal{F}-\mathcal{C})+sH(\mathcal{F}+\mathcal{C}),\\
&\partial_t \mathcal{H}=d\partial_{\xi}^2 \mathcal{H}+d\Big(c^*-\frac{N+2}{2\lambda^*(t+t_0)}+\frac{N-1}{\xi+\xi_{t_0}(t)}\Big)\partial_\xi \mathcal{H}+b\mathcal{H}(1-\mathcal{H}-g(\mathcal{F}+\mathcal{C})),
\end{aligned}
\right.\quad\quad\xi>-\xi^{\delta}_{t_0}(t),\ t>0.
\end{equation*}

\subsection{Upper estimates on the location of the wavefront}
In this subsection, we first construct suitable super-solutions for the $F$-component and the $C$-component, and sub-solution for the $H$-component, respectively. The constructions for the cases $1+s>a$, $1+s<a$ and $a=1+s$ are different. 

We first deal with the case $a>1+s$, which means the $F$-component always moves faster than the $C$-component. The construction for this case is rather straightforward.
Let $\delta=\delta^*:=\frac{N+2}{2\lambda^*}$ and $(\bar{u}(t,\xi),\bar{v}(t,\xi))$ be the solution of 
\[
\left\{
\begin{aligned}
&\partial_t \bar{u}=\partial^2_{\xi}\bar{u}+\Big(2c^*-\frac{\delta^*}{t+t_0}+\frac{N-1}{\xi+\xi^{\delta^*}_{t_0}(t)}\Big)\partial_{\xi}\bar{u}+a \bar{u},\\
&\partial_t \bar{v}=\partial^2_{\xi}\bar{v}+\Big(2c^*-\frac{\delta^*}{t+t_0}+\frac{N-1}{\xi+\xi^{\delta^*}_{t_0}(t)}\Big)\partial_{\xi}\bar{v}+(1+s)\bar{v}+s \bar{u},
\end{aligned}
\right.\quad\quad\xi>0,\ t>0,
\]
with the boundary condition $(\bar{u}(t,0),\bar{v}(t,0))=(0,0)$, and with compactly supported initial data 
$$\bar{u}_0(\xi)=\frac{a-1-s}{s}\bar{v}_0(\xi)=e^{-\lambda^*\xi}\zeta_0(t_0^{-1/2}\xi)\ \ \mbox{for all}\ \ \xi\ge0.$$
Then under the same notation as in Proposition \ref{Dirichlet eq}, one has
\[
(\bar{u}(t,\xi),\bar{v}(t,\xi))=\Big(z^{\delta^*}_{t_0}(t,\xi),\frac{s}{a-1-s}z^{\delta^*}_{t_0}(t,\xi)\Big)\ \ \mbox{for all}\ \ \xi\ge0,\ t\ge0.
\]
Moreover, for any $\mu>0$, $(\bar{u}^*(t,x),\bar{v}^*(t,x)):=\mu(\bar{u}(t,\lVert x\rVert-\xi^{\delta^*}_{t_0}(t)),\bar{v}(t,\lVert x\rVert-\xi^{\delta^*}_{t_0}(t)))$ satisfies 
\[
\left\{
\begin{aligned}
&\partial_t \bar{u}^*\ge\Delta \bar{u}^*+a\bar{u}^*(1-F-C),\\
&\partial_t \bar{v}^*\ge\Delta \bar{v}^*+\bar{v}^*(1-F-C)+s H (\bar{u}^*+\bar{v}^*),
\end{aligned}
\right.\quad\quad \xi>\xi^{\delta^*}_{t_0}(t),\ t>0,
\]
where $(F,C,H)$ is the solution of the system \eqref{fch-equation} with the initial data \eqref{intial data}. 

Next, we define the functions $\widebar{F}(t,x)$ and $\widebar{C}(t,x)$ as: 
\begin{equation*}
\widebar{F}(t,x)=\left\{
\begin{array}{lcl}\vspace{6pt}
&B^*,&\ \lVert x\rVert\le\xi^{\delta^*}_{t_0}(t)+ A,\\
&\min\{B^*, \mu \bar{u}(t,\lVert x\rVert-\xi^{\delta^*}_{t_0}(t))\},&\ \lVert x\rVert \ge\xi^{\delta^*}_{t_0}(t)+ A,
\end{array}
\right.
\end{equation*}
\begin{equation*}
\widebar{C}(t,x)=\left\{
\begin{array}{lcl}\vspace{6pt}
&1+s,&\ \lVert x\rVert\le\xi^{\delta^*}_{t_0}(t)+ A,\\
&\min\{1+s, \mu  \bar{v} (t,\lVert x\rVert-\xi^{\delta^*}_{t_0}(t))\},&\ \lVert x\rVert \ge\xi^{\delta^*}_{t_0}(t) +A,
\end{array}
\right.
\end{equation*}
where $B^*:=\sup_{(t,x)\in\mathbb{R}_+\times\mathbb{R}^N}F(t,x)$.
Furthermore, in the case of $d=1$, we define $\widebar{H}(t,x)$ as
\begin{equation*}
\widebar{H}(t,x)=\left\{
\begin{array}{lcl}\vspace{6pt}
&0,&\ \lVert x\rVert\le\xi^{\delta^*}_{t_0}(t)+ A,\\
&\max\{0, 1-g\mu  (u(t,\lVert x\rVert-\xi^{\delta^*}_{t_0}(t))+v(t,\lVert x\rVert-\xi^{\delta^*}_{t_0}(t)))\},&\ \lVert x\rVert \ge\xi^{\delta^*}_{t_0}(t) +A.
\end{array}
\right.
\end{equation*}
 Note that, we choose constant $\mu>0$ large enough such that, for all $t\ge 0$, 
\[
\begin{aligned}
&\mu \bar{u}(t,A)=\mu z^{\delta^*}_{t_0}(t,A)>B^*,\ \ \mu \bar{v}(t,A)=\frac{\mu s z^{\delta^*}_{t_0}(t,A)}{a-1-s}>1+s,\\ &\mu g(\bar{u}(t,A)+\bar{v}(t,A))=\frac{\mu g(a-1)}{a-1-s}z^{\delta^*}_{t_0}(t,A)>1.
\end{aligned}
\]
We remark that by applying Proposition \ref{Dirichlet eq}, such $\mu>0$ always exists for sufficiently large $t_0>0$. Then by applying the comparison principle, one obtain
\begin{equation}\label{log delay upper bound}
F(t,x)\le\widebar{F}(t,x),\ C(t,x)\le \widebar{C}(t,x),\ H(t,x)\ge \widebar{H}(t,x)\ \ \mbox{for all}\ \ x\in{\mathbb R}^N,\ t>0.
\end{equation}

\begin{remark}
Note that, for the case $d>1$, the functions $\widebar{F}$ and $\widebar{C}$ are still super-solutions for the $F$-component and the $C$-component. However,
the function $\widebar{H}$ could no longer be a suitable sub-solution for the $H$-component.
\end{remark}

Next, we deal with the construction for the case $a=1+s$ and $d=1$.
Under the same notation as in Proposition \ref{Dirichlet eq}, we consider
\[
(\hat{u}(t,\xi),\hat{v}(t,\xi))=\Big(z^{\delta^*}_{t_0}(t,\xi),(1+st)z^{\delta^*}_{t_0}(t,\xi)\Big)\ \ \mbox{for all}\ \ \xi\ge0,\ t\ge0,
\]
where
$z^{\delta^*}_{t_0}(\xi,t)$ is the solution of the equation \eqref{symmetric linear eq} with boundary condition 
$z^{\delta^*}_{t_0}(t,0)=0$ for all $t>0$.Then, we define 
$$(\hat{u}^*(t,x),\hat{v}^*(t,x)):=\mu(\hat{u}(t,\lVert x\rVert-\xi^{\delta^*}_{t_0}(t)),\hat{v}(t,\lVert x\rVert-\xi^{\delta^*}_{t_0}(t))),$$
which satisfies
\[
\left\{
\begin{aligned}
&\partial_t \hat{u}^*=\Delta \hat{u}^*+a\hat{u}^*,\\
&\partial_t \hat{v}^*=\Delta \hat{v}^*+(1+s)\hat{v}^*+s\hat{u}^*,
\end{aligned}
\right.\quad\quad \xi>\xi^{\delta^*}_{t_0}(t),\ t>0.
\]
Then we choose constant $\mu>0$ large enough such that, for all $t\ge 0$, 
\[
\begin{aligned}
&\mu \hat{u}(t,A)=\mu z^{\delta^*}_{t_0}(t,A)>B^*,\ \ \mu \hat{v}(t,A)=\mu (1+st) z^{\delta^*}_{t_0}(t,A)>1+s,\\ & g\mu(\hat{u}(t,A)+\hat{v}(t,A))=\mu g(2+st)z^{\delta^*}_{t_0}(t,A)>1.
\end{aligned}
\]
Then we construct super and sub-solutions as that for the case $a>1+s$,
\begin{equation*}
\widehat{F}(t,x)=\left\{
\begin{array}{lcl}\vspace{6pt}
&B^*,&\ \lVert x\rVert\le\xi^{\delta^*}_{t_0}(t)+ A,\\
&\min\{B^*, \mu \hat{u}(t,\lVert x\rVert-\xi^{\delta^*}_{t_0}(t))\},&\ \lVert x\rVert \ge\xi^{\delta^*}_{t_0}(t)+ A,
\end{array}
\right.
\end{equation*}
\begin{equation*}
\widehat{C}(t,x)=\left\{
\begin{array}{lcl}\vspace{6pt}
&1+s,&\ \lVert x\rVert\le\xi^{\delta^*}_{t_0}(t)+ A,\\
&\min\{1+s, \mu  \hat{v}(t,\lVert x\rVert-\xi^{\delta^*}_{t_0}(t))\},&\ \lVert x\rVert \ge\xi^{\delta^*}_{t_0}(t) +A,
\end{array}
\right.
\end{equation*}
\begin{equation*}
\widehat{H}(t,x)=\left\{
\begin{array}{lcl}\vspace{6pt}
&0,&\ \lVert x\rVert\le\xi^{\delta^*}_{t_0}(t)+ A,\\
&\max\{0, 1-g\mu  (\hat{u}(t,\lVert x\rVert-\xi^{\delta^*}_{t_0}(t))+\hat{v}(t,\lVert x\rVert-\xi^{\delta^*}_{t_0}(t)))\},&\ \lVert x\rVert \ge\xi^{\delta^*}_{t_0}(t) +A.
\end{array}
\right.
\end{equation*}
where $B^*:=\sup_{(t,x)\in\mathbb{R}_+\times\mathbb{R}^N}F(t,x)$.
By applying the comparison principle, one has
\begin{equation}\label{log delay upper bound 3}
F(t,x)\le\widehat{F}(t,x),\ \ C(t,x)\le\widehat{C}(t,x),\ \ H(t,x)\ge\widehat{H}(t,x)\ \ \mbox{for all}\ \ x\in{\mathbb R}^N,\ t>0.
\end{equation}

Next, we deal with the construction for the case $a< 1+s$.
Recall that, in Remark \ref{remark of F leading edge}, we showed that $\widebar{F}^*(t,x)=\min\{A^*,A^*e^{-\sqrt{a}(\lVert x\rVert-2\sqrt{a}t)}\}\ge F(t,x)$ for all $(t,x)\in\mathbb{R}_+\times\mathbb{R}^N$. Now, 
let us consider super-solutions of the $F$-component and $C$-component as
\begin{equation*}
\widetilde{F}(t,x)=\left\{
\begin{array}{lcl}
&\widebar{F}^*(t,x),&\ \lVert x\rVert\le\xi^{\delta^*}_{t_0}(t)+ A,\\
&\min\Big\{\widebar{F}^*(t,x), \mu \tilde{u}(t,\lVert x\rVert-\xi^{\delta^*}_{t_0}(t))\Big\},&\ \lVert x\rVert\ge\xi^{\delta^*}_{t_0}(t)+ A,
\end{array}
\right.
\end{equation*}
\begin{equation*}
\widetilde{C}(t,x)=\left\{
\begin{array}{lcl}
&1+s,&\ \lVert x\rVert\le\xi^{\delta^*}_{t_0}(t)+ A,\\
&\min\Big\{1+s, \mu \tilde{v}(t,\lVert x\rVert-\xi^{\delta^*}_{t_0}(t))\Big\},&\ \lVert x\rVert\ge\xi^{\delta^*}_{t_0}(t)+ A,
\end{array}
\right.
\end{equation*}
and in the case of $d=1$, we consider a sub-solution of the $H$-component as 
\begin{equation*}
\widetilde{H}(t,x)=\left\{
\begin{array}{lcl}
&0,&\ \lVert x\rVert\le\xi^{\delta^*}_{t_0}(t)+ A,\\
&\min\Big\{0,1-g \mu( \tilde{u}(t,\lVert x\rVert-\xi^{\delta^*}_{t_0}(t))+\tilde{v}(t,\lVert x\rVert-\xi^{\delta^*}_{t_0}(t)))\Big\},&\ \lVert x\rVert\ge\xi^{\delta^*}_{t_0}(t)+ A,
\end{array}
\right.
\end{equation*}
where the function $(\tilde{u}(t,\xi),\tilde{v}(t,\xi))$ is the solution of 
\[
\left\{
\begin{aligned}
&\partial_t\tilde{u}=\partial_{\xi}^2\tilde{u}+\Big(2c^*-\frac{\delta^*}{t+t_0}+\frac{N-1}{\xi+\xi^{\delta^*}_{t_0}(t)}\Big)\partial_{\xi}\tilde{u}+(1+s-\varepsilon)\tilde{u},\\
&\partial_t\tilde{v}=\partial_{\xi}^2\tilde{v}+\Big(2c^*-\frac{\delta^*}{t+t_0}+\frac{N-1}{\xi+\xi^{\delta^*}_{t_0}(t)}\Big)\partial_{\xi}\tilde{v}+(1+s)\tilde{v}+s\tilde{u},
\end{aligned}
\right.\quad\quad\xi>0,\ t>0,
\]
with Dirichlet boundary condition $(\tilde{u}(t,0),\tilde{v}(t,0))=(0,0)$ for all $t>0$ and compactly supported initial data $\tilde{u}(0,\xi)=\frac{s}{\varepsilon}\tilde{v}(0,\xi)=e^{-\lambda^*\xi}\zeta_0(t_0^{-1/2}\xi)$. Let us choose $\varepsilon=\sqrt{a}(\sqrt{1+s}-\sqrt{a})$. Then under the same notation as that in Proposition \ref{Dirichlet eq}, 
\[
(\tilde{u}(t,\xi),\tilde{v}(t,\xi))=\Big(e^{-\varepsilon t}z^{\delta^*}_{t_0}(t,\xi),\frac{s(2-e^{-\varepsilon t})}{\varepsilon}z^{\delta^*}_{t_0}(t,\xi)\Big)
\]
Here, we choose $\mu>0$ large enough such that
\[
\begin{aligned}
&\mu \tilde{u}(t,A)=\mu e^{-\varepsilon t}z^{\delta^*}_{t_0}(t,A)>\overline{F}^*(t,x)\big|_{\lVert x\rVert=\xi^{\delta^*}_{t_0}(t)+A},\ \ \mu\tilde{v}(t,A)=\mu\frac{s(2-e^{\varepsilon t})}{\varepsilon}z^{\delta^*}_{t_0}(t,A)>1+s,\\
&\mu g(\tilde{u}(t,A)+\tilde{v}(t,A))=\mu g\frac{2s+(1-s)e^{-\varepsilon t}}{\varepsilon}z^{\delta^*}_{t_0}(t,A)>1.
\end{aligned}
\]
We remark that inferring from Proposition \ref{Dirichlet eq}, such $\mu$ always exists for sufficiently large $t_0>0$. 
Then, for the solution $(F,C,H)$ of the system \eqref{fch-equation} with the initial data \eqref{intial data},  
$$(\tilde{u}^*(t,x),\tilde{v}^*(t,x)):=(\mu\tilde{u}(t,\lVert x\rVert-\xi^{\delta^*}_{t_0}(t)),\mu\tilde{v}(t,\lVert x\rVert-\xi^{\delta^*}_{t_0}))$$
satisfies
\[
\left\{
\begin{aligned}
&\partial_t \tilde{u}^*\ge\Delta \tilde{u}^*+a \tilde{u}^*(1-F-C),\\
&\partial_t \tilde{v}^*\ge\Delta \tilde{v}^*+\tilde{v}^*(1-F-C)+sH(\tilde{u}^*+\tilde{v}^*),
\end{aligned}
\right.\quad\qquad\lVert x\rVert>\xi^{\delta^*}_{t_0}(t),\ t>0.
\]
By applying the comparison principle, one has
\begin{equation}\label{log delay upper bound 2}
F(t,x)\le\widetilde{F}(t,x),\ \ C(t,x)\le\widetilde{C}(t,x),\ \ H(t,x)\ge\widetilde{H}(t,x),\ \ \mbox{for all}\ \ x\in{\mathbb R}^N,\ t>0.
\end{equation}

Then, the following propositions are an immediate result from the upper estimates \eqref{log delay upper bound}, \eqref{log delay upper bound 3} and \eqref{log delay upper bound 2}.
\begin{prop}\label{prop 1 of log delay}
If $a\neq 1+s$, for the solution $(F,C,H)$ of the system \textnormal{(\ref{fch-equation})} with the initial data \textnormal{(\ref{intial data})}, it holds:
\begin{equation}\label{log delay upper bound on F 2}
\underset{t\to+\infty}{\limsup}\underset{c^*t-\frac{N+2}{2\lambda^*}\ln t+r\le\lVert x\rVert}{\sup}F(t,x)\to 0\ \ \mbox{as}\ \ r\to+\infty,
\end{equation}
\begin{equation}\label{log delay upper bound on C 2}
\underset{t\to+\infty}{\limsup}\underset{c^*t-\frac{N+2}{2\lambda^*}\ln t+r\le\lVert x\rVert}{\sup}C(t,x)\to 0\ \ \mbox{as}\ \ r\to+\infty.
\end{equation}
Moreover, if $d=1$, one has:
\begin{equation}\label{log delay upper bound on H 2}
\underset{t\to+\infty}{\liminf}\underset{c^*t-\frac{N+2}{2\lambda^*}\ln t+r\le\lVert x\rVert}{\inf}H(t,x)\to 1\ \ \mbox{as}\ \ r\to+\infty.
\end{equation}
\end{prop}

\begin{prop}\label{prop 1 of log delay 1}
If $a=1+s$, for the solution $(F,C,H)$ of the system \textnormal{(\ref{fch-equation})} with the initial data \textnormal{(\ref{intial data})}, it holds:
\begin{eqnarray}
\label{log delay upper bound on F 2 1}
&&\underset{t\to+\infty}{\limsup}\underset{c^*t-\frac{N+2}{2\lambda^*}\ln t+r\le\lVert x\rVert}{\sup}F(t,x)\to 0\ \ \mbox{as}\ \ r\to+\infty\\,
\label{log delay upper bound on C 2 1}
&&\underset{t\to+\infty}{\limsup}\underset{c^*t-\frac{N}{2\lambda^*}\ln t+r\le\lVert x\rVert}{\sup}C(t,x)\to 0\ \ \mbox{as}\ \ r\to+\infty.
\end{eqnarray}
Moreover, if $d=1$, one has:
\begin{equation}\label{log delay upper bound on H 2 1}
\underset{t\to+\infty}{\liminf}\underset{c^*t-\frac{N}{2\lambda^*}\ln t+r\le\lVert x\rVert}{\inf}H(t,x)\to 1\ \ \mbox{as}\ \ r\to+\infty.
\end{equation}
\end{prop}

Furthermore, the following proposition which is concerned with the upper estimates of the location of the wavefront holds:
\begin{prop}\label{prop 2 of log delay}
If $d=1$, then for the solution $(F,C,H)$ of the system \textnormal{(\ref{fch-equation})} with the initial data \textnormal{(\ref{intial data})} and for any $R>0$, it holds:
\begin{equation}\label{log delay low bound on H}
\underset{t\to+\infty}{\liminf}\inf_{|x|\le R,e\in S^{N-1}}\ H\Big(t,x+\Big(c^*t-\frac{N+2}{2\lambda^*}\ln t\Big)e\Big)>0\ \ \mbox{if}\ \ a\neq 1+s,
\end{equation}
\begin{equation}\label{log delay low bound on H 1}
\underset{t\to+\infty}{\liminf}\inf_{|x|\le R,e\in S^{N-1}}\ H\Big(t,x+\Big(c^*t-\frac{N}{2\lambda^*}\ln t\Big)e\Big)>0\ \ \mbox{if}\ \ a= 1+s.
\end{equation}
\end{prop}
\begin{proof}\it{of Proposition \ref{prop 2 of log delay}.}
\normalfont
We first deal with the proof the statement \eqref{log delay low bound on H}. To proceed by contradiction, we  assume that there exist $r_0\in\mathbb{R}$, $e_{\infty}\in S^{N-1}$ and a sequence of times $t_n\to+\infty$ such that 
$$H\Big(t_n,\Big(c^*t_n-\frac{N+2}{2\lambda^*}\ln t_n+r_0\Big)e_{\infty}\Big)\to 0\ \ \mbox{as}\ \ n\to+\infty.$$
Up to extraction of a subsequence, the functions $H_n(t,x)=H(t+t_n,x+(c^*t_n-\frac{N+2}{2\lambda^*}\ln t_n+r_0)e_{\infty})$ converge locally uniformly to an entire solution $H_{\infty}$ that satisfies
$$\partial_t H_{\infty}=\Delta H_{\infty}+bH_{\infty}(1-gF_{\infty}-gC_{\infty}-H_{\infty}).$$
Since $0\le H_{\infty}\le 1$ in $\mathbb{R}\times\mathbb{R}$ and $H_{\infty}(0,0)=0$. The strong maximum principle implies that $H_{\infty}\equiv 0$. However, from the estimate (\ref{log delay upper bound on H 2}), $H_{\infty}(0,re_{\infty})\ge 1/2$ for $r$ large enough. One has reached a contradiction which implies that the estimate (\ref{log delay low bound on H}) holds true. Since the \eqref{log delay low bound on H 1} follows from the same argument, the proof of this proposition is thereby complete.
\end{proof}
\subsection{Lower estimates on the location of the wavefront}
In this subsection, we complete the proof of Theorem \ref{thm of log delay} by showing lower estimates for the total population $F+C$ of farmers, and the upper estimate for the $H$ at the position $c^*t-O(\ln t)$. 
Let us first consider the case $a= 1+s$, which can be proved by simple comparison argument.
\begin{prop}\label{prop 3 of log delay 1}
Let $(F,C,H)$ be the solution of the system \textnormal{(\ref{fch-equation})} with the initial data \textnormal{(\ref{intial data})} and $G(t,x):=F(t,x)+C(t,x)$. If $a= 1+s$ and $d=1$, then for any $R>0$, it holds:
\begin{equation}\label{log delay lower bound on F 1}
\underset{t\to+\infty}{\liminf}\inf_{|x|\le R,e\in S^{N-1}}\ G\Big(t,x+\Big(c^*t-\frac{N+2}{c^*}\ln t\Big)e\Big)>0,
\end{equation}
\begin{equation}\label{log delay upper bound on H 1}
\underset{t\to+\infty}{\limsup}\sup_{|x|\le R,e\in S^{N-1}}\ H\Big(t,x+\Big(c^* t-\frac{N+2}{c^*}\ln t\Big)e\Big)<1.
\end{equation}
\end{prop}
\begin{proof}{\it of Proposition \ref{prop 3 of log delay 1}.}
In the case of $d=1$, we know that $1-\max\{1,g\}(F+C)$ is a suitable sub-solution of $H$, such that $H\ge 1-\max\{1,g\}(F+C)$ for all $x\in\mathbb{R}^N$ and $t> 0$. Then, since $a=1+s$, one has
\begin{equation*}
\begin{aligned}
\partial_t G-\Delta G&\ge(aF+C)(1-F-C)+sHC\ge(aF+C)(1-G)+s(1-\max\{1,g\}G)C\\
&=aG-(aF+C+\max\{1,g\}sC)G\ge(a-kG)G,
\end{aligned}
\end{equation*}
where $k:=\max\{a,1+\max\{1,g\}sC\}$. Then by applying the argument in \cite{log delay 3} for scalar KPP equation and Proposition \ref{Dirichlet eq}, one can complete the proof of Proposition \ref{prop 3 of log delay 1}.
\end{proof}
\vspace{10pt}

Next, we deal with the case $a>1+s$. Note that, in this case, our estimate is not very sharp, since the coefficient $c^*(N+2)/\min\{1,a\}$ in front of the $\ln t$ is greater than $(N+2)/c^*$ which has been proved by Bramson and Hamel {\it et al.} for the scalar KPP equation.

\begin{prop}\label{prop 3 of log delay}
Let $(F,C,H)$ be the solution of the system \textnormal{(\ref{fch-equation})} with the initial data \textnormal{(\ref{intial data})} and $G(t,x):=F(t,x)+C(t,x)$. If $a> 1+s$, then for any $R>0$, it holds:
\begin{equation}\label{log delay lower bound on F}
\underset{t\to+\infty}{\liminf}\inf_{|x|\le R,e\in S^{N-1}}\ G\Big(t,x+\Big(c^*t-\frac{c^*(N+2)}{\min\{1,a\}}\ln t\Big)e\Big)>0,
\end{equation}
\begin{equation}\label{log delay upper bound on H}
\underset{t\to+\infty}{\limsup}\sup_{|x|\le R,e\in S^{N-1}}\ H\Big(t,x+\Big(c^* t-\frac{c^*(N+2)}{\min\{1,a\}}\ln t\Big)e\Big)<1.
\end{equation}
\end{prop}

\begin{remark}
Proposition \ref{prop 3 of log delay} holds for any $d\ge 1$. Indeed, for all $d\ge 1$, the functions $\widebar{F}$ and $\widebar{C}$ introduced in subsection 4.1 are always suitable super-solutions of the $F$-component and the $C$-component. However, the function $\widebar{H}$ is not a sub-solution of the $H$-component anymore if $d\neq 1$.  As a matter of fact, in the proof of Proposition \ref{prop 3 of log delay}, only the properties of $\widebar{F}$ and $\widebar{C}$ will be used.
\end{remark}

To prove Proposition \ref{prop 3 of log delay}, one need to construct a suitable sub-solution of $G$ near the wavefront. To do this, a lower estimate of $F$ in the region which moves with a speed slightly faster that $c^*$
is very important.
\begin{lemma}\label{lm:log delay lower bdd of F}
Let $(F,C,H)$ be the solution of the system \eqref{fch-equation} with the initial data \eqref{intial data}. If $a>1+s$, then there exist $A_2>0$ and $t_0>0$ such that the following holds
\begin{equation}\label{eq:log delay lower bdd of F}
F(t,x)\ge A_2 (t+t_0)^{-\frac{N+2}{2}}(\lVert x\rVert-c^*(t+t_0))e^{-\lambda^*(\lVert x\rVert-c^*(t+t_0))},
\end{equation}
where
$$c^*(t+t_0)\le\lVert x\rVert\le c^*(t+t_0)+\sqrt{t+t_0}\ \ \mbox{and}\ \ t>1.$$
\end{lemma}
\begin{proof}\it{of Lemma \ref{lm:log delay lower bdd of F}.}
Let $z_{t_0}(t,\xi)$ $(\xi\ge\xi_{t_0}^0(t),\ t\ge0)$ be the solution of 
\[
\left\{\begin{aligned}
&\partial_t z_{t_0}=\Delta z_{t_0}+\Big(c^*+\frac{N+2}{\xi+\xi^0_{t_0}(t)}\Big)\partial_{\xi}z_{t_0}+\lambda^{*2}z_{t_0},\quad \xi>0,\ t>0,\\
&z_{t_0}(t,0)=0,\quad t>0,\\
&z_{t_0}(0,\xi)=e^{-\lambda^*\xi}\zeta_0(t_0^{-1/2}\xi),\quad \xi\ge0,
\end{aligned}\right.
\]
where $\zeta_0\ge0$ is a nontrivial compactly supported smooth function and 
\[
c^*:=2\sqrt{a},\ \lambda^*=c^*/2,\ \xi_{t_0}^0(t):=c^*(t+t_0).
\]
By applying Proposition \ref{Dirichlet eq}, there exist $t_0>0$, $A_1>0$ and $A_2>0$ such that 
\begin{eqnarray}
\label{eq:upper 1}
&&z_{t_0}(t,\xi)\le A_1(t+t_0)^{-\frac{N+2}{2}}\ \ \mbox{for all}\ \  \xi>0,\ t>0,\\
\label{eq:lower 1}
&&z_{t_0}(t,\xi)\ge A_2(t+t_0)^{-\frac{N+2}{2}}\xi e^{\lambda^*\xi}\ \ \mbox{for all}\ \   0<\xi<\sqrt{t+t_0},\ t>0.
\end{eqnarray}
From \eqref{log delay upper bound}, one may obtain that, for $t>0$ and $\xi^0_{t_0}(t)\le\lVert x\rVert\le\xi^0_{t_0}(t)+\sqrt{t+t_0}$,
\[
\begin{split}
C(t,x)&\leq \overline{C}(t,x)=\mu v(t,\lVert x\rVert-\xi^{\delta^*}_{t_0}(t))\\
&\leq A_3 (\lVert x\rVert-\xi^{\delta^*}_{t_0}(t))e^{-\lambda^*(\lVert x\rVert-\xi^{\delta^*}_{t_0}(t))}\\
&\le \Big(A_3\delta^*\log\frac{t+t_0}{t_0}\Big)t^{-\frac{N+2}{2}},
\end{split}
\]
where $\delta^*:=\frac{N+2}{2\lambda^*}$ and $\xi^{\delta}_{t_0}(t):=c^*(t+t_0)-\delta\log\frac{t+t_0}{t_0}$.

Let us introduce a new function $F^{\varepsilon}(t,x):=\varepsilon\omega(t)z_{t_0}(t,\lVert x\rVert-c^*(t+t_0))$. One may find that
\[
\begin{split}
\partial_t F^{\varepsilon}&-\Delta F^{\varepsilon}- a F^{\varepsilon}(1-F^{\varepsilon}-C)
=\Big(\frac{\dot{\omega}}{\omega}+a F^{\varepsilon}+C\Big)F^{\varepsilon}\\
&\le\Big(\frac{\dot{\omega}}{\omega}+(\varepsilon a\omega +c_3\delta^*\log\frac{t+t_0}{t_0})t^{-\frac{N+2}{2}}\Big)F^{\varepsilon}.
\end{split}
\]
Then, we choose a suitable $\omega(t)$  as
\[
\omega(t)=e^{-\int_0^{t}(\varepsilon a +c_3\delta^*\log\frac{\tau+t_0}{t_0})\tau^{-\frac{N+2}{2}}d\tau}.
\]
It is not difficult to check that
$0<\omega\le1$, ${\inf}_{t\ge0}\omega(t)=\omega(\infty)>0$ and 
\[
\partial_t F^{\varepsilon}-\Delta F^{\varepsilon}- a F^{\varepsilon}(1-F^{\varepsilon}-C)\le0\ \ \mbox{for all}\ \ \lVert x\rVert\ge c^*(t+t_0)\ \ \mbox{and}\ \ t>0.
\]
For sufficiently small $\varepsilon>0$, it holds $F(1,x)\ge F^{\varepsilon}(0,x)$ for all $x\in{\mathbb R}^N$. Then, by applying the comparison principle, one obtains
\[
F(t,x)\ge F^{\varepsilon}(t-1,x),\quad x\in{\mathbb R^N},\ t>1.
\]
Therefore, from the estimate \eqref{eq:lower 1}, one can conclude that
\[
F(t,x)\ge \varepsilon \omega(\infty)c_2(t+t_0)^{-\frac{N+2}{2}}(\lVert x\rVert-c^*(t+t_0))e^{-\lambda^*(\lVert x\rVert-c^*(t+t_0))},
\]
where
\[
c^*(t+t_0)\le\lVert x\rVert\le c^*(t+t_0)+\sqrt{t+t_0}\ \ \mbox{and}\ \  t>1.
\]
The proof is complete.
\end{proof}
\vspace{10pt}

Now we are ready to deal with the proof of Proposition \ref{prop 3 of log delay}.

\begin{proof}{\it of Proposition \ref{prop 3 of log delay}.}
Let us denote the total population density of the farmers as $G(t,x):=F(t,x)+C(t,x)$. Then, one may find that
\[
\partial_t G-\Delta G\ge \min\{1,a\}G(1-G)\quad{\rm for}\quad\{(t,x)\mid G(t,x)\le1\}.
\]
We consider a new function as 
\[
\phi(\tau,x):=\phi_{\sqrt{t_1}/2,\eta,e}(\tau,x,t):=\eta e^{\frac{\min\{1,a\}}{2}(\tau-\tau^*(t))}\varphi_{\sqrt{t_1}/2}\Big(x-\Big(c^*(t+t_0)+\frac{\sqrt{t_1+t_0}}{2}\Big)e\Big),
\]
where $\eta>0$, $e\in S^{N-1}$, $\tau^*(t):=\frac{N+2}{\min\{1,a\}}\log t$ and $\varphi_{R}(x)>0$ is the eigenfunction satisfies
\[
\left\{
\begin{aligned}
-\Delta\varphi_R&=\mu_R\varphi_R,\quad \lVert x\rVert<R,\\
\varphi_R(x)&=0,\quad \lVert x\rVert=R,\\
\varphi_R(0)&=1.
\end{aligned}
\right.
\]
Since the eigenvalue satisfies $\mu_R=\mu_1 R^{-2}$, for sufficiently small $\eta>0$ and sufficiently large $t_1>0$, one has
\[
\begin{split}
\partial_{\tau}\phi-\Delta\phi-\min\{1,a\}\phi(1-\phi)&=\Big(\mu_{\sqrt{t_1}/2}-\frac{\min\{1,a\}}{2}+\phi\Big)\phi\\&\le\Big(4\mu_1 t_1^{-1}-\frac{\min\{1,a\}}{2}+\eta\Big)\phi\le0.
\end{split}
\] 
By applying the estimate \eqref{eq:log delay lower bdd of F}, if we choose 
$$\eta=A_2\frac{\sqrt{t_1+t_0}+\sqrt{t_1}}{2}e^{-\lambda^*\frac{\sqrt{t_1+t_0}+\sqrt{t_1}}{2}}(1+t_0/t_1)^{-\frac{N+2}{2}},$$
then for $\lVert x-(c^*(t+t_0)+\frac{\sqrt{t_1+t_0}}{2})e\rVert< \frac{\sqrt{t_1}}{2}$, $t>t_1$ and $e\in S^{N-1}$, one may obtain
\[
\begin{split}
F(t,x)&\ge A_2\frac{\sqrt{t_1+t_0}+\sqrt{t_1}}{2}e^{-\lambda^*\frac{\sqrt{t_1+t_0}+\sqrt{t_1}}{2}}(t+t_0)^{-\frac{N+2}{2}}\\&\ge\eta t^{-\frac{N+2}{2}}\varphi_{\sqrt{t_1}/2}\Big(x-\Big(c^*(t+t_0)+\frac{\sqrt{t_1+t_0}}{2}\Big)e\Big) =\phi_{\sqrt{t_1}/2,\eta,e}(0,x,t).
\end{split}
\]
On the other hand, for $\lVert x-(c^*(t+t_0)+\frac{\sqrt{t_1+t_0}}{2})e\rVert\ge \frac{\sqrt{t_1}}{2}$, $t>t_1$ and $e\in S^{N-1}$, it holds obviously that
\[
F(t,x)>\phi_{\sqrt{t_1}/2,\eta,e}(0,x,t)=0.
\] 
Hence, one has 
\[
G(t,x)\ge F(t,x)\ge\phi_{\sqrt{t_1}/2,\eta,e}(0,x,t)\ \ \mbox{for all}\ \  x\in{\mathbb R}^N,\ t>t_1\ \ \mbox{and}\ \ e\in S^{N-1}.
\]
Thus, by applying the comparison principle, one can conclude that 
\[
G(\tau+t,x)\ge \phi_{\sqrt{t_1}/2,\eta,e}(\tau,x,t)\ \ \mbox{for all}\ \ x\in{\mathbb R}^N,\ 0\le\tau\le\tau^*(t),\ t>t_1\ \ \mbox{and}\ \ e\in S^{N-1}.
\]
Therefore, for any $t_*>t_1+\tau^*(t_1)$, it holds
\[
G(\tau^*(t_*)+t_*,x)\ge\phi_{\sqrt{t_1}/2,\eta,e}(\tau^*(t_*),x,t_*)=\eta\varphi_{\sqrt{t_1}/2}\Big(x-\Big(c^*(t_*+t_0)+\frac{\sqrt{t_1+t_0}}{2}\Big)e\Big).
\]
If we denote $t=t_*+\tau^*(t_*)$, then one can find
\[
\begin{split}
t_*&=t-\tau^*(t-\tau^*(t_*))=t-\frac{N+2}{\min\{1,a\}}\log(t-\tau^*(t_*))\\
&=t-\frac{N+2}{\min\{1,a\}}\log t-\frac{N+2}{\min\{1,a\}}\log\Big(1-\frac{\tau^*(t_*)}{t}\Big)\\
&=t-\frac{N+2}{\min\{1,a\}}\log t-\frac{N+2}{\min\{1,a\}}\log\Big(1-\frac{\tau^*(t_*)}{t_*+\tau^*(t_*)}\Big)
=t-\frac{N+2}{\min\{1,a\}}\log t+\epsilon(t),
\end{split}
\]
where $\underset{t\rightarrow\infty}{\lim} \epsilon(t)=0$. 
Hence, there exist a bounded function $m(t)$ such that, for any large $t$, it holds
\[
\begin{split}
G(t,x)&\ge\phi_{\sqrt{t_1}/2,\eta,e}(\tau^*(t_*),x,t_*)\\&=\eta\varphi_{\sqrt{t_1}/2}\Big(x-\Big(c^*\Big(t-\frac{N+2}{\min\{1,a\}}\log t+\epsilon(t)+t_0\Big)+\frac{\sqrt{t_1+t_0}}{2}\Big)e\Big)\\&=\eta\varphi_{\sqrt{t_1}/2}\Big(x-\Big(c^*t-\frac{(N+2)c^*}{\min\{1,a\}}\log t+m(t)\Big)e\Big).
\end{split}
\]
This estimate implies that the statement \eqref{log delay lower bound on F} holds true.

By applying \eqref{log delay lower bound on F}, the statement \eqref{log delay upper bound on H} follows from a simple limit argument. Thus the proof of Proposition \ref{prop 3 of log delay}  is complete.
\end{proof}

\vspace{10pt}

Next, we complete the proof of Proposition \ref{prop of small peak} which explain the reason why $F$ would not uniformly converge to $0$ on the wavefront in the case of $a>1+s$.

\begin{proof}{\it of Proposition \ref{prop of small peak}.}
We prove by contradiction and assume that 
\[
\limsup_{t\rightarrow\infty}\sup_{c_0t\le\lVert x\rVert}F(t,x)=0\quad {\rm for\ some}\ c_0<c^*=2\sqrt{a}.
\]
Then for any $\varepsilon>0$ there exists $T>0$ such that 
\[
F(t,x)<\varepsilon\ \ \mbox{for}\ \ \lVert x\rVert\ge c_0t,\ t\ge T.
\] 
Hence, if we denote $\wideubar{C}:=C-\alpha F$ where $\alpha>0$, then for $\lVert x\rVert\ge c_0t$ and $t\ge T$, $\wideubar{C}$ satisfies
\[
\begin{split}
\partial_t \wideubar{C}-\Delta \wideubar{C}&\le (1+s)C+s F-a\alpha F+a\alpha F^2 + a\alpha F C\\
&\le(1+s+\varepsilon a\alpha)C+(s-(1-\varepsilon)a\alpha)F\\
&\le(1+s+\varepsilon a\alpha)\wideubar{C}+((1+s+\varepsilon a \alpha)\alpha+(s-(1-\varepsilon)a\alpha))F.
\end{split}
\]
Thus by choosing $\alpha:=2s/(a-\varepsilon a-1-s+\sqrt{(a-\varepsilon a-1-s)^2-4\varepsilon a s})>s/(a-s-1)$, one  can obtain 
\[
\partial_t\wideubar{C}-\Delta \wideubar{C}\le(1+s+\varepsilon a \alpha)\wideubar{C}.
\]
On the other hand, for sufficiently large $A>0$, for all $e\in S^{N-1}$, 
\[
\overline{F}_{e}(t,x):=A e^{-\sqrt{a}(e\cdot x-2\sqrt{a}\,t)},\ \overline{C}_{e}(t,x):=\frac{sA}{a-s-1}e^{-\sqrt{a}(e\cdot x-2\sqrt{a}\,t)}
\]
are super-solutions of $F$ and $C$, respectively, and hence
\[
C(t,x)\le \inf_{e\in S^{N-1}}\overline{C}_e(t,x)=A e^{-\sqrt{a}(\lVert x\rVert-2\sqrt{a}\,t)}\ \ \mbox{for all}\ \  x\in{\mathbb R}^N\ \ \mbox{and}\ \ t\ge0.
\]
Hence, for $\lambda_1<\sqrt{a}$, there exists sufficiently large constant $B>0$ such that 
\[
\wideubar{C}(T,x)\le C(T,x)\le Ae^{-\sqrt{a}(\lVert x\rVert-2\sqrt{a}\,T)}\le B e^{-\lambda_1e\cdot x}\ \ \mbox{for all}\ \  x\in{\mathbb R}^N\ \ \mbox{and}\ \ e\in S^{N-1}.
\]
Therefore, if we take $\lambda_1$ such that $\max\{c_0/2,\sqrt{1+s+\varepsilon a \alpha}\}<\lambda_1<\sqrt{a}$, then 
\[
C_e(t,x):=A_1e^{-\lambda_1(e\cdot x-2\lambda_1t)}
\] 
satisfies 
\[
\partial_t C_e-\Delta C_e\ge (1+s+\varepsilon a\alpha)C_e\ \ \mbox{for all}\ \ x\in{\mathbb R}^N,\ t\ge0,\ e\in S^{N-1},
\]
and
\[
\wideubar{C}(t,x)\le C(t,x)\le 1+s\le A_1 e^{-\lambda_1(e\cdot x-2\lambda_1t)}=C_e(t,x)\ \ \mbox{for all}\ \ \lVert x\rVert\ge c_0t,\ t\ge T,\ e\in S^{N-1}.
\]
Hence, by applying the comparison principle, 
\[
\wideubar{C}(t,x)\le \inf_{e\in S^{N-1}}C_e(t,x)=A_1 e^{-\lambda_1(\lVert x\rVert-2\lambda_1t)}\ \ \mbox{for all}\ \ \lVert x\rVert\ge c_0t,\ t\ge T.
\]
Since $C(t,x)=\wideubar{C}(t,x)+\alpha F(t,x)$, for any $c\in(2\lambda_1,2\sqrt{a})$, one has
\[
\limsup_{t\rightarrow\infty}\sup_{ct\le\lVert x\rVert}C(t,x)\leq\lim_{t\rightarrow\infty} A_1e^{-\lambda_1(c-2\lambda_1)t}+\alpha \limsup_{t\rightarrow\infty}\sup_{ct\le\lVert x\rVert}F(t,x)
\le \alpha\varepsilon.
\]
Since $\varepsilon$ could be chosen arbitrarily small, one obtains
\[
\limsup_{t\rightarrow\infty}\sup_{ct\le\lVert x\rVert}C(t,x)=0\ \ \mbox{for}\ \ c\in(2\lambda_1,2\sqrt{a}).
\]
Hence, for $G(t,x)=F(t,x)+C(t,x)$, it holds
\[
\limsup_{t\rightarrow\infty}\sup_{ct\le\lVert x\rVert}G(t,x)=0\ \ \mbox{for}\ \ c\in(2\lambda_1,2\sqrt{a}).
\]
This contradicts the statement \eqref{log delay lower bound on F} and complete the proof of Proposition \ref{prop of small peak}.
\end{proof}

\vspace{10pt}

At the end of this section, we complete the proof of Theorem \ref{thm of log delay} by showing a lower estimate for the $C$-component at the position $c^*t-\frac{N+2}{2\lambda^*}\ln t$ in the case of $a< 1+s$.
\begin{prop}\label{prop 5 of log delay}
If $a< 1+s$ and $d=1$, then for any $R>0$, it holds:
\begin{equation}\label{log delay lower bound on C}
\underset{t\to+\infty}{\liminf}\inf_{|x|\le R,e\in S^{N-1}}\ C\Big(t,x+\Big(c^*t-\frac{N+2}{2\lambda^*}\ln t\Big)e\Big)>0,
\end{equation}
\begin{equation}\label{log delay upper bound 2 on H}
\underset{t\to+\infty}{\limsup}\sup_{|x|\le R,e\in S^{N-1}}\ H\Big(t,x+\Big(c^*t-\frac{N+2}{2\lambda^*}\ln t\Big)e\Big)<1.
\end{equation}
\end{prop}

\begin{proof}{\it of Proposition \ref{prop 5 of log delay}.}
Let us denote $G(t,x):=F(t,x)+C(t,x)$ and $\wideubar{H}(t,x):=1-\max\{1,g\}G(t,x)$. Then, for $(t,x)\in\{(t,x)\mid G(t,x)\le1\}$, one has
\[
\begin{split}
\partial_t \wideubar{H}-\Delta\wideubar{H}-b\wideubar{H}(1-\wideubar{H}-g G)&=-\max\{1,g\}(\partial_t G-\Delta G)-b(\max\{1,g\}-g)\wideubar{H}G\\
&\le-\max\{1,g\}(aF+C)(1-G)\le0.
\end{split}
\]
On the other hand, for $(t,x)\in\{(t,x)\mid G(t,x)>1\}$, one has
\[
H(t,x)\ge0>\wideubar{H}(t,x).
\] 
Hence by applying the comparison principle, one can obtain that
\[
H(t,x)\ge\wideubar{H}(t,x)\ \ \mbox{for all}\ \ x\in{\mathbb R}^N\ \ \mbox{and}\ \ t\ge0.
\]
Thus, it follows immediately that
\[
\begin{split}
\partial_t C-\Delta C&\ge C(1-(F+C))+sHC\ge C(1-(F+C))+s\wideubar{H}C\\&=C(1+s-k(F+C))\ \ \mbox{for all}\ \  x\in{\mathbb R}^N\ \ \mbox{and}\ \  t>0,
\end{split}
\]
where $k:=1+s\max\{1,g\}$.

By applying the same argument as that in the proof of Proposition \ref{prop of small peak}, one may find
\[
F(t,x)\le A e^{-\sqrt{a}(\lVert x\rVert-2\sqrt{a}\,t)}\ \ \mbox{for all}\ \ x\in{\mathbb R}^N\ \ \mbox{and}\ \ t\ge0.
\]
Let us denote  $c_0:=\sqrt{a}+\sqrt{1+s}\in(2\sqrt{a},2\sqrt{1+s})$ and $\varepsilon_0:=\sqrt{a}(\sqrt{1+s}-\sqrt{a})$, then one can obtain 
\[
F(t,x)\le Ae^{-\varepsilon_0 t}\ \ \mbox{for all}\ \ \lVert x\rVert\ge c_0t\ \ \mbox{and}\ \ t\ge0.
\]
Therefore, for all $\lVert x\rVert\ge c_0t$ and $t>0$, one has 
\[
\partial_t C-\Delta C\ge C(1+s-k C)-Ak e^{-\varepsilon_0 t}C.
\]
By similar arguments to those in the proof of statements \eqref{eq:log delay lower bdd of F} and \eqref{log delay lower bound on F}, one find that
\begin{equation}\label{eq:log delay lower bdd of C}
\begin{split}
C(t,x)\ge &c_2(t+t_0)^{-\frac{N+2}{2}}(\lVert x\rVert-c^*(t+t_0))e^{-\lambda^*(\lVert x\rVert-c^*(t+t_0))}\\&{\rm for}\ \ c^*(t+t_0)\le\lVert x\rVert\le c^*(t+t_0)+\sqrt{t+t_0}\ \ \mbox{and}\ \ t>1,
\end{split}
\end{equation} 
\begin{equation}\label{log delay lower bound on C 0}
\liminf_{t\rightarrow\infty}\inf_{e\in S^{N-1}}C\Big(t,\Big(c^*t-\frac{(N+2)c^*}{1+s}\log t\Big)e\Big)\ (=:c_0)>0.
\end{equation}

Let $U(z)$ be a solution of 
\[
\left\{
\begin{aligned}
&c^* U'+U''+U(1+s-LU)=0,\quad z\in{\mathbb R},\\
&U(-\infty)=\frac{1+s}{L},\ U(+\infty)=0,
\end{aligned}
\right.
\]
where $L:=\max\{k,\frac{2(1+s)}{c_0}\}$. Then $U'(z)<0$ for all $z\in{\mathbb R}$ and 
\begin{equation}\label{eq 20}
\frac{U(z)}{z e^{-\lambda^* z}}\rightarrow B\quad{\rm as}\ z\rightarrow\infty\quad{\rm for\ some\ constant}\ B>0,
\end{equation}
which have been proved in \cite{AW}.
Next, let us introduce two functions
$$u_{r_0}(t,x):=\omega(t)U(\lVert x\rVert-c^*t+\frac{N+2}{2\lambda^*}\log t+r_0),$$ 
$$\omega(t):=e^{-Ak\int_0^{t}e^{-\varepsilon_0\tau}d\tau}.$$
There exists $T_0>0$ such that for $c^*t-\frac{(N+2)c^*}{1+s}\log t\le\lVert x\rVert$, $t\ge T_0$ and $r_0\ge0$, it holds
\[
\begin{split}
\partial_t u_{r_0}&-\Delta u_{r_0}-u_{r_0}(1+s-ku_{r_0})+Ake^{-\varepsilon_0t}u_{r_0}\\
&=\Big(\frac{N+2}{2\lambda^*t}-\frac{N-1}{\lVert x\rVert}\Big)\omega U'+\Big(\frac{\dot{\omega}}{\omega}+Ake^{-\varepsilon_0t}-(L-K\omega)U\Big)u_{r_0}\le0.
\end{split}
\]
From \eqref{eq:log delay lower bdd of C} and  \eqref{log delay lower bound on C 0}, there exists $T_1>0$ such that for $\lVert x\rVert=c^*t-\frac{(N+2)c^*}{1+s}\log t,\ t\ge T_1$ and  $\ r_0\ge0$, it holds
\[
C(t,x)>\frac{c_0}{2}\ge\frac{1+s}{L}>u_{r_0}(t,x).
\]
From the estimates (\ref{eq 20}) and \eqref{eq:log delay lower bdd of C}, there exists $T_2>0$ such that for $\lVert x\rVert=c^*t+\sqrt{t}$, $t\ge T_2$ and $r_0\ge0$, 
\[
\begin{split}
u_{r_0}(t,x)&=\omega(t)U(\sqrt{t}+\frac{N+2}{2\lambda^*}\log t+r_0)\leq 2 B (\sqrt{t}+\frac{N+2}{2\lambda^*}\log t+r_0)e^{-\lambda^*(\sqrt{t}+\frac{N+2}{2\lambda^*}\log t+r_0)}\\&\leq c_2(\sqrt{t}-c^*t_0)e^{-\lambda^*(\sqrt{t}+\frac{N+2}{2\lambda^*}\log (t+t_0)-c^*t_0)}\leq C(t,x).
\end{split}
\]
Moreover, for sufficiently large $r_0>0$, one has
\[
C(T_*,x)\ge u_{r_0}(T_*,x),\quad c^*T_*-\frac{(N+2)c^*}{1+s}\log T_*\le\lVert x\rVert\le c^*T_*+\sqrt{T_*},\ \ \mbox{where}\ \ T_*:=\max\{T_0,T_1,T_2\}.
\]
Therefore, by applying the comparison principle, for $c^*t-\frac{(N+2)c^*}{1+s}\log t\le\lVert x\rVert\le c^*t+\sqrt{t}$ and $t\ge T_*$, one can conclude that
\[
C(t,x)\ge u_{r_0}(t,x)\ge\omega(\infty)U(\lVert x\rVert-c^*t+\frac{N+2}{2\lambda^*}\log t+r_0).
\]
This implies that the statement \eqref{log delay lower bound on C} holds true. 
Further, the statement \eqref{log delay upper bound 2 on H} follows immediately from a simple limit argument. Thus, the proof of Proposition \ref{prop 5 of log delay}  is complete.
\end{proof}

\vspace{10pt}

In conclusion, Theorem \ref{thm of log delay} is an immediate result following from Proposition \ref{prop 2 of log delay}, Proposition \ref{prop 3 of log delay} and Proposition \ref{prop 5 of log delay}.

\section{Lower estimates on the spreading speed}
In this section, we deal with the lower estimate on the spreading speeds of solutions of the system (\ref{fch-equation}). The proof of Theorem \ref{thm final zone 1+s>a} and Theorem \ref{thm final zone 1+s<a}  will be performed through several subsections.  Since original hunter-gatherers, initial farmers and converted farmers heavily interact with each other in the final zone, one need very delicate analysis to get uniform upper estimate of the $H$-component and uniform lower estimate of the $F+C$-component.
Note that, our result only implies that the total population density of initial farmers and converted farmers is uniformly greater than $0$ in the final zone as $t\to+\infty$. However, we still do not how to investigate the population density of each populations of farmers separately.

Before stating our arguments, we would like to introduce some basic properties at first.
Denote $X$ as a Banach space of $\mathbb{R}^3$-valued bounded and uniformly continuous functions on $\mathbb{R}^N$ endowed with the usual sup-norm. Let us define $\Psi(r)\subset X$ as
$$\Psi(r)=\{(\psi_1,\psi_2,\psi_3)\in X:\ 0\le \psi_1,\ 0\le \psi_2,\ \psi_1+\psi_2\le r\ \mbox{and}\ 0\le\psi_3\le 1\}.$$
Although the comparison principle does not hold for the full system, one can apply it on each equation separately. Let us denote the  nonlinear semiflow  generated by the system (\ref{fch-equation}) by $Z(t)$ and add the both sides of the $F$-equation and $C$-equation of the system (\ref{fch-equation}), then one gets
\begin{equation}\label{F+C equation}
\partial_t(F+C)-\Delta(F+C)=(aF+C)(1-F-C)+sH(F+C).
\end{equation} 
Since $0\le H\le 1$, the right hand side of the equation (\ref{F+C equation}) is not greater than
$$\min\{a,1\}\Big(\frac{\max\{a,1\}+s}{\min\{a,1\}}-F-C\Big)(F+C).$$
By applying the comparison principle, one can conclude immediately that 
$$\sup_{x\in\mathbb{R}^N}\Big(F(t,x)+C(t,x)\Big)\le m(t),$$
where $m(t)$ is a function satisfying
$$\frac{dm}{dt}=\min\{a,1\}\Big(\frac{\max\{a,1\}+s}{\min\{a,1\}}-m\Big)m\ \ \mbox{and}\ \ m(0)=\sup_{x\in\mathbb{R}^N}(F(0,x)+C(0,x)).$$
Therefore, if we introduce a new function as
$$M(r):=\max\{r,\frac{\max\{a,1\}+s}{\min\{a,1\}}\},$$
then one may obtain that, for each $r>0$, it holds
$$Z(t)[\Psi(r)]\subset \Psi(M(r))\ \ \mbox{for all}\ \  t>0.$$
\begin{remark} 
If $H_0(x)\equiv 0$, then by applying the comparison principle, one can rewrite the system (\ref{fch-equation}) to a two-component competition system as
\begin{equation*}
\left\{
\begin{aligned}
&\partial_tF=\Delta F+aF(1-F-C),\\
&\partial_tC=\Delta C+C(1-F-C),
\end{aligned}
\right.
\end{equation*}
of which the spreading properties are partly studied by Girardin and Lam in \cite{cs}. 
\end{remark}

Furthermore, for the special case when the diffusion coefficient $d$ of the $H$-equation is equal to $1$, we have a direct observation as follows:
\begin{prop}\label{prop of F+C+H>0 2}
If $d=1$, then there exists $\varepsilon^*>0$ such that the solution $(F,C,H)$ of the system \textnormal{(\ref{fch-equation})} satisfies:
\begin{equation*}
\underset{t\ge 0,x\in\mathbb{R}^N}{\inf}(F+C+H)(t,x)\ge \varepsilon^*,
\end{equation*}  
\end{prop}
provided that the initial data $(F_0,C_0,H_0)\in\Psi(r)$ satisfies $F_0+C_0+H_0\ge \varepsilon^*$.

\begin{proof}\it{of Proposition \ref{prop of F+C+H>0 2}.}
\normalfont
Let us first add the both sides of the $F$-equation, $C$-equation
and $H$-equation. One may find that, it holds 
\begin{align*}
\partial_t(F+C+H)-\Delta(F+C+H)=&aF+C+H-(aF+C)(F+C)\\
&+sH(F+C)-gbH(F+C)-bH^2\\
\ge &\varepsilon_2(F+C+H)-\varepsilon_3(F+C+H)^2,
\end{align*}
where $\varepsilon_2=\min\{1,a,b\}$, $\varepsilon_3=\max\{1,\varepsilon_1,\varepsilon_1b\}$ and $\varepsilon_1=\max\{1,a,s,g\}$.
Therefore, by applying the comparison principle, there exists $\varepsilon^*>0$ such that
\begin{equation*}
\underset{t\ge 0,x\in\mathbb{R}^N}{\inf}(F+C+H)(t,x)\ge \varepsilon^*,
\end{equation*}
provided that $F_0(x)+C_0(x)+H_0(x)\ge \varepsilon^*$.
\end{proof}
\vspace{10pt}

\subsection{Uniform spreading in the final zone ($1+s\ge a$)}
Throughout this paper, we denote $c^0$ as an arbitrarily chosen constant speed in 
$[0,c^*)$.
The proof of Theorem \ref{thm final zone 1+s>a} is split into three steps. We first deal with a weak spreading property which states that, in final zone, for any fixed speed $c\in[0,c^0]$ and direction $e$, $H(t,cte+x)$ does not uniformly converge to $1$ and $(F+C)(t,cte+x)$ does not uniformly converge to $0$. Then, in the second step, we prove that, $H(t,cte+x)$ is uniformly smaller than $1$ and $(F+C)(t,cte+x)$ is uniformly greater than $0$ with respect to $t$. 
At last, we conclude the proof by showing that these properties hold with respect to $\lVert x\lVert\le ct$.

\subsubsection{First step: pointwise weak spreading property}

The first step is to prove the following lemma, from which one can find that the $H$-component does not uniformly converge to $1$, and the $F+C$-component does not uniformly converge to $0$ in the final zone. Moreover, this property is in some sense uniform with respect to the initial data.   

\begin{lemma}\label{lemma of weak pointwise spreading c<c* H}
If $d=1$ and $1+s\ge a$,  
there exists $\varepsilon_1>0$  such that, for any given initial data $(F_0,C_0,H_0)\in\Psi(r)$ satisfying $F_0+C_0+H_0\ge \varepsilon^*$ and $F_0+C_0\not\equiv 0$, for all $c\in[0, c^0]$, $e\in S^{N-1}$ and $x\in\mathbb{R}^N$, the solution $(F,C,H)$ of the system \textnormal{(\ref{fch-equation})} satisfies:
\begin{equation}\label{eq of weak pointwise upper estimate of H c<c*}
\underset{t\to+\infty}{\liminf}\ H(t,x+cte)\le 1-\varepsilon_1,
\end{equation}
\begin{equation}\label{eq of weak pointwise lower estimate of F+C c<c*}
\underset{t\to+\infty}{\limsup}\ (F+C)(t,x+cte)\ge \varepsilon_1.
\end{equation}  
\end{lemma}

\begin{remark}
Note that, from the statement of Lemma \ref{lemma of weak pointwise spreading c<c* H}, it is immediately that $\varepsilon_1(c^0)$ can be chosen to be nonincreasing with respect to $c^0$. Hence, we slightly change our notation and denote it as $\varepsilon_1$.
\end{remark}

\begin{proof}\it{of Lemma \ref{lemma of weak pointwise spreading c<c* H}.}
\normalfont
For $H_0(x)\equiv 0$, since the $F$-component and $C$-component are nonnegative, the estimate (\ref{eq of weak pointwise upper estimate of H c<c*}) holds immediately with $\varepsilon_1=1$. Moreover, from Proposition \ref{prop of F+C+H>0 2}, the statement (\ref{eq of weak pointwise lower estimate of F+C c<c*}) also holds true. Hence, without loss of generality, we assume $H_0(x)\not\equiv 0$.
We argue by contradiction once again and assume there exist sequences
$$\{(F_{0,n},C_{0,n},H_{0,n})\}_{n\ge 0},\ \ \{c_n\}_{n\ge0}\subset[0,c^0],\ \ \{x_n\}_{n\ge0}\subset\mathbb{R}^N,$$
$$\{e_n\}_{n\ge 0}\subset S^{N-1}\ \ \mbox{and}\ \ \{t_n\}_{n\ge0}\subset[0,\infty)\ \ \mbox{such that}\ \ t_n\to+\infty, $$
such that one of the following statements hold true:
\begin{equation}\label{construction of F+C (lower estimate) c<c*}
\mbox{for all}\ \ t\ge t_n,\ (F_n+C_n)(t,x_n+c_nte_n)\le \frac{1}{n},
\end{equation}
\begin{equation}\label{construction of H (upper estimate) c<c*}
\mbox{for all}\ \ t\ge t_n,\ H_n(t,x_n+c_nte_n)\ge 1-\frac{1}{n},
\end{equation}
wherein $(F_n,C_n,H_n)$ denotes the solution with the initial data $(F_{0,n},C_{0,n},H_{0,n})$. Note without loss of generality that 
$$c_n\to c_{\infty}\in [0,c^0]\ \ \mbox{and}\ \ e_n\to e_{\infty}\in S^{N-1}.$$ 

Then, we first claim that 
\begin{claim}\label{claim from pointwise converge H F+C c<c*}
Either (\ref{construction of F+C (lower estimate) c<c*}) or (\ref{construction of H (upper estimate) c<c*}) holds true, there exists a sequence $\{t'_n\}_{n\ge0}$ satisfying $t'_n\ge t_n$ such that, for any $R>0$, it holds:
\begin{equation}\label{uniform converget to 0 of F+C c<c* in claim}
\lim_{n\to\infty}\sup_{t\ge0,x\in B_R}(F_n+C_n)(t'_n+t,x_n+c_n(t'_n+t)e_n+x)=0,
\end{equation}
\begin{equation}\label{uniform converget to 1 of H c<c* in claim}
\lim_{n\to\infty}\sup_{t\ge0,x\in B_R}|1-H_n(t'_n+t,x_n+c_n(t'_n+t)e_n+x)|=0.
\end{equation}
\end{claim}

\begin{proof}\it{of Claim \ref{claim from pointwise converge H F+C c<c*}.}
\normalfont
We first prove that the statement (\ref{construction of H (upper estimate) c<c*}) implies the statements (\ref{uniform converget to 0 of F+C c<c* in claim}) and (\ref{uniform converget to 1 of H c<c* in claim}) hold true.
To proceed by contradiction, we assume that for any $R>0$, there exist $\delta>0$, $s_n>t_n$ and $x'_n\in B_R$ such that 
$$|1-H_n(s_n,x_n+c_ns_ne_n+x'_n)|\ge \delta.$$
Due to standard parabolic estimates, possibly along a subsequence, one may assume that 
\begin{equation*}
\left\{
\begin{aligned}
&\underset{n\to\infty}{\lim}F_n(s_n+t,x_n+c_n(s_n+t)e_n+x)= F_{\infty}(t,x),\\
&\underset{n\to\infty}{\lim}C_n(s_n+t,x_n+c_n(s_n+t)e_n+x)= C_{\infty}(t,x),\\
&\underset{n\to\infty}{\lim}H_n(s_n+t,x_n+c_n(s_n+t)e_n+x)= H_{\infty}(t,x).
\end{aligned}
\right.
\end{equation*}
The above convergences hold locally uniformly in $(t,x)\in\mathbb{R}\times \mathbb{R}^N$ and $(F_{\infty},C_{\infty},H_{\infty})$ is an entire solution of the following system
\begin{equation}\label{F+C infty system}
\left\{
\begin{aligned}
&\partial_tF_{\infty}=\Delta F_{\infty}+c_{\infty}\nabla F_{\infty}\cdot e_{\infty}+aF_{\infty}(1-C_{\infty}-F_{\infty}),\\
&\partial_tC_{\infty}=\Delta C_{\infty}+c_{\infty}\nabla C_{\infty}\cdot e_{\infty}+C_{\infty}(1-C_{\infty}-F_{\infty})+s(F_{\infty}+C_{\infty})H_{\infty},\\
&\partial_tH_{\infty}=\Delta H_{\infty}+c_{\infty}\nabla H_{\infty}\cdot e_{\infty}+bH_{\infty}(1-gF_{\infty}-gC_{\infty}-H_{\infty}).
\end{aligned}
\right.
\end{equation}

From the strong maximum principle and the construction (\ref{construction of H (upper estimate) c<c*}), one has $H_{\infty}\equiv 1$, and hence $F_{\infty}+C_{\infty}\equiv 0$ by considering the $H_{\infty}$-equation in the system (\ref{F+C infty system}). However, since the sequence $\{x'_n\}\subset B_R$ is relatively compact, $H_{\infty}\equiv 1$ contradicts the fact that $|1-H_{\infty}|(0,x'_{\infty})\ge \delta$. It proves that the statement (\ref{uniform converget to 1 of H c<c* in claim}) holds true. The  statement (\ref{uniform converget to 0 of F+C c<c* in claim}) follows from the same approach.

Next, we prove that the statement (\ref{construction of F+C (lower estimate) c<c*}) implies the statements (\ref{uniform converget to 0 of F+C c<c* in claim}) and (\ref{uniform converget to 1 of H c<c* in claim}) hold true.
To proceed by contradiction, we assume that for any $R>0$, there exist $\delta>0$, $s_n>t_n$ and $x'_n\in B_R$ such that 
$$(F_n+C_n)(s_n,x_n+c_ns_ne_n+x'_n)\ge \delta.$$
Due to standard parabolic estimates, possibly along a subsequence, one may assume that
\begin{equation*}
\left\{
\begin{aligned}
&\underset{n\to\infty}{\lim}F_n(s_n+t,x_n+c_n(s_n+t)e_n+x)= F_{\infty}(t,x),\\
&\underset{n\to\infty}{\lim}C_n(s_n+t,x_n+c_n(s_n+t)e_n+x)= C_{\infty}(t,x),\\
&\underset{n\to\infty}{\lim}H_n(s_n+t,x_n+c_n(s_n+t)e_n+x)= H_{\infty}(t,x).
\end{aligned}
\right.
\end{equation*}
The above convergences hold locally uniformly in $(t,x)\in\mathbb{R}\times \mathbb{R}^N$ and $(F_{\infty},C_{\infty},H_{\infty})$ is an entire solution of the system (\ref{F+C infty system}).

From the strong maximum principle and the construction (\ref{construction of F+C (lower estimate) c<c*}), one has $F_{\infty}+C_{\infty}\equiv 0$. Then, from Proposition \ref{prop of F+C+H>0 2}, one has $H_{\infty}\ge \varepsilon^*>0$,  and hence $H_{\infty}\equiv 1$. However, since the sequence $\{x'_n\}\subset B_R$ is relatively compact, $F_{\infty}+C_{\infty}\equiv 0$ contradicts the fact that $(F_{\infty}+C_{\infty})(0,x'_{\infty})\ge \delta$. It proves that the statement  (\ref{uniform converget to 0 of F+C c<c* in claim}) holds true. The  statement (\ref{uniform converget to 1 of H c<c* in claim}) follows from the same approach.
\end{proof}
\vspace{10pt}

Now, we can go back to the proof of Lemma \ref{lemma of weak pointwise spreading c<c* H}.  From the statements (\ref{uniform converget to 0 of F+C c<c* in claim}) and (\ref{uniform converget to 1 of H c<c* in claim}), for any $R>0$ and small enough $\delta>0$ , for any $n$ large enough, one has for all $t>0$ and $x\in\mathbb{R}^N$,
$$F_n(t_n+t,x_n+c_n(t_n+t)e_n+x)\le \chi_{\mathbb{R}^N\setminus B_R}+\delta\chi_{B_R}(x):=\widebar{F}(x),$$
and
$$H_n(t_n+t,x_n+c_n(t_n+t)e_n+x)\ge (1-\delta)\chi_{B_R}(x):=\wideubar{H}(x).$$
Then one infers from the comparison principle that
$$C_n(t_n+t,x_n+c_n(t_n+t)e_n+x)\ge \wideubar{C}_n(t,x)\ \mbox{for all}\ \ t\ge 0,\ x\in\mathbb{R}^N,$$
wherein $\wideubar{C}_n$ is the solution of 
\begin{equation}\label{equation of sub solution of C_n g>1}
\left\{
\begin{aligned}
&\partial_t\wideubar{C}_n=\Delta\wideubar{C}_n+c_n\nabla\wideubar{C}_n\cdot e_n+\wideubar{C}_n(1+s\wideubar{H}-\widebar{F}-\wideubar{C}_n),\\
&\wideubar{C}_n(0,x)=C_n(t_n,x_n+c_nt_ne_n+x).
\end{aligned}
\right.
\end{equation}

For each $R>0$, let $\phi_R$ to be the principal eigenfunction as 
\begin{equation}\label{eigenfunction}
\left\{
\begin{aligned}
&\Delta\phi_R=\mu_R\phi_R\ \ \mbox{in}\ \ B_R,\\
&\phi_R=0\ \ \mbox{on}\ \ \partial B_R,
\end{aligned}
\right.
\end{equation}
that is
normalized so that $\lVert\phi_R\lVert=1$, and extend it by $0$ outside of the ball $B_R$. Then, we construct a stationary sub-solution $\psi(t,x;\eta)$, for each $\eta>0$, as
$$\psi_n(t,x;\eta)=\eta e^{-c_nx\cdot e_n/2}\phi_R(x).$$
Since $c^0<c^*=2\sqrt{1+s}$, one can check that there exist $\eta_0$ depending only on $c^0$ such that for any $\delta$ small enough, $0<\eta\le \eta_0$ and $R$ large enough,  for each $n>0$, $\psi_n$ satisfies

\begin{equation*}
\begin{aligned}
&\partial_t\psi_n-\Delta\psi_n-c_n\nabla\psi_n\cdot e_n-\psi_n(1-\psi_n+s(1-\delta)-\delta)\\
=&\Big(c_n^2/4-\mu_R-(1+s(1-\delta)-\delta)\Big)\psi_n(t,x;\eta)+\psi_n(t,x;\eta)^2\\
\le&\ 0.
\end{aligned}
\end{equation*}
Moreover, since $\mbox{supp}\psi_n\subset B_R$, the function $\psi_n$ is  a  stationary sub-solution of the equation (\ref{equation of sub solution of C_n g>1}).
Therefore, the solution of the equation (\ref{equation of sub solution of C_n g>1}) associated with initial data $\psi_n(t,x;\eta)$ is increasing in time, and converges to a positive stationary solution that is denoted by $p_{n,R,\delta}(x)$. 

Moreover, we claim that:
\begin{claim}\label{claim  of stationary solution}
$p_{n,R,\delta}(x)$ does not depend on the choice of $\eta\in(0,\eta_0]$\ . 
\end{claim}
\begin{proof}\it{of Claim \ref{claim  of stationary solution}.}
\normalfont
To check this, let us change our notation for simplicity and denote the stationary sub-solution and  this stationary solution as $\psi$ and $p_{\eta}$. 
We first note that the comparison principle implies that $p_{\eta}\le p_{\eta'}$ for any $\eta <\eta'$. Next, let us assume by contradiction that there exists $\eta_1<\eta_0$ with $p_{\eta_1}\not\equiv p_{\eta_0}$. Hence, infer from the strong maximum principle, one has $p_{\eta_1}< p_{\eta_0}$. Moreover, there exists a point $x_0\in B_R$ such that $\psi(0,x_0;\eta_0)>p_{\eta_1}(x_0)$. Indeed, if not, then $\psi(0,x;\eta_0)\le p_{\eta_1}(x)$ for all $x\in\mathbb{R}^N$, which yields $p_{\eta_1}\ge p_{\eta_0}$, and reaches a contradiction.

Then, we consider
$$\eta^*=\sup\{\eta\ge\eta_1\ ;\ \psi(0,x;\eta)\le p_{\eta_1}(x)\ \ \mbox{for all}\ \  x\in\mathbb{R}^N\}.$$
One can deduce from the comparison principle and the strong maximum principle that 
$$\psi(0,x;\eta^*)<\psi(t,x;\eta^*)<p_{\eta_1}(x)\ \ \mbox{for all}\ \ t>0,\ x\in\mathbb{R}^N. $$
On the other hand, from the definition of $\eta^*$ and recalling that the function $\psi$ has compact support $B_R$, there exists $x_0\in B_R$ such that $\psi(0,x_0;\eta^*)=p_{\eta_1}(x_0)$, which reaches a contradiction. 
\end{proof} 
\vspace{10pt}

Now, since the initial data satisfies $F_0(x)+C_0(x)\not\equiv 0$ and $H_0(x)\not\equiv 0$, the strong maximum principle implies that $C_n$ is not trivial. Hence, we can choose $\eta$ sufficiently small such that $\wideubar{C}_n(0,x)\ge\psi_n(0,x;\eta)$ for all $x\in\mathbb{R}^N$. Then, it follows from the comparison principle that for any
$R>0$ large enough and $\delta>0$ small enough and $n$ large enough, it holds
\begin{equation}\label{eq 1}
\liminf_{t\to\infty}C_n(t_n+t,x_n+c_n(t_n+t)e_n+x)\ge \liminf_{t\to\infty}\wideubar{C}_n(t,x)\ge p_{n,R,\delta}(x)\ \ \mbox{for all}\ \ x\in\mathbb{R}^N.
\end{equation}

To complete the proof of this lemma, it remains to check that $p_{n,R,\delta}$ is far way from $0$ as $n$ and $R$ are large enough and $\delta$ is small.
Since $p_{n,R,\delta}$ is bounded from above by $1+s$, one can use standard elliptic estimates to get that, as $n\to +\infty$, $R\to+\infty$ and $\delta\to 0$, the function $p_{n,R,\delta}(x)$ converges locally uniformly to a stationary solution $p_{\infty}(x)$ of the equation
\begin{equation*}
\Delta p_{\infty}+c_{\infty}\nabla p_{\infty}\cdot e_n+p_{\infty}(1+s-p_{\infty})=0.
\end{equation*}
Moreover, since the map $t\to\psi(t,x;\eta_0)$ is nondecreasing, one has $p_{n,R,\delta}(0)\ge\psi(0,0;\eta_0)\ge\eta_0\psi_R(0)$. Note that $\varphi_R\to 1$ locally uniformly as $R\to+\infty$, hence $p_{\infty}(0)\ge\eta_0$ and $p_{\infty}(x)>0$ for all $x\in\mathbb{R}^N$. Therefore, from the  statements (\ref{uniform converget to 0 of F+C c<c* in claim}) and (\ref{eq 1}), one can reach a contradiction and completes the proof of  Lemma \ref{lemma of weak pointwise spreading c<c* H}.
\end{proof}
\vspace{10pt}
\subsubsection{Second step: pointwise strong spreading property}
Next, we deal with the following improved result of Lemma \ref{lemma of weak pointwise spreading c<c* H}. 
\begin{lemma}\label{lemma of strong spreading of H c<c*}
If $d=1$ and $1+s\ge a$, there exists $\varepsilon_2>0$ such that, for any given initial data $(F_0,C_0,H_0)\in\Psi(r)$ satisfying $F_0+C_0+H_0\ge \varepsilon^*$ and $F_0+C_0\not\equiv 0$, for all $c\in[0, c^0]$, $e\in S^{N-1}$ and $x\in\mathbb{R}^N$, the solution $(F,C,H)$ of the system \textnormal{(\ref{fch-equation})} satisfies:
\begin{equation}\label{eq of strong estimate of H  c<c*}
\underset{t\to+\infty}{\limsup}\ H(t,x+cte)\le 1-\varepsilon_2,
\end{equation}
\begin{equation}\label{eq of strong estimate of F+C c<c*}
\underset{t\to+\infty}{\liminf}\ (F+C)(t,x+cte)\ge \varepsilon_2.
\end{equation}
\end{lemma}

\begin{proof}\it{of Lemma \ref{lemma of strong spreading of H c<c*}.}
\normalfont
We first deal with the proof of the statement (\ref{eq of strong estimate of H  c<c*}).
Without loss of generality, we assume $H_0(x)\not\equiv 0$.
We proceed by contradiction and  assume that there exists a sequences $\{x_n\}_{n>0}\subset\mathbb{R}^N$, $\{c_n\}_{n>0}\subset[0, c^0]$ and $\{e_n\}_{n>0}\subset S^{N-1}$ such that
\begin{equation*}
\limsup_{t\to+\infty}H_n(t,x_n+c_nte_n)\ge 1-\frac{1}{n}.
\end{equation*}
From Lemma \ref{lemma of weak pointwise spreading c<c* H}, there exist two sequences $\{t_n\}_{n\ge 0}$ with $t_n\to\infty$ and $\{s_n\}_{n\ge 0}\subset\mathbb{R}_+$ such that for each $n>0$,
\begin{equation*}
H_n(t_n+s_n,x_n+c_n(t_n+s_n)e_n)=1-\frac{1}{n},
\end{equation*}
\begin{equation*}
H_n(t,x_n+c_nte_n)\ge 1-\frac{\varepsilon_1}{2}\ \ \mbox{for all}\ \ t\in[t_n,t_n+s_n],
\end{equation*}
\begin{equation*}
H_n(t_n,x_n+c_nt_ne_n)=1-\frac{\varepsilon_1}{2}.
\end{equation*}

We assume as before, possibly along a subsequence, the functions 
$$(F_n, C_n, H_n)(t_n+t,x_n+c_n(t_n+s_n)e_n+x)$$
converge locally uniformly to $(F_{\infty}, C_{\infty}, H_{\infty})$, which is an entire solution of
\begin{equation}\label{eq F C H infty without drift 2}
\left\{
\begin{aligned}
&\partial_tF_{\infty}=\Delta F_{\infty}+aF_{\infty}(1-C_{\infty}-F_{\infty}),\\
&\partial_tC_{\infty}=\Delta C_{\infty}+C_{\infty}(1-C_{\infty}-F_{\infty})+s(F_{\infty}+C_{\infty})H_{\infty},\\
&\partial_tH_{\infty}=\Delta H_{\infty}+bH_{\infty}(1-gF_{\infty}-gC_{\infty}-H_{\infty}).
\end{aligned}
\right.
\end{equation}
From the choices of sequences $\{t_n\}_{n\ge 0}$ and $\{s_n\}_{n\ge 0}$, one has $H_{\infty}(0,0)=1$, and hence $H_{\infty}\equiv 1$. In particular, the sequence $\{s_n\}_{n\ge 0}$ is unbounded since it contradicts the fact that
\begin{equation*}
\lim_{n\to\infty}H_n(t_n,x_n+c_nt_ne_n)=1-\frac{\varepsilon_1}{2}<1.
\end{equation*}
Therefore, we assume that $s_n\to+\infty$ as $n\to\infty$.

Now let us consider the limit functions as follows:
\begin{equation*}
\widetilde{F}(t,x)=\lim_{n\to\infty}F_n(t_n+t,x_n+c_nt_ne_n+x),
\end{equation*}
\begin{equation*}
\widetilde{C}(t,x)=\lim_{n\to\infty}C_n(t_n+t,x_n+c_nt_ne_n+x),
\end{equation*}
\begin{equation*}
\widetilde{H}(t,x)=\lim_{n\to\infty}H_n(t_n+t,x_n+c_nt_ne_n+x),
\end{equation*}
which are well defined thanks to parabolic estimates. The pair $(\widetilde{F}, \widetilde{C}, \widetilde{H})$ is an entire solution of the system (\ref{fch-equation}). 
Then we look on $(\widetilde{F}, \widetilde{C}, \widetilde{H})$ as a solution of the system (\ref{fch-equation}) with initial data
\begin{equation*}
(\widetilde{F}_0, \widetilde{C}_0, \widetilde{H}_0):=\lim_{n\to\infty}(F_n(t_n,x_n+c_nt_ne_n+x), C_n(t_n,x_n+c_nt_ne_n+x), H_n(t_n,x_n+c_nt_ne_n+x)).
\end{equation*}
Note that, it follows from Proposition \ref{prop of F+C+H>0 2} that $\widetilde{F}_0(x)+\widetilde{C}_0(x)+\widetilde{H}_0(x)\ge \varepsilon^*$.

Since $\widetilde{H}_0(0)=1-\varepsilon_1/2$, by applying Proposition \ref{prop of F+C+H>0 2}, one may find $\widetilde{F}_0(x)+\widetilde{C}_0(x)\not\equiv 0$. Thus, by applying Lemma \ref{lemma of weak pointwise spreading c<c* H}, one has
\begin{equation}\label{eq 5}
\mbox{for all}\ \ x\in\mathbb{R}^N,\ \ \liminf_{t\to\infty}\widetilde{H}(t,x+c_{\infty}te)\le 1-\varepsilon_1. 
\end{equation}
One the other hand, for all $t\in[0,s_n)$, it holds
\begin{equation*}
H_n(t_n+t,x_n+c_nt_ne_n+c_nte_n)\ge 1-\frac{\varepsilon_1}{2}.
\end{equation*}
Since $s_n\to+\infty$, we get by the locally uniform convergence that
\begin{equation*}
\widetilde{H}(t,c_{\infty}te_{\infty})\ge 1-\frac{\varepsilon_1}{2} \ \ \mbox{for all}\ \  t\ge 0,
\end{equation*}
which contradicts the result (\ref{eq 5}) concluded from Lemma \ref{lemma of strong spreading of H c<c*}. Thus, the proof of the statement (\ref{eq of strong estimate of H c<c*}) is complete. The statement (\ref{eq of strong estimate of F+C c<c*}) follows immediately from the same approach.
\end{proof}
\vspace{10pt}

\subsubsection{Third step: uniform spreading property}
In this subsection,  we complete the proof of Theorem \ref{thm final zone 1+s>a} by showing that results of Lemma \ref{lemma of strong spreading of H c<c*} holds uniform on $\lVert x\lVert\le ct$ for all $0\le c< c^*$.

\begin{lemma}\label{lemma of uniform spreading of H and  F+C c<c* on x}
If $d=1$ and $1+s\ge a$, there exists $\varepsilon_3>0$ such that, for any $c\in[0.c^0]$, for any given initial data $(F_0,C_0,H_0)\in \Psi(r)$ satisfying $F_0+C_0+H_0\ge \varepsilon^*$ and $F_0+C_0\not\equiv 0$, the solution $(F,C,H)$ of the system \textnormal{(\ref{fch-equation})} satisfies:
\begin{equation*}
\begin{array}{cl}
&\underset{t\to+\infty}{\liminf}\underset{\lVert x\lVert\le ct}{\inf}(F+C)(t,x)\ge \varepsilon_3,\\
&\underset{t\to+\infty}{\limsup}\underset{\lVert x\lVert\le ct}{\sup}H(t,x)\le 1-\varepsilon_3.
\end{array}
\end{equation*}
\end{lemma}

\begin{proof}\it{of Lemma \ref{lemma of uniform spreading of H and  F+C c<c* on x}.}
\normalfont
We proceed by contradiction and assume that there exist sequences $\{t_n\}_{n\ge 0}$ with $t_n\to+\infty$, $\{c_n\}_{n\ge 0}\subset[0,c_0]$ and $\{e_n\}_{n\ge 0}\subset S^{N-1}$ such that
\begin{equation}\label{1}
\lim_{n\to+\infty}H(t_n,c_nt_ne_n)=1.
\end{equation}
Without loss of generality, possibly along a subsequence, we assume that $c_n\to c_{\infty}$ and $e_n\to e_{\infty}$ as $n\to+\infty$. Choose some small $\delta>0$ such that $c_{\infty}+\delta<c^*$, and define the sequence 
\begin{equation*}
t'_n:=\frac{c_nt_n}{c_{\infty}+\delta}\in[0,t_n)\ \ \mbox{for all}\ \ n\ge0.
\end{equation*}

Let us first consider the case when the sequence $\{c_nt_n\}_{n\ge 0}$ is bounded, which may happen if $c_{\infty}=0$. Then one can infer from the strong maximum principle that as $n\to+\infty$ that $c_nt_ne_n\to x_{\infty}\in\mathbb{R}^N$ and 
$$H(t_n+t, c_nt_ne_n+x)\to 1$$
locally uniformly. Thus, one obtains that $H(t_n,0)\to 1$, which already contradicts the result of Lemma \ref{lemma of strong spreading of H c<c*} with $c=0$.
Therefore, one can assume that $t'_n\to+\infty$. Then, by applying Lemma \ref{lemma of strong spreading of H c<c*} again, one has
\begin{equation}\label{2}
H(t'_n,(c_{\infty}+\delta)t'_ne_{\infty})\le 1- \varepsilon
\end{equation}
for each $n$ large enough.

Next, let us consider the functions as follows:
\begin{equation*}
\left\{
\begin{aligned}
&\widetilde{F}_n(t,x)=F(t'_n+t,c_nt_ne_{\infty}+x),\\
&\widetilde{C}_n(t,x)=C(t'_n+t,c_nt_ne_{\infty}+x),\\
&\widetilde{H}_n(t,x)=H(t'_n+t,c_nt_ne_{\infty}+x),
\end{aligned}
\right.
\end{equation*}
and define the sequences
\begin{equation*}
\widetilde{c}_n:=\frac{c_nt_n\lVert e_n-e_{\infty}\lVert}{t_n-t'_n}\to 0\ \ \mbox{and}\ \ \widetilde{e}_n:=\frac{e_n-e_{\infty}}{\lVert e_n-e_{\infty}\lVert}.
\end{equation*}
Using the above notations, one can rewrite statements (\ref{1}) and (\ref{2}) as 
$$\widetilde{H}_n(0,0)\le 1-\varepsilon \ \ \mbox{and}\ \ \widetilde{H}_n(t_n-t'_n,\widetilde{c}_n(t_n-t'_n)\widetilde{e}_n)\to 1.$$
By introducing two time sequences
$$\widetilde{t}_n:=\sup\Big\{0\le t\le t_n-t'_n\ \Big|\ \widetilde{H}_n(t,\widetilde{c}_nt\widetilde{e}_n)> 1-\frac{\varepsilon}{2}\Big\}\in(0,t_n-t'_n),$$
$$\widetilde{s}_n:=t_n-t_n'-\widetilde{t}_n,$$
one may find  that the following properties hold true:
$$\widetilde{H}_n(\widetilde{t}_n,\widetilde{c}_n\widetilde{t}_n\widetilde{e}_n)=1-\frac{\varepsilon}{2},$$
$$\widetilde{H}_n(t,\widetilde{c}_nt\widetilde{e}_n)\le 1- \frac{\varepsilon}{2}\ \ \mbox{for all}\  \ t\in[\widetilde{t}_n,\widetilde{t}_n+\widetilde{s}_n],$$
$$\widetilde{H}_n(\widetilde{t}_n+\widetilde{s}_n,\widetilde{c}_n(\widetilde{t}_n+\widetilde{s}_n)\widetilde{e}_n)\to 1\ \ \mbox{as}\ \ n\to+\infty.$$
Proceeding the argument as in the proof of Lemma \ref{lemma of strong spreading of H c<c*}, one can  reach a contradiction.
The other statement follows from the same approach. The proof of Lemma \ref{lemma of uniform spreading of H and  F+C c<c* on x} and Theorem \ref{thm final zone 1+s>a} are complete.
\end{proof}
\subsection{Uniform spreading in the final zone ($1+s<a$)}
Before dealing with the proof of Theorem \ref{thm final zone 1+s<a}, we note that, by proceeding the same argument in subsection 5.1, one can obtain a lemma as follows:
\begin{lemma}\label{lemma of uniform spreading of H and  F+C c<1+s}
If $d=1$ and $1+s<a$, for any $c\in[0,c^{**})$, there exists $\varepsilon>0$ such that, for any given initial data $(F_0,C_0,H_0)\in \Psi(r)$ satisfying $F_0+C_0+H_0\ge \varepsilon^*$ and $F_0+C_0\not\equiv 0$, the solution $(F,C,H)$ of the system \textnormal{(\ref{fch-equation})} satisfies:
\begin{equation*}
\begin{array}{cl}
&\underset{t\to+\infty}{\liminf}\underset{\lVert x\lVert\le ct}{\inf}(F+C)(t,x)\ge \varepsilon,\\
&\underset{t\to+\infty}{\limsup}\underset{\lVert x\lVert\le ct}{\sup}H(t,x)\le 1-\varepsilon.
\end{array}
\end{equation*}
\end{lemma}

Indeed, for all $c\in[0,c^*)$, one could always conclude a similar weak pointwise spreading result as Lemma \ref{lemma of weak pointwise spreading c<c* H}. However, to obtain the strong pointwise spreading property for $c^{**}<c<c^*$, the argument in Lemma \ref{lemma of strong spreading of H c<c*} is not workable anymore. More precisely, one could not ensure $\widetilde{F}(0,x)\not\equiv 0$, which is a necessary condition to apply the weak pointwise spreading property to reach the contradiction.
In this section, we complete the proof of Theorem \ref{thm final zone 1+s<a} by applying a totally different approach.
\vspace{10pt}

\begin{proof}\it{of Theorem \ref{thm final zone 1+s<a}.}
\normalfont
We complete the proof by showing that for any $0<c_1<c_2<c^*$, there exists $\varepsilon>0$ such that
$$\liminf_{t\to+\infty}\inf_{c_1t\le\lVert x\lVert\le c_2t}(F+C)(t,x)\ge\varepsilon.$$
To do this, we start by assuming there exist sequences $\{t_n\}_{n\ge 0}\subset\mathbb{R}_+$ with $t_n\to+\infty$ and $\{x_n\}_{n\ge 0}\subset\mathbb{R}^N$ with $c_1t_n\le \lVert x_n\lVert\le c_2t_n$, such that $(F+C)(t_n,x_n)\le 1/n$.

We first add the both sides of the $F$-equation and $C$-equation, 
and find $G=F+C$ satisfies
$$\partial_tG\ge \Delta G+\min\{1,a\}G(1-G)\ \ \mbox{for all}\ \ (t,x)\in\mathbb{R}_+\times\mathbb{R}^N.$$
Then, let us consider a  stationary sub-solution of the $G$-equation as 
$\varphi_R(x):=\eta\phi_R(x)$, where $\phi_R(x)$ is the principal eigenfunction defined as (\ref{eigenfunction}). One can check that
$$\eta\Delta \phi_R(x)+\min\{1,a\}\eta\phi_R(x)(1-\eta\phi_R(x))\ge 0,$$
provided that $\eta$ is small enough and $R$ is large enough.
Recall that, in Theorem \ref{thm of log delay}, we proved that, for all $R'>R>0$, there exists $\varepsilon'>0$ such that
$$\liminf_{t\to+\infty}\inf_{x\in B_{R'},e\in S^{N-1}}\ G\Big(t,x+\Big(c^*t-\frac{c^*(N+2)}{\min\{1,a\}}\ln t\Big)e\Big)>\varepsilon'\ \ \mbox{if}\ \ 1+s<a.$$
Hence, for any $t'>T_0$, one can choose $\eta$ small enough such that
\begin{equation*}
G(t',x)\ge \varphi_R\Big(x-\Big(c^*t'-\frac{c^*(N+2)}{\min\{1,a\}}\ln t'\Big)e\Big)\ \ \mbox{for all}\ \ x\in\mathbb{R}^N,\ \ e\in S^{N-1}.
\end{equation*}
Since $\varphi_R(x)$ is a stationary sub-solution of the $G$-equation, by applying the comparison principle, one obtains
$$G(t,x)\ge \varphi_R\Big(x-\Big(c^*t'-\frac{c^*(N+2)}{\min\{1,a\}}\ln t'\Big)e\Big)\ \ \mbox{for all}\ \ t>t', \ \ x\in\mathbb{R}^N,\ \ e\in S^{N-1}.$$
This implies that, for any $t'\ge T_0$ and $t\ge t'$, it holds
\begin{equation}\label{eq 7}
G\Big(t,\Big(c^*t'-\frac{c^*(N+2)}{\min\{1,a\}}\ln t'\Big)e\Big)\ge \eta\phi_R(0).
\end{equation}
Moreover, since $0<c_1<c_2<2\sqrt{a}$, for each large enough $n$, one can find $t_n>t_n'\ge T_0$ such that 
$$x_n=\Big(c^*t_n'-\frac{c^*(N+2)}{\min\{1,a\}}\ln t_n'\Big)e.$$
Thus, from the estimate (\ref{eq 7}), one gets
$G(t_n,x_n)\ge \eta\phi_R(0)$,
which contradicts that
$$G(t_n,x_n)=F(t_n,x_n)+C(t_n,x_n)\le \frac{1}{n}\to 0 \ \ \mbox{as}\ \ n\to+\infty.$$ 
Therefore, the proof of Theorem \ref{thm final zone 1+s<a} is complete.
\end{proof}
\vspace{10pt}
\begin{remark}
Note that, one also can prove Theorem \ref{thm final zone 1+s>a} for the case $d=1$ by applying the same argument as above.
\end{remark}

\section{Asymptotic profiles in the final zone}

In section 5, we have already shown that the propagation of farmers occurs with the speed $c^*$. However, whether the profiles of solutions converges to the steady states $(\widehat{F},\widehat{C},0)$ or $(0,C^*,H^*)$ are still unknown.  In this section, we mainly deal with the asymptotic profiles of solutions in the final zone. From the numerical work of Aoki {\it et al.}, we expect that the profiles of solutions in the final zone are different between the high conversion rate case and the low conversion rate case. Therefore, we split the justification of the numerical results into two parts by dealing with the cases $g\ge 1$ and $g< 1$, respectively.
\subsection{Asymptotic profiles in the high conversion rate case ($g\ge 1$)}
Our first result comes from a direct observation on the $F$-equation of the system (\ref{fch-equation}).

\begin{prop}\label{thm profile c<1,a}
If $g\ge 1$, then for any  $0\le c<\min\{2,2\sqrt{a}\}$, the solution $(F,C,H)$ of the system \textnormal{(\ref{fch-equation})} with the initial data \textnormal{(\ref{intial data})} satisfies:
\begin{equation}\label{profile of H c<1,a}
\limsup_{ t \to \infty}\sup_{ \lVert x\rVert\le ct}H(t,x)=0,
\end{equation}
\begin{equation}\label{profile of F+C c<1,a}
\limsup_{ t \to \infty}\sup_{ \lVert x\rVert\le ct}|1-(F+C)|(t,x)=0.
\end{equation}
\end{prop}

\begin{proof}\it{of Proposition \ref{thm profile c<1,a}.}
\normalfont
Let us consider the solution $G_2(t,x)$ of the following equation
\begin{equation*}
\left\{
\begin{aligned}
&\partial_tG_2=\Delta G_2+\min\{1,a\}G_2(1-G_2),\\
&G_2(0,x)=G_{2,0}(x).
\end{aligned}
\right.
\end{equation*}
Note that, $G_2(t,x)\le 1$ for all $t\ge 0$ and $x\in\mathbb{R}^N$, provided that the initial data $G_{2,0}(x)\le 1$.
Moreover, the function $G=F+C$ satisfies
\begin{align*}
\partial_t(F+C)&=\Delta (F+C)+(aF+C)(1-F-C)+sH(F+C)\\
&\ge\Delta G+\min\{1,a\}G(1-G).
\end{align*}
The last inequality holds for all $(t,x)\in\{(t,x)\ |\ G(t,x)\le 1\}$. 
Thus, by applying the comparison principle and the spreading properties of KPP equation, one can claim that:
\begin{claim}\label{claim direct observation}
For both cases $g\ge 1$ and $g<1$, it holds:
\begin{equation}\label{lower estimate profile F+C c<1,a}
\liminf_{t\to+\infty}\inf_{\lVert x \lVert\le ct}(F+C)(t,x)\ge 1,
\end{equation}
for all $c\in[0,\min\{2,2\sqrt{a}\})$.
\end{claim}

Now, let us choose sequences $\{c_n\}_{n\ge 0}\subset[0,c_0]$ where $0\le c_0<\min\{2,2\sqrt{a}\}$, $\{t_n\}_{n\ge 0}\subset\mathbb{R}_+$ with $t_n\to+\infty$ and $\{x_n\}_{n\ge 0}\subset\mathbb{R}^N$ with $\lVert x_n\lVert\le c_nt_n$. Then, we consider the limit functions
$$\lim_{n\to+\infty}F(t_n+t,x_n+x)=F_{\infty}(t,x),$$
$$\lim_{n\to+\infty}C(t_n+t,x_n+x)=C_{\infty}(t,x),$$
$$\lim_{n\to+\infty}H(t_n+t,x_n+x)=H_{\infty}(t,x),$$
which converge locally uniformly to $(F_{\infty}, C_{\infty}, H_{\infty})$, an entire solution of the system 
\begin{equation*}
\left\{
\begin{aligned}
&\partial_tF_{\infty}=\Delta F_{\infty}+aF_{\infty}(1-C_{\infty}-F_{\infty}),\\
&\partial_tC_{\infty}=\Delta C_{\infty}+C_{\infty}(1-C_{\infty}-F_{\infty})+s(F_{\infty}+C_{\infty})H_{\infty},\\
&\partial_tH_{\infty}=d\Delta H_{\infty}+bH_{\infty}(1-gF_{\infty}-gC_{\infty}-H_{\infty}).
\end{aligned}
\right.
\end{equation*}
Note that, the lower estimate (\ref{lower estimate profile F+C c<1,a}) implies that $(F_{\infty}+C_{\infty})(t,x)\ge 1$ for all $(t,x)\in\mathbb{R}\times\mathbb{R}^N$. 

Next, we show that $H_{\infty}(t,x)\equiv 0$. Indeed, since $(F_{\infty}+C_{\infty})(t,x)\ge 1$ for all $(t,x)\in\mathbb{R}\times\mathbb{R}^N$, one may find $H_{\infty}$ is a sub-solution of $\widetilde{H}_{\infty}$ which satisfies
\begin{equation*}
\left\{
\begin{aligned}
&\partial_t\widetilde{H}_{\infty}=\Delta \widetilde{H}_{\infty}+b\widetilde{H}_{\infty}(1-g-\widetilde{H}_{\infty}),\\
&\widetilde{H}_{\infty}(0,x)=H_{\infty}(0,x).
\end{aligned}
\right.
\end{equation*}
It is clear that, if $g> 1$, the function $(t,x)\to e^{-(g-1)(t+t_0)}$ is a super-solution of the above equation for any $t>-t_0$. 
Since $H_{\infty}(-t_0,x)\le 1$ for any $t_0\in\mathbb{R}_+$, it follows from the comparison principle that 
$$H_{\infty}(0,x)\le e^{-(g-1)t_0}.$$
By passing the limit as $t_0\to+\infty$, one gets that $H_{\infty}\equiv 0$. Therefore, for any arbitrarily chosen sequences $\{t_n\}_{n\ge 0}\subset\mathbb{R}_+$ with $t_n\to+\infty$ and $\{x_n\}_{n\ge 0}\subset\mathbb{R}^N$ with $\lVert x_n\lVert\le ct_n$, one can obtain
$$\lim_{n\to+\infty}H(t_n,x_n)= 0,$$
which implies (\ref{profile of H c<1,a}) holds true. If $g=1$, one can consider the super-solution as $1/(bt+bt_0+b)$, for all $-t_0< t< +\infty$ where $t_0\in\mathbb{R}_+$. 

To complete the proof of this proposition, we consider the limit functions $(F_{\infty},C_{\infty},H_{\infty})$ again, for which  sequences $\{t_n\}_{n\ge 0}\subset\mathbb{R}_+$ with $t_n\to+\infty$ and $\{x_n\}_{n\ge 0}\subset\mathbb{R}^N$ with $\lVert x_n\lVert\le ct_n$ are chosen arbitrarily. The statement (\ref{profile of H c<1,a}) and the strong maximum principle implies that $H_{\infty}\equiv 0$. Then, by applying Claim \ref{claim direct observation} again, one can conclude that $(F_{\infty}+C_{\infty})\equiv 1$, which completes the proof of the statement (\ref{profile of F+C c<1,a}).
\end{proof}
\vspace{10pt}

However, to investigate the profiles of solutions in the region of $c^{**}t\le \lVert x\lVert \le c^0t$, one need to apply the uniform lower estimate of the $F+C$-component in Theorem \ref{thm final zone 1+s>a} and Theorem \ref{thm final zone 1+s<a}.  Since the arguments for proving Theorem \ref{thm profile g>1 1+s>a} and Theorem \ref{thm profile g>1 1+s<a} are almost same, here we just show the proof of Theorem \ref{thm profile g>1 1+s>a}.
\vspace{10pt}

\begin{proof}\it{of Theorem \ref{thm profile g>1 1+s>a}.}
\normalfont
Let us choose sequences  $\{c_n\}_{n\ge 0}\subset[0,c^0]$, $\{t_n\}_{n\ge 0}\subset\mathbb{R}_+$ with $t_n\to+\infty$ as $n\to+\infty$, and $\{x_n\}_{n\ge 0}\subset\mathbb{R}^N$ with $\lVert x_n\lVert\le c_nt_n$. Next, we prove
$$\lim_{n\to+\infty}H(t_n,x_n)=0\ \ \mbox{and}\ \ \lim_{n\to+\infty}(F+C)(t_n,x_n)=1.$$
We consider the limit functions
$$\lim_{n\to+\infty}F(t_n+t,x_n+x)=F_{\infty}(t,x),$$
$$\lim_{n\to+\infty}C(t_n+t,x_n+x)=C_{\infty}(t,x),$$
$$\lim_{n\to+\infty}H(t_n+t,x_n+x)=H_{\infty}(t,x),$$
which converge locally uniformly to $(F_{\infty}, C_{\infty}, H_{\infty})$, an entire solution of the system 
\begin{equation*}
\left\{
\begin{aligned}
&\partial_tF_{\infty}=\Delta F_{\infty}+aF_{\infty}(1-C_{\infty}-F_{\infty}),\\
&\partial_tC_{\infty}=\Delta C_{\infty}+C_{\infty}(1-C_{\infty}-F_{\infty})+s(F_{\infty}+C_{\infty})H_{\infty},\\
&\partial_tH_{\infty}=\Delta H_{\infty}+bH_{\infty}(1-gF_{\infty}-gC_{\infty}-H_{\infty}).
\end{aligned}
\right.
\end{equation*}
The result of Theorem \ref{thm final zone 1+s>a} implies that $H_{\infty}(x)\le1-\varepsilon$ and $(F_{\infty}+C_{\infty})\ge\varepsilon$ for all $(t,x)\in\mathbb{R}\times\mathbb{R}^N$. Adding the both sides of the $F_{\infty}$-equation and $C_{\infty}$-equation, one may find that $1-(1-\varepsilon)e^{-(1-\varepsilon)(t+t_0)}$ is a sub-solution of $(F_{\infty}+C_{\infty})(t,x)$ for all $x\in\mathbb{R}^N$ and $t>-t_0$ where $t_0\in\mathbb{R}_+$. By passing $t_0\to+\infty$, one obtains that $(F_{\infty}+C_{\infty})(0,x)\ge 1$. Since the sequences $\{t_n\}_{n\ge 0}$, $\{c_n\}_{n\ge 0}$ and $\{x_n\}_{n\ge 0}$ are chosen arbitrarily, one can conclude that 
\begin{equation*}
\liminf_{ t \to \infty}\inf_{ \lVert x\rVert\le ct}(F+C)(t,x)\ge 1\ \ \mbox{for all}\ \ c\in[0,c^0].
\end{equation*}
By applying the same argument as Proposition \ref{thm profile c<1,a}, the above statement implies that, if $g\ge 1$,
\begin{equation*}
\limsup_{ t \to \infty}\sup_{ \lVert x\rVert\le ct}H(t,x)= 0\ \ \mbox{for all}\ \ c\in[0,c^0].
\end{equation*}
Then, one can conclude that
\begin{equation*}
\limsup_{ t \to \infty}\sup_{ \lVert x\rVert\le ct}|1-(F+C)(t,x)|= 0\ \ \mbox{for all}\ \ c\in[0,c^0].
\end{equation*}
Since $c^0$ can be chosen arbitrary close to $c^*$, the proof of the statement (\ref{profile of F+C g>1 1+s>a}) is complete. Furthermore, by applying Remark \ref{remark of F leading edge}, the statement (\ref{profile of F and C g>1 1+s>a}) follows immediately.
\end{proof}
\vspace{10pt}

\subsection{Asymptotic profiles in the low conversion rate case ($g<1$)}
The key point of studying the asymptotic profiles in the low conversion rate case is to provide a uniform lower estimate of the $H$-component in the final zone. However, for the general case, it is hard to give the necessary and sufficient condition under which the $H$-component is uniform positive from below.
In this subsection, we first show two sufficient conditions for obtaining the uniform lower estimate of the $H$-component. Then, we show that the $C$-component and $H$-component would converge to $(C^*,H^*)$ as $t\to+\infty$ in the final zone. 

The first sufficient condition means if conversion rate $g$ is small enough, then the population density of hunter-gatherers alway stay uniform positive. 
\begin{lemma}\label{low estimate of H 1}
If $g<\min\{1,a\}/(\min\{1,a\}+s)$, there exists $\varepsilon>0$ such that, for any given initial data $(F_0,C_0,H_0)\in \Psi(r)$ satisfying $H_0\equiv 1$, the solution $(F,C,H)$ of the system \textnormal{(\ref{fch-equation})} satisfies:
\begin{equation*}
\underset{t\to+\infty}{\liminf}\underset{ x\in\mathbb{R}^N}{\inf}H(t,x)\ge \varepsilon.\\
\end{equation*}
\end{lemma}

\begin{proof}\it{of Lemma \ref{low estimate of H 1}.}
\normalfont
The proof of this lemma is rather straightforward. Adding the both sides of the $F$-equation and $C$-equation of the system (\ref{fch-equation}), the function $G=F+C$ satisfies
$$\partial_t(F+C)-\Delta (F+C)=(aF+C)(1-F-C)+sH(F+C).$$
The right hand of the above equation can be rewritten as
$$(F+C)(1+sH-F-C)+(a-1)F(1-F-C)\le G(1+s-G)\ \ \mbox{if}\ \ a\ge 1,\ \ G\ge 1, $$
$$(F+C)(a+sH-aF-aC)+(1-a)C(1-F-C)\le G(a+s-aG)\ \ \mbox{if}\ \ a\le 1,\ \ G\ge 1.$$
Since $H(t,x)\le 1$, $F(t,x)\ge 0$ and $C(t,x)\ge 0$ for all $(t,x)\in\mathbb{R}\times\mathbb{R}^N$, then by applying the comparison principle, one obtains
$$\underset{t\to+\infty}{\limsup}\underset{ x\in\mathbb{R}^N}{\sup}(F+C)(t,x)\le (\min\{1,a\}+s)/\min\{1,a\}.$$
Thus, for any sufficiently small $\varepsilon>0$, there exist $T>0$, such that
$$\sup_{x\in\mathbb{R}^N}\ G(t,x)\le \frac{\min\{1,a\}+s}{\min\{1,a\}}+\varepsilon\ \ \mbox{for all}\ \ t\ge T.$$
Then, by applying the comparison principle again, one has $H(t,x)\ge\wideubar{H}(t,x)$ for all $(t,x)\in\mathbb{R}\times\mathbb{R}^N$, which satisfies the equation as follows:
\begin{equation*}
\left\{
\begin{aligned}
&\partial_t\wideubar{H}=\Delta \wideubar{H}+b\wideubar{H}(1-g\varepsilon-g(\min\{1,a\}+s)/\min\{1,a\}-\wideubar{H}),\\
&\wideubar{H}(0,x)= 1.
\end{aligned}
\right.
\end{equation*}
Therefore, one can conclude that 
$$\underset{t\to+\infty}{\liminf}\underset{ x\in\mathbb{R}^N}{\inf}H(t,x)\ge 1-g(\min\{1,a\}+s)/\min\{1,a\}-\varepsilon g,$$
which completes the proof.
\end{proof}
\vspace{10pt}

The second sufficient condition states that if the diffusion speed $d$ and intrinsic growth rate $b$ of hunter-gatherers are large enough,  then the population density of hunter-gatherers alway stay uniform positive. 

\begin{lemma}\label{low estimate of H 2}
If $bd\ge c^*/(1-g)$, for any $c\in[0,c^*)$, there exists $\varepsilon>0$ such that,  for any given initial data $(F_0,C_0,H_0)\in \Psi(r)$ satisfying $F_0\not\equiv 0$ and $H_0\equiv 1$, the solution $(F,C,H)$ of the system \textnormal{(\ref{fch-equation})} satisfies:
\begin{equation*}
\underset{t\to+\infty}{\liminf}\underset{ \lVert x\lVert\le ct}{\inf}H(t,x)\ge \varepsilon.\\
\end{equation*}
\end{lemma}
The proof of this lemma is similar to that for Theorem \ref{thm final zone 1+s>a}. The first step is to show a weak pointwise property as follows:

\begin{lemma}\label{low estimate of H 3}
If $bd\ge c^*/(1-g)$, there exists $\varepsilon_1>0$ such that, for any given initial data $(F_0,C_0,H_0)\in \Psi(r)$ satisfying $H_0\not\equiv 0$, for all $c\in[0,c^0]$, $e\in S^{N-1}$ and $x\in\mathbb{R}^N$, the solution $(F,C,H)$ of the system \textnormal{(\ref{fch-equation})} satisfies:
\begin{equation*}
\underset{t\to+\infty}{\limsup}\ H(t,cte+x)\ge \varepsilon_1.\\
\end{equation*}
\end{lemma}
\begin{proof}\it{of Lemma \ref{low estimate of H 3}.}
\normalfont
For $F_0(x)+C_0(x)\equiv 0$, the lemma holds immediately with $\varepsilon_1=1.$ Hence, without loss of generality, we assume $F_0(x)+C_0(x)\not\equiv 0$.
We argue by contradiction once again and assume there exist sequences
$$\{(F_{0,n},C_{0,n},H_{0,n})\}_{n\ge 0},\ \ \{c_n\}_{n\ge0}\subset[0,c^0],\ \ \{x_n\}_{n\ge0}\subset\mathbb{R}^N,$$
$$\{e_n\}_{n\ge 0}\subset S^{N-1}\ \ \mbox{and}\ \ \{t_n\}_{n\ge0}\subset[0,\infty)\ \ \mbox{such that}\ \ t_n\to+\infty, $$
such that the following statement holds true
\begin{equation}\label{construction of H (lower estimate) c<c*}
\mbox{for all}\ \ t\ge t_n,\ H_n(t,x_n+c_nte_n)\le \frac{1}{n},
\end{equation}
wherein $(F_n,C_n,H_n)$ denotes the solution with the initial data $(F_{0,n},C_{0,n},H_{0,n})$. Note without loss of generality that 
$$c_n\to c_{\infty}\in [0,c^0]\ \ \mbox{and}\ \ e_n\to e_{\infty}\in S^{N-1}.$$ 

Then, by applying a similar argument to that for Claim \ref{claim from pointwise converge H F+C c<c*}, one can claim that 
\begin{claim}
If (\ref{construction of H (lower estimate) c<c*}) holds true, then there exists a sequence $\{t'_n\}_{n\ge0}$ satisfying $t'_n\ge t_n$ such that, for any $R>0$, it holds:
\begin{equation}\label{uniform  F+C<1 c<c* in claim}
\lim_{n\to\infty}\sup_{t\ge0,x\in B_R}(F_n+C_n)(t'_n+t,x_n+c_n(t'_n+t)e_n+x)\le 1,
\end{equation}
\begin{equation}\label{uniform converget to 0 of H c<c* in claim}
\lim_{n\to\infty}\sup_{t\ge0,x\in B_R}H_n(t'_n+t,x_n+c_n(t'_n+t)e_n+x)=0.
\end{equation}
\end{claim}

Now, we can go back to the proof of Lemma \ref{low estimate of H 3}.  From the statement (\ref{uniform F+C<1 c<c* in claim}), for any $R>0$ and small enough $\delta>0$ , for any $n$ large enough, one has for all $t>0$ and $x\in\mathbb{R}^N$,
$$(F_n+C_n)(t_n+t,x_n+c_n(t_n+t)e_n+x)\le \frac{\min\{1,a\}+s}{\min\{1,a\}}\chi_{\mathbb{R}^N\setminus B_R}+(1+\delta)\chi_{B_R}(x):=\widebar{G}(x).$$
Then one infers from the comparison principle that
$$H_n(t_n+t,x_n+c_n(t_n+t)e_n+x)\ge \wideubar{H}_n(t,x)\ \mbox{for all}\ \ t\ge 0,\ x\in\mathbb{R}^N,$$
wherein $\wideubar{H}_n$ is the solution of the equation
\begin{equation}\label{equation of sub solution of H_n g<1}
\left\{
\begin{array}{rl}
&\partial_t\wideubar{H}_n=d\Delta\wideubar{H}_n+c_n\nabla\wideubar{H}_n\cdot e_n+b\wideubar{H}_n(1-\wideubar{H}_n-g\widebar{G}),\\
&\wideubar{H}_n(0,x)=H_n(t_n,x_n+c_nt_ne_n+x).
\end{array}
\right.
\end{equation}

Then, we consider a stationary sub-solution $\psi(x;\eta)$, for each $\eta>0$, 
$$\psi(t,x;\eta)=\eta e^{-c_nx\cdot e_n/2}\phi_R(x).$$
Since $c^0<2\sqrt{db(1-g)}$, one can check that there exist $\eta_0$ depending only on $c^0$ such that for any $\delta$ small enough, $0<\eta\le \eta_0$ and $R$ large enough, the function $\psi(x)$ is a stationary sub-solution of the equation (\ref{equation of sub solution of H_n g<1}).
Therefore, the solution of the equation (\ref{equation of sub solution of H_n g<1}) associated with initial data $\psi(x;\eta)$ is increasing in time, and converges to some positive stationary solution that denote by $p_{n,R,\delta}(x)$. 
Moreover, the stationary state
$p_{n,R,\delta}(x)$ does not depend on the choice of $\eta\in(0,\eta_0]$\ . 
Then, we can choose $\eta$ sufficiently small such that $\wideubar{H}^n(0,x)\ge\psi(0,x;\eta)$ for all $x\in\mathbb{R}^N$. Then, it follows from the comparison principle that for any
$R>0$ large enough and $\delta>0$ small enough and $n$ large enough
\begin{equation}\label{eq 8}
\liminf_{t\to\infty}H_n(t_n+t,x_n+c_n(t_n+t)e_n+x)\ge \liminf_{t\to\infty}\wideubar{H}_n(t,x)\ge p_{n,R,\delta}(x)\ \ \mbox{for all}\ \ x\in\mathbb{R}^N.
\end{equation}

To complete the proof of lemma, it remains to check that $p_{n,R,\delta}$ is far way from $0$ as $n$ and $R$ are large enough and $\delta$ is small.
Since $p_{n,R,\delta}$ is bounded from above by $1$, one can use standard elliptic estimates to get that, as $n\to +\infty$, $R\to+\infty$ and $\delta\to 0$, the function $p_{n,R,\delta}(x)$ converges locally uniformly to a stationary solution $p_{\infty}(x)$ of the equation
\begin{equation*}
d\Delta p_{\infty}+c_{\infty}\nabla p_{\infty}\cdot e_n+bp_{\infty}(1-g-p_{\infty})=0.
\end{equation*}
Moreover, since the map $t\to\psi(t,x;\eta_0)$ is nondecreasing, one has $p_{n,R,\delta}(0)\ge\psi(0,0;\eta_0)\ge\eta_0\psi_R(0)$. Note that $\varphi_R\to 1$ locally uniformly as $R\to+\infty$, hence $p_{\infty}(0)\ge\eta_0$ and $p_{\infty}(x)>0$ for all $x\in\mathbb{R}^N$. Therefore, from the statements (\ref{uniform converget to 0 of H c<c* in claim}) and (\ref{eq 8}), we reached a contradiction and  proved the Lemma \ref{low estimate of H 3}.
\end{proof}
\vspace{10pt}

By applying the similar argument to that for Lemma \ref{lemma of strong spreading of H c<c*} and Lemma \ref{lemma of uniform spreading of H and  F+C c<c* on x}, one can complete the proof of Lemma \ref{low estimate of H 2}. Then, by proceeding the proof of Theorem \ref{thm final zone 1+s>a}, one can immediately conclude a lemma as follows:
\begin{lemma}\label{lemma profile F+C>0 g<1}
If $g<1$, then for any given $0\le c<c^*$, there exists $\varepsilon>0$, such that, for the solution $(F,C,H)$ of the system \textnormal{(\ref{fch-equation})} with the initial data \textnormal{(\ref{intial data})}, it holds:
\begin{equation}\label{estimate profile F+C>0 g<1}
\liminf_{ t \to \infty}\inf_{\lVert x\rVert\le ct}(F+C)(t,x)\ge \varepsilon,
\end{equation}
provided that
$$g<\frac{\min\{1,a\}}{\min\{1,a\}+s}\ \ \mbox{or}\ \ bd\ge \frac{c^*}{1-g}.$$
\end{lemma}

With the uniform lower estimate of the $H$-component, we can first prove the $F$-component converges to $0$ in the final zone.
\begin{lemma}\label{lemma profile F g<1}
If $g<1$, then for any given $0\le c<c^*$, the solution $(F,C,H)$ of the system \textnormal{(\ref{fch-equation})} with the initial data \textnormal{(\ref{intial data})} satisfies:
\begin{equation}\label{estimate profile F g<1}
\lim_{ t \to \infty}\sup_{\lVert x\rVert\le ct}F(t,x)=0,
\end{equation}
provided that
$$g<\frac{\min\{1,a\}}{\min\{1,a\}+s}\ \ \mbox{or}\ \ bd\ge \frac{c^*}{1-g}.$$
\end{lemma}

\begin{proof}\it{of Lemma \ref{lemma profile F g<1}.}
\normalfont
Let us choose sequences  $\{c_n\}_{n\ge 0}\subset[0,c^0]$, $\{t_n\}_{n\ge 0}\subset\mathbb{R}_+$ with $t_n\to+\infty$ as $n\to+\infty$, and $\{x_n\}_{n\ge 0}\subset\mathbb{R}^N$ with $\lVert x_n\lVert\le c_nt_n$. 
We consider the limit functions again
$$\lim_{n\to+\infty}F(t_n+t,x_n+x)=F_{\infty}(t,x),$$
$$\lim_{n\to+\infty}C(t_n+t,x_n+x)=C_{\infty}(t,x),$$
$$\lim_{n\to+\infty}H(t_n+t,x_n+x)=H_{\infty}(t,x),$$
which converge locally uniformly to $(F_{\infty}, C_{\infty}, H_{\infty})$, an entire solution of the system 
\begin{equation*}
\left\{
\begin{aligned}
&\partial_tF_{\infty}=\Delta F_{\infty}+aF_{\infty}(1-C_{\infty}-F_{\infty}),\\
&\partial_tC_{\infty}=\Delta C_{\infty}+C_{\infty}(1-C_{\infty}-F_{\infty})+s(F_{\infty}+C_{\infty})H_{\infty},\\
&\partial_tH_{\infty}=d\Delta H_{\infty}+bH_{\infty}(1-gF_{\infty}-gC_{\infty}-H_{\infty}).
\end{aligned}
\right.
\end{equation*}
The result of Lemma \ref{low estimate of H 2} and Lemma \ref{lemma profile F+C>0 g<1}
imply that $(F_{\infty}+C_{\infty})\ge\varepsilon$ and  $H_{\infty}\ge\varepsilon$
for all $(t,x)\in\mathbb{R}\times\mathbb{R}^N$. Adding the both sides of the $F_{\infty}$-equation and $C_{\infty}$-equation, one may find that $G_{\infty}=F_{\infty}+C_{\infty}$ satisfies
$$\partial_tG_{\infty}-\Delta G_{\infty}\ge \min\{1,a\}G_{\infty}(1+s\varepsilon-G_{\infty})\ \ \mbox{for all}\ \ (t,x)\in\{(t,x)\ |\ G_{\infty}(t,x)\le 1+s\varepsilon\},$$
and $G_{\infty}(t,x)\ge \varepsilon$ when $t=-t_0$. Then, by passing $t_0\to+\infty$, one has $G_{\infty}(t,x)\ge 1+s\varepsilon$ for all $(t,x)\in\mathbb{R}_+\times\mathbb{R}^N$.
Then, it implies that 
$$\partial_tF_{\infty}-\Delta F_{\infty}\le -as\varepsilon F_{\infty},$$
and hence $F_{\infty}\equiv 0$.
Since the sequences $\{t_n\}_{n\ge 0}$ and $\{x_n\}_{n\ge 0}$  are chosen arbitrarily, the proof of this lemma is complete.
\end{proof}
\vspace{10pt}

Then, for the special case $d=1$, we can investigate the profiles of the $C$-component and $H$-component in the finial zone by considering the dynamics of the underlying ODE system:
\begin{equation*}
\left\{
\begin{aligned}
&C_t=C(1-C)+sCH,\\
&H_t=bH(1-H-gC).
\end{aligned}
\right.
\end{equation*}
We expect the solution of the PDE system (\ref{fch-equation}) to converge uniformly to the equilibrium $(C^*,H^*)$ as $t\to+\infty$.

Let us introduce the set $\Sigma=\{(C,H)\in \mathbb{R}^2\ :\ 0<C<1+s,\ 0<H<1 \}$, in which 
$(C^*,H^*)$ is the unique singular point.
There exists
a strictly convex function $\Phi:\Sigma\to\mathbb{R}$ of class $C^2$ that attains its minimum point at $(C^*,H^*)$ and satisfies
$$(C(1-C+sH),bH(1-H-gC))\cdot \nabla\Phi(C,H)\le 0\ \mbox{for all}\ \ (C,H)\in\Sigma.$$
The function $\Phi$ is a strict Lyapunov function in the sense that:
if $(C,H)$ denotes the solution of corresponding Ode system (\ref{ode system}) with 
the initial data $(C_0,H_0)$, then
$$\Phi(C(t),H(t))=\Phi(C_0,H_0)\ \mbox{for all}\ \ t>0\Rightarrow (C_0,H_0)=(C^*,H^*).$$
As a matter of fact, we can consider the strictly convex functional as
$$\Phi(C,H):=bg\int_{C^*}^C\frac{\eta-C^*}{\eta}+s\int_{H^*}^H\frac{\xi-H^*}{\xi}.$$
It is not difficult to check that, for all $(C,H)\in \Sigma$, it holds
$$(C(1-C+sH),bH(1-H-gH))\cdot \nabla\Phi(C,H)=-bg(C-C^*)^2-bs(H-H^*)^2\le 0.$$
Furthermore, for any solution $(C,H)$ of the ODE system, one has
\begin{align*}
\Phi(C,H)_t&=bg(C-C^*)(1-C+sH)+bs(H-H^*)(1-H-gH)\\
&=-bg(C-C^*)^2-bs(H-H^*)^2,
\end{align*}
hence it is a strict Lyapunov function.
Since $\Phi$ is bounded from below, we assume without loss of generality that $\Phi\ge 0$, and the equality only holds at the unique
minimizer $(C^*,H^*)$.

Let us argue by contradiction and assume that there exist $c\in [\ 0,c^*)$ and a sequence $\{(t_k,x_k)\}_{k\ge0}\subset(0,\infty)\times\mathbb{R}^N$ such that $t_k\to+\infty$ and $\delta>0$ such that for all $k>0$,
\begin{equation*}
\lVert x_k\lVert\le ct_k\ \ \mbox{and}\ \ |C(t_k,x_k)-C^*|+|H(t_k,x_k)-H^*|\ge\delta.
\end{equation*}
Consider the sequence of functions $(F_k,C_k,H_k)(t,x)=(F,C,H)(t+t_k,x+x_k)$.
Now, let us fix $c'>0$ such that $c<c'<c^*$. Therefore, there exist $N>0$ large enough and $\varepsilon>0$ small enough such that, for $k\ge 0$ and $t\in \mathbb{R}$, one has  
\begin{equation*}
t+t_k\ge A\ \ \mbox{and}\ \ \lVert x\lVert\le c't+(c'-c)t_k\Rightarrow 
\left\{
\begin{aligned}
&F_k(t,x)\le \frac{1}{k},\\
&\varepsilon\le C_k(t,x)\le 1+s-\varepsilon,\\
&\varepsilon\le H_k(t,x)\le 1-\varepsilon.
\end{aligned}
\right.
\end{equation*}
Then, by parabolic estimates, possibly along a subsequence, one may assume that
$$(F_k,C_k,H_k)(t,x)\to (F_{\infty},C_{\infty},H_{\infty})(t,x)\ \ \mbox{locally uniformly for}\ \ (t,x)\in\mathbb{R}\times\mathbb{R}^N,$$
where $(F_{\infty},C_{\infty},H_{\infty})$ is a bounded entire solution and satisfies
$$\underset{(t,x)\in\mathbb{R}\times\mathbb{R}^N}{\sup}F_{\infty}(t,x)=0,$$
$$\underset{(t,x)\in\mathbb{R}\times\mathbb{R}^N}{\inf}C_{\infty}(t,x)>0\ \ \mbox{and}\ \ \underset{(t,x)\in\mathbb{R}\times\mathbb{R}}{\inf}H_{\infty}(t,x)>0,$$
$$\underset{(t,x)\in\mathbb{R}\times\mathbb{R}^N}{\sup}C_{\infty}(t,x)<1+s\ \ \mbox{and}\ \ \underset{(t,x)\in\mathbb{R}\times\mathbb{R}}{\sup}H_{\infty}(t,x)<1.$$
Moreover, it holds
$$|C_{\infty}(0,0)-C^*|+|H_{\infty}(0,0)-H^*|>0.$$
In order to reach a contradiction, we claim that
\begin{claim}\label{claim 4}
Let $(U,V)$ be a bounded entire solution satisfying the above estimates, then $(U,V)(t,x)\equiv(C^*,H^*)$.
\end{claim}
\begin{proof}\it{of Claim \ref{claim 4}.}
\normalfont
To prove this claim, we consider
$$W(t,x):=\Phi(U(t,x),V(t,x)).$$
Then one has 
\begin{align*}
W_t-\Delta W=&-(\Phi_{UU}|\nabla U|^2+2\Phi_{UV}\nabla U\cdot\nabla V+\Phi_{VV}|\nabla V|^2)\\
&+\Phi_UU(1-U+sV)+\Phi_VbV(1-V-gU)\\
\le& 0.
\end{align*}
Choose sequences $\{t_n\}_{n\ge0}$ and $\{x_n\}_{n\ge0}$ such that
$$\lim_{n\to \infty}W(t_n,x_n)=\sup_{(t,x)\in \mathbb{R}\times\mathbb{R}^N}W(t,x).$$
By considering the sequence 
$$W_n(t,x)=W(t+t_n,x+x_n)=\Phi(U_n(t,x),V_n(t,x)),$$
where
$U_n(t,x)=U(t+t_n,x+x_n)$ and $V_n(t,x)=V(t+t_n,x+x_n)$,
then one obtains, possibly along a subsequence, $(U_n,V_n)\to(U_{\infty},V_{\infty})$ locally uniformly and $W_n\to W_{\infty}:=\Phi(U_{\infty},V_{\infty})$ locally uniformly where $(U_{\infty},V_{\infty})$ is an entire solution of the system (\ref{ode system}). Note that $W_{\infty}$ satisfies
$$W_{\infty}(0,0)=\sup_{(t,x)}W(t,x)=\sup_{(t,x)}W_{\infty}(t,x),$$
and $W_{\infty}$ is a sub-solution of the heat equation, hence $W_{\infty}(t,x)\equiv W_{\infty}(0,0)$ is a constant function. The strict convexity of $\Phi$ provide that
\begin{equation*}
U_{\infty}(t,x)\equiv U_{\infty}(t),\ V_{\infty}(t,x)\equiv V_{\infty}(t),
\end{equation*}
$$\Big(U_{\infty}(1-U_{\infty}+sV_{\infty}),bV_{\infty}(1-V_{\infty}-gU_{\infty})\Big)\cdot\nabla\Phi(U_{\infty},V_{\infty})\equiv 0.$$
Also by using the fact that $W_{\infty}$ is a constant, we have $\Phi(U_{\infty}(t),V_{\infty}(t))=\Phi(U_{\infty}(0),V_{\infty}(0))$ for all $t\in\mathbb{R}$. Since the Lyapunov function is strict, one obtains that $U_{\infty}(t)=C^*,\ V_{\infty}(t)=H^*$. Hence, one can conclude that
$$0\le W[U,V](t,x)\le \Phi(C^*,H^*)=0,$$
which completes the proof.
\end{proof}

\section*{Acknowledgments}
The authors would like to express sincere thanks to Prof. Matano for many helpful suggestions. 


\section*{Appendix: Proof of Proposition \ref{Dirichlet eq}}
We consider the equation
\[
\partial_t z=\partial_{\xi}^2 z + \Big(c^*-\frac{\delta}{t+t_0}+\frac{N-1}{\xi+\xi^{\delta}_{t_0}(t)}\Big)\partial_{\xi} z + \lambda^{*2}z,\ \ t>0,\ \xi>0,\quad\quad\eqref{symmetric linear eq}
\]
where 
$$\lambda^*=c^*/2\ \ \mbox{and}\ \ \xi^{\delta}_{t_0}(t):=c^*(t+t_0)-\delta\ln\frac{t+t_0}{t_0}.$$
We prove the following proposition.

\vskip 10pt

\noindent
{\bf Proposition \ref{Dirichlet eq} (\cite{AD})}\ {\it
Let $z^{\delta}_{t_0}(\xi,t)$ be the solution of the equation \eqref{symmetric linear eq} with boundary condition 
$$z^{\delta}_{t_0}(t,0)=0\ \ \mbox{for all}\ \ t>0$$ 
and the initial data
$$z^{\delta}_{t_0}(0,\xi)=e^{-\lambda^*\xi}\zeta_0(t_0^{-1/2}\xi)\geq0\ \ \mbox{for all}\ \ \xi\geq0,$$
where $\zeta_0(\cdot)$ is an nontrivial compactly supported smooth function.
Then it holds: 
\[
z^{\delta}_{t_0}(t,\xi)=\frac{(t+t_0)^{\gamma-\frac{1}{2}}}{t_0^{\gamma}}\xi e^{-\lambda^*\xi}\Big\{\frac{\int_0^{\infty}\zeta_0(\rho)\rho d\rho+h_1(t,t_0)}{\sqrt{\pi}}e^{-\frac{\xi^2}{4(t+t_0)}}+h_2(t,\xi,t_0)\Big\},\ \ \xi\geq0,\ t\geq0,\quad\eqref{estimate Dirichlet eq}
\]
where $\gamma:=\delta\lambda^*-\frac{N+1}{2}$, $h_1$ and $h_2$ are smooth functions satisfying
\[
\left\{
\begin{aligned}
&|h_1(t,t_0)|\leq B_1t_0^{-1/2}\lVert\zeta_0\rVert_m,\\
&|h_2(t,\xi,t_0)|\leq B_2\Big\{\frac{t_0^{1/4}\lVert\zeta_0\rVert_m}{(t+t_0)^{1/2}}+\Big(\frac{t_0}{t+t_0}\Big)^{1-\frac{B_2}{\sqrt{t_0}}}\lVert\partial_{\rho}^2\zeta_0\rVert_m\Big\}e^{-\frac{\xi^2}{8(t+t_0)}},
\end{aligned}
\right.\ \ \ \ \xi\ge0,\ t\ge0,
\]
for some positive constants $B_1$ and $B_2$. Here, the norm $\lVert\cdot\rVert_m$ is defined as
\[
\lVert\zeta_0\rVert_m^2:=\int_0^{\infty}\zeta_0(\rho)^2e^{\frac{\rho^2}{4}}d\rho.
\]  
}

\begin{proof}:
Let $z(t,\xi)=z^{\delta}_{t_0}(t,\xi)$ be be the solution of the equation \eqref{symmetric linear eq} with boundary condition 
$$z^{\delta}_{t_0}(t,0)=0\ \ \mbox{for all}\ \ t>0$$ 
and the initial data
$$z^{\delta}_{t_0}(0,\xi)=e^{-\lambda^*\xi}\zeta_0(t_0^{-1/2}\xi)\geq0\ \ \mbox{for all}\ \ \xi\geq0,$$
where $\zeta_0(\cdot)$ is an nontrivial compactly supported smooth function. 
By using new coordinates $\rho=\frac{\xi}{\sqrt{t+t_0}}$, $\tau=\log\frac{t+t_0}{t_0}$ and a new unkown function $e^{\lambda^* \xi}z(t,\xi)=\widetilde{\zeta}(\tau,\rho)$, the equation \eqref{symmetric linear eq} could be rewritten as follows:
\begin{equation}\label{eq:spherical eq 2}
\left\{
\begin{aligned}
&\partial_{\tau}\widetilde{\zeta}=\partial_{\rho}^2\widetilde{\zeta}+\frac{\rho}{2}\partial_{\rho}\widetilde{\zeta}+\widetilde{\zeta}+\frac{e^{-\frac{\tau}{2}}}{\sqrt{t_0}}\{\alpha_{1,t_0}\partial_{\rho}\widetilde{\zeta}+(\rho+1)\alpha_{2,t_0}\widetilde{\zeta}\}+\gamma\widetilde{\zeta},\ \ \rho>0,\ \tau>0,\\
&\widetilde{\zeta}(\tau,0)=0,\ \ \tau>0,\\
&\widetilde{\zeta}(0,\xi)=\zeta_0(\xi),\ \ \xi\ge0,
\end{aligned}
\right.
\end{equation}
where 
$$\gamma:=\delta\lambda^*-\frac{N-1}{2}-1,$$
\[
\begin{split}
\alpha_{1,t_0}(\tau,\rho):=&\frac{N-1}{2\lambda^*+\frac{e^{-\frac{\tau}{2}}}{\sqrt{t_0}}\rho-\frac{\delta\tau}{t_0}e^{-\tau}}-\delta,\\
\alpha_{2,t_0}(\tau,\rho):=&\frac{N-1}{\rho+1}\frac{\rho-\frac{\delta\tau}{\sqrt{t_0}}e^{-\frac{\tau}{2}}}{2\lambda^*+\frac{e^{-\frac{\tau}{2}}}{\sqrt{t_0}}\rho-\frac{\delta\tau}{t_0}e^{-\tau}}.
\end{split}
\]
Define $\zeta(\tau,\rho):=e^{-\gamma\tau}\widetilde{\zeta}(\tau,\rho)$ $(\rho\geq0,\ \tau\geq0)$, $\mathcal{L}\varphi:=\frac{{\rm d}^2}{{\rm d}\rho^2}\varphi+\frac{\rho}{2}\frac{\rm d}{{\rm d}\rho}\varphi+\varphi$, 
\[
\begin{split}
&m(\rho):=e^{\frac{\rho^2}{4}},\ \ D(\mathcal{L}):=\{\phi\in L^2_m(0,\infty)\mid\phi',\phi''\in L^2_m(0,\infty),\ \phi(0)=0\},\\ 
&L^2_m(0,\infty):=\Big\{\phi\in L^2(0,\infty)\,\Big|\,\|\phi\|_m^2:=\langle\phi,\phi\rangle_m:=\int_0^{\infty}\phi(\rho)^2m(\rho)\,{\rm d}\rho<\infty\Big\}.
\end{split}
\]
Then, it holds 
\begin{equation}\label{eq:parabolic}
\frac{\rm d}{{\rm d}\tau}\zeta=\mathcal{L}\zeta+\frac{e^{-\frac{\tau}{2}}}{\sqrt{t_0}}\big\{\alpha_{1,t_0}\partial_{\rho}\zeta+(\rho+1)\alpha_{2,t_0}\zeta\big\},\ \ \tau>0.
\end{equation}
Here we remark that $\mathcal{L}$ is a self-adjoint operator whose resolvent is compact and eigenvalues and corresponding eigenfunctions are as follows:
\[
\lambda_k:=-(k-1),\ \ \varphi_k:=\Big\lVert\frac{{\rm d}^{2k-1}}{{\rm d}\rho^{2k-1}}\frac{1}{m}\Big\rVert_m^{-1}\frac{{\rm d}^{2k-1}}{{\rm d}\rho^{2k-1}}\frac{1}{m}\ \ \Big(\varphi_1(\rho)=\pi^{-\frac{1}{4}}\rho e^{-\frac{\rho^2}{4}}\Big).
\]
We also remark some useful inequalities:
\begin{eqnarray}
\label{eq:1}
&&\lVert\zeta\rVert_m^2\leq\lVert\partial_{\rho}\zeta\rVert_m^2,\\
\label{eq:2}
&&\lVert(\rho+1)\zeta\rVert_m\leq3\sqrt{2}\lVert\partial_{\rho}\zeta\rVert_m+\sqrt{6}\lVert\zeta\rVert_m,\\
\label{eq:3}
&&2\lVert Q\zeta\rVert_m^2\leq\lVert\partial_{\rho}Q\zeta\rVert_m^2,\\
\label{eq:4}
&&\frac{3}{4}\lVert\widetilde{\mathcal{L}}\zeta\rVert_m^2\leq\lVert\partial_{\rho}^2\zeta\rVert_m^2\leq\frac{5}{4}\lVert\widetilde{\mathcal{L}}\zeta\rVert_m^2,
\end{eqnarray}
where $\widetilde{\mathcal{L}}\varphi:=\mathcal{L}\varphi-\varphi=\frac{{\rm d}^2}{{\rm d}\rho^2}\varphi+\frac{\rho}{2}\frac{\rm d}{{\rm d}\rho}\varphi$ and  $Q\zeta:=\zeta-\langle\zeta,\varphi_1\rangle_m\varphi_1$ which is the component of $\zeta$ which orthogonal to $\varphi_1$. 

Then there exist $c>0$, $T_0>0$ depending only on 
\[
\underset{t_0\geq\frac{\delta}{2\lambda^*}}{\sup}\lVert\alpha_{i,t_0}\rVert_{L^{\infty}([0,\infty)^2)}<\infty,\ \ \underset{t_0\geq\frac{\delta}{2\lambda^*}}{\sup}\lVert\partial_{\rho}\alpha_{i,t_0}\rVert_{L^{\infty}([0,\infty)^2)}<\infty\ \ (i=1,2)
\] 
such that the following holds for any $t_0\geq T_0$:
\begin{eqnarray}
\label{eq:L^2 bdd}
&&\lVert\zeta(\tau)\rVert_{m}\leq c\lVert\zeta_0\rVert_{m},\\
\label{eq:principal term}
&&|\langle\zeta(\tau),\varphi_1\rangle_m-\langle\zeta_0,\varphi_1\rangle_m|\leq\frac{c}{\sqrt{t_0}}\lVert\zeta_0\rVert_m,\\
\label{eq:LO term}
&&\lVert Q\zeta(\tau)\rVert_m^2\leq\frac{ce^{-\tau}}{\sqrt{t_0}}\lVert\zeta_0\rVert_m^2+e^{-(2-\frac{c}{\sqrt{t_0}})\tau}\lVert Q\zeta_0\rVert_m^2,\\
\label{eq:LO term 2}
&&\lVert\widetilde{\mathcal{L}}Q\zeta(\tau)\rVert_m^2\leq\frac{ce^{-\tau}}{\sqrt{t_0}}\lVert\zeta_0\rVert_m^2+e^{-(2-\frac{c}{\sqrt{t_0}})\tau}\lVert\widetilde{\mathcal{L}}Q\zeta_0\rVert_m^2.
\end{eqnarray}
As a matter of fact, 
\[
\frac{\rm d}{{\rm d}\tau}\lVert\zeta\rVert_m^2=2\langle\zeta,\dot{\zeta}\rangle_m=I_1+I_2,
\]
where we denote $\dot{\zeta}:=\frac{\rm d}{{\rm d}\tau}\zeta$ and
\[
I_1:=2\langle\zeta,\mathcal{L}\zeta\rangle_m,\ \ I_2:=\frac{2e^{-\frac{\tau}{2}}}{\sqrt{t_0}}\langle\zeta,\alpha_{1,t_0}\partial_{\rho}\zeta+(\rho+1)\alpha_{2,t_0}\zeta\rangle_m.
\]
Integrating by part, one has
\[
I_1=-2(\lVert\partial_{\rho}\zeta\rVert_m-\lVert\zeta\rVert_m).
\]
By applying Schwarz inequality and \eqref{eq:2}, 
\[
\begin{split}
|\langle\zeta,\alpha_{1,t_0}\partial_{\rho}\zeta&+(\rho+1)\alpha_{2,t_0}\zeta\rangle_m|\leq\lVert\zeta\rVert_m(\lVert\alpha_{1,t_0}\rVert_{L^{\infty}}\lVert\partial_{\rho}\zeta\rVert_m
+\lVert\alpha_{2,t_0}\rVert_{L^{\infty}}\lVert(\rho+1)\zeta\rVert_m)\\
&\leq(\lVert\alpha_{1,t_0}\rVert_{L^{\infty}}
+3\sqrt{2}\lVert\alpha_{2,t_0}\rVert_{L^{\infty}})\lVert\zeta\rVert_m\lVert\partial_{\rho}\zeta\rVert_m
+\sqrt{6}\lVert\alpha_{2,t_0}\rVert_{L^{\infty}}\lVert\zeta\rVert_m^2.
\end{split}
\]
Hence, if we denote $C:=\lVert\alpha_{1,t_0}\rVert_{L^{\infty}}+3\sqrt{2}\lVert\alpha_{2,t_0}\rVert_{L^{\infty}}$, then it holds
\[
I_2\leq \frac{Ce^{-\frac{\tau}{2}}}{\sqrt{t_0}}(\lVert\partial_{\rho}\zeta\rVert_m^2-\lVert\zeta\rVert_m^2)+\frac{3Ce^{-\frac{\tau}{2}}}{\sqrt{t_0}}\lVert\zeta\rVert_m^2.
\]
Thus by \eqref{eq:1}, for $t_0\geq C^2/4$, one has
\[
\frac{\rm d}{{\rm d}\tau}\lVert\zeta\rVert_m^2\leq-(2-\frac{Ce^{-\frac{\tau}{2}}}{\sqrt{t_0}})(\lVert\partial_{\rho}\zeta\rVert_m^2-\lVert\zeta\rVert_m^2)+\frac{3Ce^{-\frac{\tau}{2}}}{\sqrt{t_0}}\lVert\zeta\rVert_m^2\leq\frac{3Ce^{-\frac{\tau}{2}}}{\sqrt{t_0}}\lVert\zeta\rVert_m^2.
\]
Then, one can obtain that
\[
\lVert\zeta\rVert_m^2\leq e^{\frac{6C}{\sqrt{t_0}}(1-e^{-\tau/2})}\lVert\zeta_0\rVert_m^2\leq e^{\frac{6C}{\sqrt{t_0}}}\lVert\zeta_0\rVert_m^2.
\]
Therefore the inequality \eqref{eq:L^2 bdd} holds true. By \eqref{eq:L^2 bdd}, $\mathcal{L}\varphi_1=0$, $\lVert\varphi_1\rVert_m=\lVert\partial_{\rho}\varphi_1\rVert_m=1$ and integrating by part, one has
\[
\begin{split}
\frac{\rm d}{{\rm d}\tau}\langle\zeta,\varphi_1\rangle_m&=\frac{2e^{-\frac{\tau}{2}}}{\sqrt{t_0}}\langle\alpha_{1,t_0}\partial_{\rho}\zeta+(1+\rho)\alpha_{2,t_0}\zeta,\varphi_1\rangle_m\\
&=-\frac{2e^{-\frac{\tau}{2}}}{\sqrt{t_0}}\langle\zeta,\alpha_{1,t_0}\partial_{\rho}\varphi_1+\partial_{\rho}\alpha_{1,t_0}\varphi_1+\alpha_{1,t_0}\frac{\rho}{2}\varphi_1-(\rho+1)\alpha_{2,t_0}\varphi_1\rangle_m\\
&\hspace{-20pt}\leq\frac{2e^{\frac{3C}{\sqrt{t_0}}}e^{-\frac{\tau}{2}}}{\sqrt{t_0}}\lVert\alpha_{1,t_0}\partial_{\rho}\varphi_1
+\partial_{\rho}\alpha_{1,t_0}\varphi_1+\alpha_{1,t_0}\frac{\rho}{2}\varphi_1-(\rho+1)\alpha_{2,t_0}\varphi_1\rVert_m\lVert\zeta_0\rVert_m.
\end{split}
\] 
Thus, if we denote $C_1:=4e^{\frac{3C}{\sqrt{t_0}}}\lVert\alpha_{1,t_0}\partial_{\rho}\varphi_1+\partial_{\rho}\alpha_{1,t_0}\varphi_1+\alpha_{1,t_0}\frac{\rho}{2}\varphi_1-(\rho+1)\alpha_{2,t_0}\varphi_1\rVert_m$, it holds 
\[
|\langle\zeta,\varphi_1\rangle_m-\langle\zeta_0,\varphi_1\rangle_m|\leq\frac{C_1}{\sqrt{t_0}}(1-e^{-\frac{\tau}{2}})\lVert\zeta_0\rVert_m.
\]
Therefore the inequality \eqref{eq:principal term} holds. Inferring from the facts that $\langle Q\psi,Q\varphi\rangle_m=\langle Q\psi,\varphi\rangle_m$, $Q\widetilde{\mathcal{L}}=\widetilde{\mathcal{L}}Q$, one has
\[
\frac{\rm d}{{\rm d}\tau}\lVert\widetilde{\mathcal{L}}Q\zeta\rVert_m^2=2\langle\widetilde{\mathcal{L}}Q\zeta,\widetilde{\mathcal{L}}Q\dot{\zeta}\rangle_m=\langle\widetilde{\mathcal{L}}^2Q\zeta,\dot{\zeta}\rangle_m=I_1+I_2+I_3,
\]
where 
\[
\begin{split}
I_1:=&2\langle\widetilde{\mathcal{L}}^2Q\zeta,\mathcal{L}\zeta\rangle_m=2\langle\widetilde{\mathcal{L}}^2Q\zeta,\widetilde{\mathcal{L}}Q\zeta+Q\zeta\rangle_m=-2(\lVert\partial_{\rho}\widetilde{\mathcal{L}}Q\zeta\rVert_m^2
-\lVert\widetilde{\mathcal{L}}Q\zeta\rVert_m^2),\\ 
I_2:=&\frac{2e^{-\frac{\tau}{2}}}{\sqrt{t_0}}\langle\widetilde{\mathcal{L}}^2Q\zeta, \alpha_{1,t_0}\partial_{\rho}Q\zeta+(\rho+1)\alpha_{2,t_0}Q\zeta\rangle_m,\\
I_3:=&\frac{2e^{-\frac{\tau}{2}}}{\sqrt{t_0}}\langle\widetilde{\mathcal{L}}^2Q\zeta, \alpha_{1,t_0}\partial_{\rho}P\zeta+(\rho+1)\alpha_{2,t_0}P\zeta\rangle_m\ \ (P\varphi:=\langle\varphi,\varphi_1\rangle_m\varphi_1=\varphi-Q\varphi).
\end{split}
\]
By \eqref{eq:2}, \eqref{eq:3}, \eqref{eq:4} and $\lVert\partial_{\rho}\zeta\rVert_m^2=-\langle\widetilde{\mathcal{L}}\zeta,\zeta\rangle_m\leq\lVert\widetilde{\mathcal{L}}\zeta\rVert_m\lVert\zeta\rVert_m$, one may find
\[
\begin{split}
&\langle\widetilde{\mathcal{L}}^2Q\zeta, \alpha_{1,t_0}\partial_{\rho}Q\zeta+(\rho+1)\alpha_{2,t_0}Q\zeta\rangle_m=
-\langle\partial_{\rho}\widetilde{\mathcal{L}}Q\zeta, \partial_{\rho}\alpha_{1,T}\partial_{\rho}Q\zeta+\alpha_{1,t_0}\partial_{\rho}^2Q\zeta\\
&\hspace{120pt}+(\rho+1)\alpha_{2,t_0}\partial_{\rho}Q\zeta+(\rho+1)\partial_{\rho}\alpha_{2,t_0}Q\zeta
+\alpha_{2,t_0}Q\zeta\rangle_m\\
&\leq C_2\lVert\partial_{\rho}\widetilde{\mathcal{L}}Q\zeta\rVert_m(
\lVert\partial_{\rho}\widetilde{\mathcal{L}}Q\zeta\rVert_m^{1/2}\lVert Q\zeta\rVert_m^{1/2}+\lVert\partial_{\rho}\widetilde{\mathcal{L}}Q\zeta\rVert_m+\lVert Q\zeta\rVert_m).
\end{split}
\]
Thus, one conclude that 
\[
I_2\leq\frac{C_2}{2\sqrt{t_0}}(3e^{-\frac{\tau}{3}}+4e^{-\frac{\tau}{2}}+2)\lVert\partial_{\rho}\widetilde{\mathcal{L}}Q\zeta\rVert_m^2+\frac{3C_2}{2\sqrt{t_0}}e^{-\tau}\lVert Q\zeta\rVert_m^2.
\]
By similar computation, 
\[
I_3\leq\frac{C_2'}{\sqrt{t_0}}\lVert\partial_{\rho}\widetilde{\mathcal{L}}Q\zeta\rVert_m^2+\frac{C_2' e^{-\tau}}{\sqrt{t_0}}\lVert\zeta\rVert_m^2,
\]
where 
\[
C_2':=\lVert \partial_{\rho}\alpha_{1,t_0}\partial_{\rho}\varphi_1+\alpha_{1,t_0}\partial_{\rho}^2\varphi_1
+(\rho+1)\alpha_{2,t_0}\partial_{\rho}\varphi_1+(\rho+1)\partial_{\rho}\alpha_{2,t_0}\varphi_1
+\alpha_{2,t_0}\varphi_1\rVert_m.
\]
Thus by applying the fact $2\lVert\widetilde{\mathcal{L}}Q\zeta\rVert_m^2\leq\lVert\partial_{\rho}\widetilde{\mathcal{L}}Q\zeta\rVert_m^2$ and \eqref{eq:L^2 bdd}, one has
\[
\frac{\rm d}{{\rm d}\tau}\lVert\widetilde{\mathcal{L}}Q\zeta\rVert_m^2\leq-\Big(2-\frac{C_2(3e^{-\frac{\tau}{3}}+4e^{-\frac{\tau}{4}}+2)+2C_2'}{\sqrt{t_0}}\Big)\lVert\widetilde{\mathcal{L}}Q\zeta\rVert_m^2+\frac{(3C_2+2C_2')e^{\frac{6C}{\sqrt{t_0}}}}{\sqrt{t_0}}e^{-\tau}\lVert\zeta_0\rVert_m^2.
\]
Therefore the inequality \eqref{eq:LO term 2} holds. Since the proof of \eqref{eq:LO term} is similar to that of \eqref{eq:LO term 2}, we omit it. 

Next, one can obtain that
\[
\begin{split}
|(Q\zeta(\tau,\rho))^2m(\rho)|&\leq\int_0^{\infty}|\partial_{\rho}((Q\zeta)^2m)|\,{\rm d}\rho\\&\leq2\int_0^{\infty}|Q\zeta\partial_{\rho}Q\zeta|m\,{\rm d}\rho+\int_0^{\infty}(Q\zeta)^2\frac{\rho}{2}m\,{\rm d}\rho\\&=4\int_0^{\infty}|Q\zeta\partial_{\rho}Q\zeta|m\,{\rm d}\rho\leq4\lVert Q\zeta\rVert_m\lVert\partial_{\rho}Q\zeta\rVert_m\leq4\lVert Q\zeta\rVert_m^{\frac{3}{2}}\lVert\widetilde{\mathcal{L}}Q\zeta\rVert_m^{\frac{1}{2}}\\&\leq\lVert\widetilde{\mathcal{L}}Q\zeta\rVert_m^2+3\lVert Q\zeta\rVert_m^2
\end{split}
\]
Similarly, it also holds
\[
|(\partial_{\rho}Q\zeta(\tau,\rho))^2m(\rho)|\leq4\lVert\partial_{\rho}Q\zeta\rVert_m\lVert\partial_{\rho}^2Q\zeta\rVert_m\leq2\sqrt{5}\lVert Q\zeta\rVert_m^{\frac{1}{2}}\lVert\widetilde{\mathcal{L}}Q\zeta\rVert_m^{\frac{3}{2}}\leq\frac{\sqrt{5}}{2}\lVert Q\zeta\rVert_m^2+\frac{3\sqrt{5}}{2}\lVert\widetilde{\mathcal{L}}Q\zeta\rVert_m^2.
\]
By \eqref{eq:LO term} and \eqref{eq:LO term 2}, one has
\[
\begin{split}
\lVert\widetilde{\mathcal{L}}Q\zeta\rVert_m^2+3\lVert Q\zeta\rVert_m^2&\leq\frac{4ce^{-\tau}}{\sqrt{t_0}}\lVert\zeta_0\rVert_m^2+e^{-(2-\frac{c}{\sqrt{t_0}})\tau}(3\lVert Q\zeta_0\rVert_m^2+\lVert\widetilde{\mathcal{L}}Q\zeta_0\rVert_m^2),\\
\frac{\sqrt{5}}{2}\lVert Q\zeta\rVert_m^2+\frac{3\sqrt{5}}{2}\lVert\widetilde{\mathcal{L}}Q\zeta\rVert_m^2&\leq\frac{2\sqrt{5}ce^{-\tau}}{\sqrt{t_0}}\lVert\zeta_0\rVert_m^2+\frac{\sqrt{5}}{2}e^{-(2-\frac{c}{\sqrt{t_0}})\tau}(\lVert Q\zeta_0\rVert_m^2+3\lVert\widetilde{\mathcal{L}}Q\zeta\rVert_m^2).
\end{split}
\]
Hence by $\lVert Q\zeta_0\rVert_m\leq\lVert\zeta_0\rVert_m$ and $\lVert\widetilde{\mathcal{L}}Q\zeta_0\rVert_m\leq\lVert\widetilde{\mathcal{L}}\zeta_0\rVert_m$, one obtains
\[
\begin{split}
|Q\zeta(\tau,\rho)|&\leq\sqrt{\lVert\widetilde{\mathcal{L}}Q\zeta\rVert_m^2+3\lVert Q\zeta\rVert_m^2}m(\rho)^{-\frac{1}{2}}\\&\leq\Big(\frac{\sqrt{4c}e^{-\frac{\tau}{2}}}{t_0^{1/4}}\lVert\zeta_0\rVert_m+\sqrt{3}e^{-(1-\frac{c}{2\sqrt{t_0}})}(\lVert\zeta_0\rVert_m+\lVert\widetilde{\mathcal{L}}\zeta_0\rVert_m)\Big)e^{-\frac{\rho^2}{8}},\\
|\partial_{\rho}Q\zeta(\tau,\rho)|&\leq\Big(\frac{\sqrt{2\sqrt{5}c}e^{-\frac{\tau}{2}}}{t_0^{1/4}}\lVert\zeta_0\rVert_m+\sqrt{\frac{3\sqrt{5}}{2}}e^{-(1-\frac{c}{2\sqrt{t_0}})\tau}(\lVert\zeta_0\rVert_m+\lVert\widetilde{\mathcal{L}}\zeta_0\rVert_m)\Big)e^{-\frac{\rho^2}{8}}.
\end{split}
\]
Thus if we denote $\widetilde{h}_1(\tau):=(4\pi)^{\frac{1}{4}}\{\langle\zeta(\tau),\varphi_1\rangle_m-\langle\zeta_0,\varphi_1\rangle_m\}$ and $\widetilde{h}_2(\tau,\rho):=\frac{Q\zeta(\tau,\rho)}{\rho}$, by \eqref{eq:principal term} and \eqref{eq:4}, it holds
\[
\begin{split}
|\widetilde{h}_1(\tau)|&\leq\frac{(4\pi)^{\frac{1}{4}}c}{\sqrt{t_0}}\lVert\zeta_0\rVert_m,\\
|\widetilde{h}_2(\tau,\rho)|&=\Big|\frac{Q\zeta(\tau,\rho)}{\rho}\Big|\leq2\Big(\frac{\sqrt{c}e^{-\frac{\tau}{2}}}{t_0^{1/4}}\lVert\zeta_0\rVert_m+e^{-(1-\frac{c}{2\sqrt{t_0}})\tau}(\lVert\zeta_0\rVert_m+\lVert\widetilde{\mathcal{L}}\zeta_0\rVert_m)\Big)e^{-\frac{\rho^2}{8}}\\&\leq2\Big(\frac{\sqrt{c}e^{-\frac{\tau}{2}}}{t_0^{1/4}}\lVert\zeta_0\rVert_m+e^{-(1-\frac{c}{2\sqrt{t_0}})\tau}(\lVert\zeta_0\rVert_m+\frac{2}{\sqrt{3}}\lVert\partial_{\rho}^2\zeta_0\rVert_m)\Big)e^{-\frac{\rho^2}{8}}.
\end{split}
\]
On the other hand, one has
\[
\begin{split}
\zeta(\tau,\rho)&=P\zeta(\tau,\rho)+Q\zeta(\tau,\rho)=\langle\zeta(\tau),\varphi_1\rangle_m\varphi_1(\rho)+Q\zeta(\tau,\rho)\\&=\rho\Big\{(\langle\zeta_0,\varphi_1\rangle_m+(4\pi)^{-\frac{1}{4}}\widetilde{h}_1(\tau))\frac{\varphi(\rho)}{\rho}+\widetilde{h}_2(\tau,\rho)\Big\}.
\end{split}
\]
Since 
\[
z(t,\xi)=\frac{(t+t_0)^{\gamma}}{t_0^{\gamma}}e^{-\lambda^*\xi}\zeta\Big(\log\frac{t+t_0}{t_0},\frac{\xi}{\sqrt{t+t_0}}\Big), 
\]
if we denote $B_1=(4\pi)^{1/4}$, $B_2=\max\{4\sqrt{\frac{c}{3}},\frac{c}{2}\}$ and
\[
h_1(t,t_0):=\widetilde{h}_1\Big(\log\frac{t+t_0}{t_0}\Big),\ \ h_2(t,\xi,t_0):=\widetilde{h}_2\Big(\log\frac{t+t_0}{t_0},\frac{\xi}{\sqrt{t+t_0}}\Big)
\] 
then we obtain the conclusion \eqref{estimate Dirichlet eq} of Proposition \ref{Dirichlet eq} and complete the proof.
\end{proof}




\end{document}